\documentclass[11pt]{article}
\usepackage[english]{babel}
\usepackage{amsfonts}
\usepackage[utf8]{inputenc}
\usepackage{amsthm}
\usepackage{ mathrsfs }
\usepackage{amsmath }
\usepackage{mathrsfs}
\usepackage{enumerate}
\usepackage{mathtools}
\usepackage{bbm}
\usepackage{makeidx}
\usepackage{graphicx}
\usepackage[noadjust]{cite}
\usepackage{frontespizio}
\usepackage{color}
\usepackage[makeroom]{cancel}
\usepackage[normalem]{ulem}
\usepackage{ amssymb }
\usepackage[a4paper, total={7in, 9in}]{geometry}
\usepackage{hyperref}
\usepackage{mathtools}
\usepackage{enumitem}
\usepackage[titletoc]{appendix}
\usepackage{ esint }
\mathtoolsset{showonlyrefs=true}

\newcommand{\F}{\mathcal{F}}

\newcommand{\diam}{\text{diam}}
\newcommand{\bu}{{\bf u}}
\newtheorem{theorem}{Theorem}[section]
\newtheorem{lemma}[theorem]{Lemma}
\newtheorem{prop}[theorem]{Proposition}
\newcommand{\prob}[1]{\mathscr P(#1)} 
\newcommand{\restr}[1]{\lower3pt\hbox{$|_{#1}$}}
\newcommand{\la}{\langle}
\newcommand{\ra}{\rangle}
\newcommand{\sfd}{{\sf d}}
\newcommand{\ppi}{{\mbox{\boldmath$\pi$}}}
\newcommand{\nn}{\mathbb{N}}
\newcommand{\Lip}{{\rm Lip}}

\theoremstyle{definition}
\newtheorem{definition}[theorem]{Definition}
\renewcommand{\phi}{\varphi}
\renewcommand{\div}{{\rm div}}

\newcommand{\e}{{\rm{e}}}    
\newcommand{\lip}{{\rm lip}}

\newtheorem{cor}[theorem]{Corollary}
\renewcommand{\H}[1]{{\rm Hess}(#1)}

\newcommand{\mea}{\mathfrak{m}}
\newcommand{\Per}{{\rm Per}}
\newcommand{\LIP}{\mathsf{LIP}}
\newcommand{\testloc}{{\sf{Test}}_\loc}

\newcommand{\test}{{\sf{Test}}}
\newcommand{\bd}{{\mathbf\Delta}}

\renewcommand{\Cap}{\mathsf{Cap}}

\renewcommand{\d}{{\mathrm d}}
\newcommand{\supp}{\mathop{\rm supp}\nolimits} 
\renewcommand{\c}{{\mathrm c}}
\newcommand{\loc}{\mathsf{loc}}
\newcommand{\nchi}{{\raise.3ex\hbox{$\chi$}}}
\newcommand{\W}{\mathit{W}^{1,2}}

\DeclareMathOperator\rr{\mathbb{R}}
\DeclareMathOperator\eps{\varepsilon}

\DeclareMathOperator*{\esssup}{ess\,sup}
\DeclareMathOperator*{\essinf}{ess\,inf}
	\makeindex
	\numberwithin{equation}{section}
\newcommand{\ch}{{\rm Ch}}
\newcommand{\dis}{{\sf D}}
\newcommand{\X}{{\rm X}}

\newcommand{\fr}{\penalty-20\null\hfill$\blacksquare$}     

\theoremstyle{remark}
\newtheorem{remark}[theorem]{Remark}

\title{Monotonicity formulas for harmonic functions in ${\rm RCD}(0,N)$ spaces}

\begin{document}

\author{Nicola Gigli\ \thanks{SISSA, ngigli@sissa.it}  \and
	Ivan Yuri Violo
	\thanks{SISSA, iviolo@sissa.it}}

\maketitle	
	
\begin{abstract}
	We generalize to the ${\rm RCD}(0,N)$ setting a family of monotonicity formulas by Colding and Minicozzi  for positive harmonic functions in Riemannian manifolds with non-negative Ricci curvature. Rigidity and almost rigidity statements are also proven, the second appearing to be new even in the smooth setting. 
	
	Motivated by the recent work in \cite{AFM} we also introduce the notion of electrostatic potential in ${\rm RCD}$ spaces, which also satisfies our monotonicity formulas.
	
	Our arguments are mainly based  on new estimates for harmonic functions in ${\rm RCD}(K,N)$ spaces and on a new functional version of  the `(almost) outer volume cone implies (almost) outer metric cone' theorem.
\end{abstract}

\tableofcontents

\section{Introduction and main results}

Monotonicity formulas are in general an important tool in analysis and geometry and have  proven to be crucial in the study of functional inequalities, regularity of  PDEs and minimal surfaces, one example being the celebrated Almgren frequency function \cite{A}. We refer also to \cite{CMsurvey} for an overview on this topic and further references.   Recently these formulas also found application in the theory of metric geometry with particular interest in the regularity of singular spaces. 
It is however important to recall that, especially on this context, monotonicity formulas are often coupled with  rigidity and almost rigidity properties. In other words, besides knowing that a quantity is monotone it is often important to know that, if it happens to be constant or almost constant, then some extra regularity or almost regularity of the objects involved is present. An example of this, in the context of singular metric spaces,  is the Bishop-Gromov volume ratio, where this said (almost) regularity is realized as (almost) local conical structure of the space. This feature was exploited first in \cite{CC2} to non-collapsed Ricci-limit spaces and subsequently in \cite{DGnonc} to ${\sf ncRCD}$ spaces in order  show that blow-up limits are cones and was recently developed in \cite{volumebounds}  to deduce volume bounds for the singular set, extending the analogous result for Ricci-limit spaces (\cite{ChNa}). These works fall in the more general theory of quantitative differentiation  which, roughly said, allows to pass from a monotone quantity with almost rigidity properties (also called coerciveness) to non trivial results about effective regularity of  the space (or a particular function). We refer to \cite{quantdiff} for a detailed overview on this topic.

All of this motivates the study of new monotonicity formulas in ${\rm RCD}$ spaces, where the Bishop-Gromov volume ratio  and the Perelman $\cal W$ functional (see \cite{KL}) were up to now essentially the unique examples. 
	
Our  interest will be in extending a class of monotonicity formulas for harmonic functions in Riemannian manifolds with nonnegative Ricci curvature, to the ${\rm RCD}$ setting. 
These type of formulas were first introduced by a series of works by Colding and Minicozzi (\cite{CMmonotone}, \cite{CMmonotone2}, \cite{CMsurvey}) and were used by the same authors to prove the uniqueness of tangent cones for Einstein manifolds (\cite{CMuniq}).  Recently the same monotone quantities were reinterpreted in \cite{AFM} in the study of the electrostatic potential of a set and  lead to the proof of new Willmore-type geometrical inequalities for Riemannian manifolds (see also \cite{AMeucl} for the Euclidean setting).

\bigskip

\textbf{The monotonicity formula}

Our first main result is the extension to the non-smooth setting of the whole class of monotonicity formulas derived in  \cite{AFM}.
More precisely we show that in a nonparabolic ${\rm RCD}(0,N)$ space  (see Def. \ref{def:nonparabolic}), if $u\in L^\infty(\Omega)\cap C(\Omega)$ solves
\begin{equation}\label{mainpde}\tag{P}
	\begin{cases}
		\bd|_{\Omega} u=0,\\
		\liminf_{y\to x} u(y)\ge 1, & \text{ for every }x \in \partial \Omega,\\
		u(x) \to 0 & \text{as }\sfd(x,\partial\Omega )\to +\infty,\\
	\end{cases}
\end{equation}
where $\Omega$ is open, unbounded and with $\partial \Omega$ bounded, then the function 
\[
(0,1)\ni t \mapsto U_{\beta}(t)\coloneqq\frac{1}{t^{\beta \frac{N-1}{N-2}}}\int |\nabla u|^{\beta+1}\, \d \Per(\{u<t\}), \,
\]
with $\beta \ge \frac{N-2}{N-1}$, is (continuous) and non-decreasing (see the beginning of Section \ref{sec:actual monotonicity} for a discussion on the well-definition of the function $U_{\beta}$). 

\bigskip

\textbf{Rigidity and almost rigidity}

As said above we are also interested in the rigidity and almost rigidity properties of the functions $U_\beta$ and it turns out that, similarly to the Bishop-Gromov inequality, the special configurations connected to $U_\beta$ are cones. This is not a coincidence, indeed as we will see the reason behind the conical structure arising from these two quantities is the same, i.e. the existence of a function satisfying a precise  PDE: $\Delta v=N$ and $|\nabla \sqrt {2v}|^2=1$. We refer to the introduction of \cite{CMmonotone} for a more detailed parallelism with the  Bishop-Gromov inequality and the interpretation of   $U_{\beta}$ as an area functional.

We will prove that
\begin{itemize}
	\item  if the derivative of $U_{\beta}$ vanishes at $t_0$, then  $\{u<t_0\}$ is isometric to a truncated ${\rm RCD(0,N)}$ cone,
	\item  if the  derivative $U_{\beta}$ at $t_0$ is sufficiently small, then  $\{u<t_0\}$ is close, in a suitable sense, to a truncated ${\rm RCD(0,N)}$ cone.
\end{itemize}
The first result is a generalization of the rigidity result in \cite{AFM}, while the second is new even in the smooth setting.  It has to be said that almost rigidity results were already present in the work of Colding-Minicozzi (cf. with \cite[Sec. 1.7]{CMuniq} and \cite[Sec. 2.2]{CMsurvey}).  However, one of the novelty of our analysis is that, similarly to what happens for the splitting theorem (\cite{split}), we are able to show that $\{u<t_0\}$ is close to a cone which is itself an ${\rm RCD(0,N)}$ space, while they only prove closeness to a `generic' cone, moreover we prove that  almost rigidity holds for all the functions $U_{\beta}$, while they only consider the case $\beta=3.$ 

Let us point out that, as it is now well understood,  non-smooth (${\rm RCD}$) spaces enter into play naturally in almost-rigidity statements under Ricci curvature lower bounds even in the smooth case. We recall for example the existence of Riemannian manifolds having tangent cones at infinity with non-smooth cross-section (see \cite{tangentcone}).

\bigskip

\textbf{Existence of solutions}

To justify the interest on the function $U_\beta$, it is clearly important to have many example of solutions to \eqref{mainpde}. The main examples in the smooth setting are the Green function (mainly explored in the papers of Colding-Minicozzi) and the electrostatic potential, which was the object of interest in \cite{AFM}. In ${\rm RCD}$ spaces the Green function was already built and studied in \cite{constancy}, while little or nothing was know about the electrostatic potential. One of our results will be the existence of an electrostatic (or capacitary)  potential for a bounded open set $E$, under some mild regularity assumptions on its boundary. That is, we will prove the existence of a solution to \eqref{mainpde} with $\Omega=\X\setminus \bar E$, continuous up to the boundary of $\Omega$ with $u=1$ in $\partial \Omega$. Moreover we will also prove a relation that links the Cheeger energy of $u$ to the Capacity of $E$ (see \eqref{eq:capacity lim}). 

\bigskip

\textbf{New functional version of the ``(almost) outer volume-cone implies (almost) outer metric-cone"  theorem}

The rigidity and almost rigidity results for $U_{\beta}$ that we described above will follow from a new functional and ``outer"  version of the volume-cone to metric-cone theorem for ${\rm RCD}(0,N)$ spaces proved in \cite{volcon}. In particular we prove that in an ${\rm RCD}(0,N)$ space if a function $u$ satisfies $\Delta u=N$ and $|\nabla \sqrt {2u}|^2=1$  then its superlevel sets are isometric to truncated cones (see Theorem \ref{thm:functional cone}). 

Moreover we will prove an effective almost version of the above result, that appears to be new also in the smooth setting. More precisely we will show that  if a function $u$ satisfies $\Delta u=N|\nabla \sqrt {2u}|^2$ and $|\nabla \sqrt {2u}|$  is almost constant, i.e. $|\nabla|\nabla \sqrt u|| $ is small (in an integral sense), then the whole space $\X$ is close in the $pmGH$-topology to an ${\rm RCD}(0,N)$ space $\X'$, that is a truncated cone outside a bounded set (see Theorem \ref{thm:almconev1}).

We will also provide a second version of the  almost rigidity result in Theorem \ref{thm:almconev1}, which is more in the spirit of the ``almost volume annulus implies metric annulus" theorems of Cheeger-Colding (\cite{CC}). Roughly said under the same hypotheses of Theorem \ref{thm:almconev1} we prove that the sets $\{t_2<u<t_1\}$, when endowed with their intrinsic metrics, are close to annuli of an ${\rm RCD(0,N)}$ cone, also endowed with their intrinsic metrics (see Theorem \ref{thm:almconev2}).

Let us remark that a main portion of the proof of Theorem \ref{thm:functional cone}  is a repetition of the arguments in \cite{volcon}, however in writing this note the authors realized that some steps in \cite{volcon} were overlooked. For this reason along our exposition we will also take the chance to fix and adjust some of the original arguments.

\bigskip

\textbf{New  estimates for harmonic functions}

As a by-product of our argument we prove some regularity results and estimates for  harmonic functions in ${\rm RCD}(K,N)$ spaces, which appear to be new in the non-smooth setting and  interesting on their own. In particular we prove that if $u$ is harmonic, then the function $|\nabla u|^{\beta}$ has $\W$-regularity for every $\beta \ge \frac12 \frac{N-2}{N-1}$, together with an explicit bound for $|\nabla |\nabla u|^{\beta}|$ (see Theorem \ref{thm:harm estimates} for the precise statement).  It is important to observe that the exponent $\beta$ is allowed to be strictly smaller than 1, which makes the result non-trivial. The estimates that we obtain are similar and strongly inspired by the ones for Riemannian manifolds obtained in \cite{codim4}.

\bigskip

\textbf{Future applications}

This work is a part of a  project that we are pursuing and whose objective is to investigate the possibility of developing  a second order analysis on ``codimension-1" sets in ${\rm RCD}$ spaces (recall that a ``first order"  theory in this setting has already been extensively studied, see eg. \cite{bakry}, \cite{rectper} and \cite{GPcap}).  This is motivated by the recent \cite{AFM} where, in the context of smooth Riemannian manifolds, it is proved that the monotonicity of $U_\beta$ (see above) implies a family of inequalities related to the mean curvature of hyperpsurfaces (more precisely to a lower bounds on their $p$-Willmore energy).

In particular in  a forthcoming work (\cite{willmore}) we will propose a  notion of mean curvature and Willmore energy
for the boundary of subsets of ${\rm RCD}$ spaces. Our definitions will be tailored to the monotonicity results obtained in this note (in particular formula \eqref{eq:Uprimo} and the estimate \eqref{eq:U'positive})  to obtain lower bounds for the Willmore energy, in the spirit of \cite{AFM}.

\bigskip

\textbf{Plan of the paper}

The exposition will be organized as follows. In Section \ref{sec:preliminari} we will introduce the needed tools and fix some notations. In Section \ref{sec:nonpara} we introduce the notion of nonparabolic ${\rm RCD}$ space and derive its main features.
In Section \ref{sec:positive div} we will introduce a class of vector fields with nonnegative divergence (see Corollary \ref{cor:div estimate}), which are at the core of the proof of the monotonicity formula. In such section we will also deduce new estimate for harmonic functions.  In Section \ref{sec:estimates for u} we will prove some key decay estimates for solution of \eqref{mainpde} and subsequently in Section \ref{sec:actual monotonicity} we will prove the main monotonicity result. In Section \ref{sec:functioncone} we will state a new functional version of the ``from outer functional cone to outer metric-cone" theorem and prove the main new ingredients for its proof. The rest of the argument, being analogous to the one in \cite{volcon}, will be postponed to Appendix \ref{ap:cone}. The `almost' version of the previous theorem will be proved in Section \ref{sec:almcone}. Then we will use these functional rigidity and almost rigidity results to deduce in Section \ref{sec:rigidity from monotonicty} the main rigidity and almost rigidity statements from the monotonicity formula. Finally the existence for the electrostatic potential will  be given in Section \ref{sec:electro}. This argument relies on results mainly taken from \cite{bjorn}, however to make the proof self contained and more readable, in Appendix \ref{ap:bjorn} we will  redo (and simplify) the proofs of all the results we need in the setting of ${\rm RCD}$ spaces.

	\section{Preliminaries and notations}\label{sec:preliminari}
	
	\subsection{Calculus tools}
	Throughout all this note a \emph{metric measure space} (abbreviated in \emph{m.m.s.}) is a triple $(\X,\sfd,\mea)$ where
	\begin{align*}
		&(\X,\sfd) \text{ is a complete and separable metric space and $\mea$  is a nonnegative and nonzero Borel}\\
		& \text{measure on $\X$, finite on bounded sets and such that $\supp \mea=\X$.}
	\end{align*}
	We will also use the notion of \emph{pointed metric measure space} (abbreviated in \emph{p.m.m.s.}), which is a quadruple $(\X,\sfd,\mea,\bar x)$, where $(\X,\sfd,\mea)$ is a metric measure space and $\bar x \in \X.$ 
	
	Two metric measure spaces $(\X_i,\sfd_i,\mea_i)_{i=1,2}$ are said to be \emph{isomorphic} if there exists an isometry $\iota : \X_1\to \X_2$ such that $\iota_*\mea_1=\mea_2.$

	We will denote by $\LIP(\X),\LIP_{\loc}(\X),\LIP_{b}(\X),\LIP_{bs}(\X)$ and $C_{bs}(\X),$ respectively the spaces of  Lipschitz functions, locally Lipschitz functions, bounded Lipschitz functions, Lipschitz functions with bounded support and bounded continuous functions with bounded support in $(\X,\sfd)$.
	For a function $f \in \LIP_{\loc}(\X)$ we denote by $\lip f: \X \to [0,+\infty)$ its local-Lipschitz constant defined by
	\begin{equation}\label{eq:local lip}
	\lip f (x)\coloneqq \limsup_{y \to x} \frac{|f(x)-f(y)|}{\sfd(x,y)}.
	\end{equation}
	\subsubsection{Sobolev spaces via test plans}
	We will adopt the notion of Sobolev spaces on metric measures spaces via test plans introduced in \cite{AGS}. This approach turns out to be equivalent (\cite{AGS})  both to the notion of Sobolev space firstly introduced by Cheeger (\cite{cheegerlip}) and to the one introduced by Shanmugalingam (\cite{shan}).
	
	A curve $\gamma \in C([0,1],\X)$ belongs to the space of \emph{absolutely continuous} curves $AC([0,1],\X)$ if there exists $f \in L^1(0,1)$ such that $\sfd(\gamma_t,\gamma_s)\le \int_s^t f(r)\, \d r,$ for every $0\le s<t\le 1.$ In this case it holds that the limit $|\dot \gamma_t|\coloneqq\lim_{h\to 0}h^{-1}\sfd(\gamma_{t+h},\gamma_t)$ exists for a.e. $t \in (0,1)$ and is called \emph{metric speed} at time $t$. The length $L(\gamma)$ of an absolutely continuous curve $\gamma$ is defined by
	\[
	L(\gamma)\coloneqq\int_0^1|\dot \gamma_t|\, \d t.
	\]
	We recall also the  evaluation map $\e_t : C([0,1],\X)\to \X$ defined by $e_t(t)\coloneqq\gamma_t.$
	
	A Borel probability measure $\ppi$ on $AC([0,1],\X)$ is said to be a \emph{test plan} if 
	\begin{align*}
		&\exists C>0 \, : \, {\e_t}_*\ppi \le C\mea, \quad \forall t \in[0,1],\\
		&\int \int_0^1 |\dot \gamma_t|^2\, \d t\, \d \ppi<+\infty.
	\end{align*}
	\begin{definition}[Sobolev class]\label{def:sobolev class}
		The Sobolev class ${\rm S}^2(\X)$ is the space of all functions $f \in L^0(\mea)$ such that there exists a non-negative $G\in L^2(\mea)$ for which
		\begin{equation}\label{eq:sobolev def}
		\int |f(\gamma_1)-f(\gamma_0)|\, \d \ppi \le \int \int_0^1 |\dot \gamma_t|G(\gamma_t)\, \d t \, \d \ppi, \quad \forall \ppi\text{ test plan.}
		\end{equation}
	\end{definition}
	For every $f \in {\rm S}^2(\X)$ there exists a unique function $G$ with minimal $L^2(\mea)$ norm such that \eqref{eq:sobolev def} holds, called \emph{minimal weak upper gradient} and denoted by $|Df|.$ 
	The minimal weak upper gradient has the following lower semicontinuity property: suppose that $(f_n)\subset  {\rm S}^2(\X)$ converges to $f\in L^0(\mea)$ $\mea$-a.e. and $|Df_n|$ converges weakly in $L^2(\mea)$ to a function $G \in L^2(\mea)$, then 
	\begin{equation}\label{eq:stability wug}
	f\in {\rm S}^2(\X) \quad \text{ and } \quad |Df|\le G, \, \mea\text{-a.e..}
	\end{equation}
	We  define the Sobolev space $\W(\X)\coloneqq L^2(\mea)\cap {\rm S}^2(\X)$. $\W(\X)$ becomes a Banach space, when endowed with the norm 
	\[
	\|f\|_{\W(\X)}\coloneqq\sqrt{\|f\|_{L^2(\mea)}^2+\||Df|\|_{L^2(\mea)}^2}.
	\]
	The \emph{Cheeger-energy} functional $\ch: L^2(\mea)\to [0,+\infty]$ is defined by
	\[
	\ch(f)\coloneqq\begin{cases}
		\frac12\int |Df|^2\, \d \mea, & f \in \W(\X),\\
		+\infty\, & \text{otherwise}.
	\end{cases}
	\]
	It follows from \eqref{eq:stability wug} that $\ch$ is a convex and lower semicontinuous functional on $L^2(\mea)$. 
	
	A metric measure space is said to be \emph{infinitesimally Hilbertian}  if  $\W(\X)$ is a Hilbert space or equivalently if $\ch$ is a quadratic form (see \cite{Gappl}).   
\subsubsection{Tangent module}\label{sec:normed mod}
We assume the reader to be familiar with theory of $L^0$-normed modules on a m.m.s. $(\X,\sfd,\mea)$ and we refer to \cite{G14} for a detailed account on this theory. 

\begin{definition}[$L^0$-normed $L^0$-module]
	An $L^0$-normed $L^0$-module is a triple $(\mathscr M,|.|,\tau)$, where $\mathscr M$ is a module over the commutative ring $L^0(\mea)$, \((\mathscr{M},\tau)\) is a topological vector space,  $|.|: \mathscr M  \to L^0(\mea)$ is a map satisfying (in the $\mea$-a.e. sense)
	\begin{equation*}\begin{split}
			|v|\geq 0&\quad\text{ for every }v\in\mathscr M,
			\text{ with equality if and only if }v=0,\\
			|v+w|\leq|v|+|w|&\quad\text{ for every }v,w\in\mathscr M,\\
			|fv|=|f||v|&\quad\text{ for every }f\in L^0(\mea)\text{ and }v\in\mathscr M,
	\end{split}\end{equation*} 
 and such that $\tau$ is induced by the distance $\sfd_{0}(v,w)\coloneqq\int |v-w|\wedge 1\, \d \mea'$ (where $\mea'$ is a probability measure on $\X$ such that $\mea\ll\mea'\ll\mea$), which is also assumed to be complete.
\end{definition}
An $L^0$-normed module $\mathscr M$ is a Hilbert module if for every $v,w \in \mathscr M$  it holds
\[
|v+w|^2+|v-w|^2=2|v|^2+2|w|^2, \quad \text{$\mea$-a.e.}.
\]
If this is the case, by polarization we can  define a scalar product $\la,\ra: \mathscr{M}^0\times \mathscr{M}^0\to L^0(\mea)$ that is symmetric, $L^0(\mea)$-bilinear and satisfies $\la v,v\ra=|v|^2$, $|\la v,w\ra|\le |v||w|$, $\mea$-a.e. for every $v,w \in \mathscr{M}^0.$

Given an $L^0$-normed module $\mathscr{M}$ and  $E$ a Borel subset of $\X$ we define the localized module
\[
\mathscr{M}\restr{E}\coloneqq \{ \nchi_Ev \ : \ v \in \mathscr{M}\}.
\]
$\mathscr{M}\restr{E}$ inherits naturally a structure of $L^0$-normed module and  if $\mathscr{M}$ is a Hilbert module then $\mathscr{M}\restr{E}$ is Hilbert as well with $\la.,.\ra_{\mathscr{M}\restr{E}}=\nchi_E\la.,.\ra_{\mathscr{M}}$. $\mathscr{M}\restr{E}$ can also be seen as the quotient of $\mathscr{M}$ by the equivalence relation `$v \sim w$ if and only if $|v-w|=0$ $\mea$-a.e. in $E$'. This identification will be used in the rest of the note without further notice.

\begin{definition}[Tangent module]
	Suppose that $(\X,\sfd,\mea)$ is an infinitesimally Hilbertian m.m.s.. Then there exists a (unique) couple $(L^0(T\X),\nabla)$, where $L^0(T\X)$ is an $L^0$-normed module and $\nabla : \W(\X)\to L^0(T\X)$, called \emph{gradient operator},  is a linear and continuous map such that
	\begin{align*}
		&|\nabla f| \text{ coincides with the minimal weak upper gradient of $f$}\\
		&\left \{ \sum_{i=1}^n \nchi_{E_i}\nabla f_i \ :  \{E_i\}_{i=1}^n \text{ Borel partition of $\X$}, \{f_i\}_{i=1}^n \subset \W(\X) \right \} \text{ is dense in $L^0(T\X)$}.
	\end{align*}
\end{definition}
The gradient operator has the following properties
\begin{align*}
	&\text{\emph{Locality:} $\nabla f=\nabla g$, $\mea$-a.e. in $\{f=g\}$},\\
	& \text{\emph{Leibniz rule:} $\nabla(fg)=g\nabla f+f\nabla g,$ for every $f,g \in \W(\X)\cap L^\infty(\mea),$}\\
	&  \text{\emph{Chain rule:} $\nabla (\phi(f))=\phi'(f)\nabla f$, for every $f\in \W(\X),\, \phi \in C^1(\rr)\cap \LIP(\rr)$}.
\end{align*}
Finally we denote by $L^2(T\X)$ the subset of $L^0(T\X)$ containing the elements with square integrable pointwise norm.
\subsubsection{Local Sobolev spaces}
Let $(\X,\sfd,\mea)$ be a proper and infinitesimally Hilbertian m.m.s. and let $\Omega$ be an open subset of $\X$. We define 
		\[ \W_\loc(\Omega ):=\{u \in L_\loc^2(\Omega) \ | \ u\eta \in \W(X), \text{ for every }\eta \in \LIP_\c(\Omega)\} ,\]
which is actually equivalent to ask that for every Borel set  $\Omega'\subset\subset \Omega$ there exists $u' \in \W(\X)$ such that $u=u'$, $\mea$-a.e. in $\Omega'.$

For any $u\in \W_\loc(\Omega )$ we define its gradient $\nabla u$ as the unique element of $L^0(TX)\restr{\Omega}$ such that
\[\nabla u\coloneqq\nabla (\eta u),\text{ $\mea$-a.e. in  $\{\eta=1\}$}, \quad \forall\eta \in \LIP_\c(\Omega), \]
which is well defined thanks to the locality property of the gradient. In particular $|\nabla u|\in L^2_{\loc}(\Omega)$. It is straightforward to check that $\nabla u$ satisfies the expected locality property, Leibniz rule and chain rule. We only state explicitly a version of the chain rule that we will need:
if $ u \in \W_{\loc}(\Omega)$ and  $\phi \in C^1(I)$, with $I$ open interval, are such that 
\begin{equation}\label{eq:chain rule hyp}
u(\Omega')\subset \subset I \text{ (up to a $\mea$-negligible set),  for every }\Omega'\subset\subset\Omega,
\end{equation}
then $\phi(u)\in \W_{\loc}(\Omega)$ and $\nabla \phi(u)=\phi'(u)\nabla u.$ 

Observe that, since $\nabla u\in L^0(T\X)\restr{\Omega}$, it makes sense to compute the scalar product $\la \nabla u,v\ra\in L^0(\Omega,\mea)$, for every $v \in L^0(T\X)\restr{\Omega}$, moreover this this scalar product also satisfies $|\la \nabla u,v\ra|\le |\nabla u||v|,$ $\mea$-a.e. in $\Omega$ (recall the discussion in Section \ref{sec:normed mod}).

We also define the spaces
\begin{align*}
	&\W(\Omega ):=\{f \in \W_\loc(\Omega ) \ | \ f, |\nabla f|\in L^2(\Omega)\},\\
	& \W_0(\Omega ):=\overline{\LIP_c(\Omega)}^{\W(X)}\subset \W(\X) .
\end{align*}

We end this subsection with the following technical lemma.
\begin{lemma}\label{lem:radice sobolev}
	Let $u \in \W_{\loc}(\Omega)$ be nonnegative and $\alpha \in(0,1)$ be such that
	\[
	\nchi_{\{u>0\}} u^{\alpha-1}|\nabla u| \in L^{2}_{\loc}(\Omega).
	\]
	Then $u^{\alpha }\in \W_{\loc}(\Omega)$ and $\nabla u=\alpha^{-1} u^{1-\alpha}\,\nabla u^{\alpha}$, $\mea$-a.e. in $\Omega.$
\end{lemma}
\begin{proof}
	For the first part it is enough to show that $f\coloneqq\eta u^{\alpha}\in \W(\X)$ for every $\eta \in \LIP_{c}(\Omega)$ with $|\eta|\le 1$. Fix $\eps\in(0,1)$ arbitrary and define $f_{\eps}=\eta (u+\eps)^{\alpha}.$ Then from the non-negativity of $u$ we have $f_{\eps}\in \W(\X)$ and recalling that $|\nabla u|=0$ $\mea$-a.e. in $\{u=0\}$ we have
	$$|\nabla f_{\eps}|\le \Lip \eta\, C_{\alpha}(u+2)+\alpha \nchi_{u>0} u^{\alpha-1}|\nabla u|,\quad \text{$\mea$-a.e..}
	$$
	It follows that the family $\{f_{\eps}\}_{\eps \in (0,1)}$ is  bounded in $\W(\X)$. Moreover $f_{\eps}\to f$ in $L^2$, therefore from the lower semicontinuity of the Cheeger energy it follows that  $f \in \W(\X)$. For the second part we observe that $u\wedge n=(u^{\alpha}\wedge n^\alpha)^{1/\alpha}$ $\mea$-a.e. in $\Omega$ and that $(t\wedge n^{\alpha})^{1/\alpha}\in \LIP(\rr)$. Therefore from the locality and the chain rule for the gradient 
	$$\nabla u=\alpha^{-1} u^{1-\alpha}\,\nabla u^{\alpha}, \quad \mea\text{-a.e. in } \{u\le n\}$$
	and we conclude from the arbitrariness of $n.$
\end{proof}

\subsubsection{Laplacian and divergence operators}
In this section we assume $(\X,\sfd,\mea)$ to be a  proper and infinitesimally Hilbertian  m.m.s.
	
	\begin{definition}[Measure-valued Laplacian (\cite{Gappl})]
		Let $\Omega\subset \X$ open. We say that  and $u \in \W_\loc(\Omega) $ belongs to the domain of the \emph{measure-valued Laplacian} $D(\bd,\Omega)$ if there exists a Radon measure $\bd\restr {\Omega} u$ in $\Omega$ such that
		\begin{equation}\label{deflap}
		-\int_{\Omega} \langle \nabla f, \nabla u\rangle \d \mea=\int_{\Omega} f \d \bd\restr{\Omega} u, 
		\end{equation}
		for every $f \in \LIP_c(\Omega).$
	\end{definition}
Let us remark that in the above definition with the term \emph{Radon measure} we denote a set function $\mu : \{ \text{Borel sets relatively compact in $\Omega$} \}\to \rr$ which can be written as $\mu(B)=\mu^+(B)-\mu^-(B)$, for some  positive Radon measures $\mu^+,\mu^-$. In particular we do not require $\mu$ to be a Borel measure on the whole $\Omega$, this weaker assumption is needed for example in Proposition 	\ref{prop:laplineq} and to write the Laplacian of the distance function (see also the discussion in \cite{CMlapl}).


When no confusion can occur we will drop the subscript $\Omega$ and simply write $\bd u.$ Moreover we will write $D(\bd)$ in place of $D(\bd,\X)$ and whenever $\bd \ll \mea$ we will use the non bold notation $\Delta.$

A function $u \in D(\bd,\Omega)$ is said to be \emph{subharmonic} if $\bd u\ge 0$, superharmonic if $\bd u\le 0$ and harmonic if $\bd u=0$.

Finally it easily follows from the definition that the Laplacian operator is linear and satisfies the following locality property:
\[
\text{if $u,v \in D(\bd,\Omega)$ and $u=v$ $\mea$-a.e. in $U$, with $U$ open and relatively compact in $\Omega$, then $\bd u\restr U=\bd v\restr U$.}
\]
The following existence and comparison result is proven in \cite[Prop 4.13]{Gappl}.
\begin{prop}\label{prop:laplineq}
	Let $u \in \W_{\loc}(\Omega)$ and suppose that there exists $g \in L^1_{\loc}(\Omega)$ such that
	\[
	-\int_{\Omega} \la \nabla f, \nabla u\ra \, \d \mea\ge \int_{\Omega} gf \, \d \mea, \quad \forall f \in \LIP_c(\Omega), \text{ with } f \ge 0,
	\]
	Then $u \in D(\bd,\Omega)$ and $\bd u\ge g\mea \restr{\Omega}.$
\end{prop}

\begin{remark}\label{extension1}
	Let $u \in D(\bd,\Omega)$. Suppose that $\bd u \ll \mea $  (resp. $\bd u\ge g \mea$) with $\frac{\d\bd u }{\d\mea}\in L^2_{\loc}(\Omega)$ (resp.  $g\in L^1_{\loc}(\Omega)$). Then, recalling that in an infinitesimally Hilbertian m.m.s. Lipschitz functions are dense in $\W(\X)$ (see \cite{AGS}), by a truncation and cut off argument it follows that \eqref{deflap} (resp.  $-\int\la \nabla f, \nabla u \ra \d \mea \ge \int gf\d \mea $) holds also for every $f \in \W(X)$ (resp. $f \in \W(X)\cap L^\infty(\mea)$, $f \ge 0$ ) with support compact in $\Omega.$ \fr
\end{remark}

		\begin{definition}[Measure-valued divergence (\cite{GM})]
		Let $v \in L^0(TX)\restr{\Omega}$ be such that $|v|\in L^2_{\loc}(\Omega)$, we say that $v \in D({\bf{div}},\Omega)$ if there exists a Radon measure ${\bf{div}}\restr{\Omega}(v)$ in $\Omega$ such that
		\begin{equation}\label{defdiv}
		 \int_{\Omega} \la \nabla f,v\ra\d \mea=-\int_{\Omega} f \, \d\, {\bf{div}}\restr \Omega(v), 
		 \end{equation}
		for every $f \in \LIP_c(\Omega).$
	\end{definition}
	It is clear from the definition that given $u \in \W_{\loc}(\Omega)$, we have $\nabla f \in D({\bf{div}},\Omega)$ if and only if $u \in D(\bd,\Omega)$ and in this case ${\bf{div}}\restr{\Omega}(\nabla u)=\bd \restr{\Omega}u.$ As for the Laplacian, we will often write simply ${\bf div}(v)$ instead of ${\bf{div}}\restr \Omega(v)$.

\begin{remark}\label{extensiondiv}
	Analogously to the measure valued Laplacian, we have that if ${\bf{div}}(v)  \in L^2_{\loc}(\Omega)$, then \eqref{defdiv}  holds also for every $f \in \W(X)$ with support compact in $\Omega.$\fr 
\end{remark}

\subsection{RCD spaces}\label{sec:rcd}
We assume the reader to be familiar with the definition and theory of ${\rm RCD}(K,N)$ spaces (see \cite{AGSflow}, \cite{Gappl}). We limit ourselves to recall some of their main properties that will be needed.

The \emph{Sobolev-to-Lipschitz property} holds (the definition we recall comes from \cite{split} and so does the argument that we adopt to prove the `local version’ below - the validity of this property on ${\rm RCD }(K,\infty)$ spaces was known from \cite{AGSrcd}): for every $f \in \W(\X)$ such that $|Df|\in L^\infty(\mea)$, $f$ has a Lipschitz representative and $\Lip f\le \||D f|\|_{L^\infty}$.

The local variant we will actually use is the following: 
\begin{prop}[Local Sobolev-to-Lipschitz property]\label{prop:sob to lip}
	Let $\X$ be an ${\rm RCD}(K,N)$ space, $K \in \rr$ and $N\in[1,+\infty)$ and let $\Omega \subset \X$ be open. Suppose  $f \in \W_{\loc}(\Omega )$ is such that $\||\nabla f|\|_{L^{\infty}(\Omega)}<+\infty $, then $f$ has a locally-Lipschitz representative. Moreover for such representative it holds
	\begin{equation}\label{eq:localsobolevtolip}
		|f(x)-f(y)|\le\||\nabla f|\|_{L^{\infty}(\Omega)}\, \sfd(x,y),
	\end{equation}
	for every $x,y \in \Omega$ such that $\sfd(x,y)\le \sfd(x,\partial \Omega).$
\end{prop}
\begin{proof}
	The fact that $f$ has a locally-Lipschitz representative it follows from the Sobolev-to-Lipschitz property and a cut-off argument.
	
	For the second part we observe that it is sufficient to consider the case $\sfd(x,y)<\sfd(x,\partial \Omega),$ since the equality case follows  by continuity. Hence for some $r>0$ we have $\sfd(x,y)<r<\sfd(x,\partial \Omega)$ and we can consider $\tilde f \in \W(\X)$ such that $\tilde f =f $ $\mea$-a.e. in $B_{r}(x)$. Then for $\eps< (r-\sfd(x,y))/4$ we define $\mu^0_{\eps}\coloneqq\mea\restr{B_{\eps}(x)}\mea(B_{\eps}(x))^{-1}$, $\mu^1_{\eps}\coloneqq\mea\restr{B_{\eps}(y)}\mea(B_{\eps}(y))^{-1}$. thanks to the results in \cite{rajala} and \cite{rajalasturm} there exists a unique $\ppi^{\eps}\in {\sf OptGeo(\mu^0_{\eps},\mu^1_{\eps})}$ such that ${\e_t}_*\ppi^{\eps}\le C\mea$, $\forall t \in[0,1]$, for some constant $C$ depending on ${\eps}$. In particular $\ppi^{\eps}$ is a test plan. Moreover  from the triangle inequality it follows that $\ppi^{\eps}$ is concentrated on curves $\gamma$ with support contained in $B_{r}(x)$. 
Therefore from \eqref{eq:sobolev def}
	\begin{align*}
		\left | \int  f \,\d \mu^1_{\eps}- \int f \, \d\mu^0_{\eps}\right|\le \int |\tilde f(\gamma_1)-\tilde f(\gamma_0)|\, \d \ppi^{\eps}(\gamma)\le \||\nabla \tilde f|\|_{L^{\infty}(B_{r}(x))} \int \int_0^1|\dot \gamma_t| \d t d\ppi^{\eps}\le  \||\nabla f|\|_{L^{\infty}}W_2(\mu^0_{\eps},\mu^1_{\eps}).
	\end{align*}
	Letting $r\to 0^+$  from the continuity of $f$ we obtain \eqref{eq:localsobolevtolip}.
\end{proof}
From Proposition \ref{prop:sob to lip} it also follows that
\begin{equation}\label{eq:loc constant}
	\text{if $\Omega$ is connected, $u \in \W_{\loc}(\Omega)$ and $|\nabla u|=0$, $\mea$-a.e., then $u$ is constant in $\Omega.$}
\end{equation}
The \emph{Bishop-Gromov inequality} holds (see \cite{sturm2}), i.e. 
\[
\frac{\mea(B_R(x))}{v_{K,N}(R)}\le \frac{\mea(B_r(x))}{v_{K,N}(r)}, \quad \text{for any $0<r<R$ and any $x \in \X$},
\]
where for the quantities $v_{K,N}(r)$ coincides, for $N \in \mathbb{N}$, with the volume of the ball with radius $r$ in the model space of dimension $N$ and curvature $K$ (see \cite{sturm2} for the definition of $v_{K,N}(r)$ for  arbitrary $N\in[1,\infty)$). In particular $(\X,\sfd,\mea)$ is proper and uniformly locally doubling. We also note that in the case $K=0$ this implies that the limit 
\[
{\sf AVR(\X)}\coloneqq\lim_{r\to +\infty}\frac{\mea(B_r(x))}{r^N}
\]
exists finite and does not depend on the point $x \in\ X$. We call the quantity ${\sf AVR(\X)}$ \emph{asymptotic volume ratio} of $\X$ and if ${\sf AVR(\X)}>0$ we say that $\X$ has \emph{Euclidean-volume growth}.

We will need the following \emph{Laplacian comparison} for ${\rm RCD}(0,N)$ spaces (see	\cite[Corollary 5.15]{Gappl}):
\begin{equation}\label{eq:laplacian comparison}
	\text{$\sfd(x_0,.)^2 \in D(\bd)$ and $\bd \sfd(x_0,.)^2 \le 2N\mea $,} \quad \text{ for every }x_0\in \X,
\end{equation}
moreover in any ${\rm RCD}(K,N)$ space with $N<+\infty$ it holds that
\begin{equation}\label{eq:grad dist}
	|\nabla \sfd (x_0,.)|=1, \quad \mea\text{-a.e.}.
\end{equation}
From to the results in \cite{cheegerlip} and the fact that $\lip \sfd(x_0,.)\equiv 1$ it follows that \eqref{eq:grad dist} actually holds in the general setting of doubling m.m.s. satisfying a Poincaré inequality, however a more direct proof in the setting of ${\rm RCD}(K,N)$ spaces is also available (see for example \cite[Prop. 3.1]{GPmeasure}).

Recall that, since  ${\rm RCD}(K,N)$ spaces are infinitesimally Hilbertian, the \emph{heat flow}  $h_t: L^2(\mea)\to L^2(\mea)$, $t \ge 0$, defined as the gradient flow of $\ch$ on $L^2(\mea)$ is linear, continuous, self-adjoin and satisfies $h_t(f)\in D(\Delta)\cap \W(\X)$, moreover the curve $(0,\infty)\ni t \mapsto h_t f \in L^2(\mea)$ is locally absolutely continuous for every $f \in L^2(\mea)$ and
$$\frac{\d }{\d t}h_t(f)=\Delta h_t(f)\in L^2(\mea), \quad \text{for a.e. } t>0,$$
(see \cite{G14} for further details). Moreover $h_t$ has the so called \emph{$L^\infty$-to Lipschitz regularization property} (see \cite{AGSrcd}), i.e. there exists a constant $C(K)>0$ such that for every $f \in L^\infty\cap L^2(\mea)$ it holds that $|\nabla h_t f|\in L^\infty(\mea)$ and 
\begin{equation}\label{eq:BE}
	\||\nabla h_t f|\|_{L^\infty(\mea)}\le \frac{C(K)}{\sqrt t} \|f\|_{L^\infty(\mea)}, \quad \, \forall \,  t \in(0,1).
\end{equation}

In an ${\rm RCD}(K,N)$ space it can be given also a notion of heat kernel (see \cite{AGSrcd}) $p_t: \X\times \X \to [0,+\infty]$ which has a locally Hölder-continuous representative (see \cite{sturmkernel1,sturmkernel2}), which satisfies the following pointwise bounds (\cite{JLZ}), generalizing the classical estimates of Li and Yau in the smooth case \cite{liyau}:
\begin{equation}\label{eq:kernelestimate}
	\begin{split}
		&\frac{1}{C_1\mea(B(x,\sqrt{t}))}\exp\left\lbrace -\frac{\sfd^2(x,y)}{3t}-ct\right\rbrace\le p_t(x,y)\le \frac{C_1}{\mea(B(x,\sqrt{t}))}\exp\left\lbrace-\frac{\sfd^2(x,y)}{5t}+ct \right\rbrace,\\  
		&	\left |{\nabla p_t(x,\cdot)}(y)\right |\le \frac{C_1}{\sqrt{t}\mea(B(x,\sqrt{t}))}\exp\left\lbrace -\frac{\sfd^2(x,y)}{5t}+ct\right\rbrace \quad\text{for $\mea$-a.e. $y\in X$},
	\end{split}
\end{equation}
for any $x,y\in X$, for any $t>0$ and where $c,C_1$ are positive constants depending only on $K,N$ such that $c=0$ if $K=0$.

We introduce the algebra of test functions $\test(\X)$ (\cite{savarè}) defined as
\[
\test(\X)\coloneqq\{f \in \ L^\infty(\mea) \cap \LIP(\X)\cap D(\Delta) \ | \ \Delta f \in \W(\X)\}.
\]
It turns out that $\test(\X)$ is dense in $\W(\X)$ and $\la\nabla f,\nabla g \ra\in \W(\X)$ for every $f,g \in \W(\X).$

We recall the notion of Hessian for a test function as constructed in  \cite{G14}:  for any $f \in \test(\X)$ there exists $\H f : [L^0(T\X)]^2 \to L^0(\mea)$,  $L^0$-bilinear symmetric and continuous in the sense that $|\H f(v,w)|\le |\H f|_{OP} |v||w|$, for some (minimal) function $|\H f|_{OP}\in L^2(\mea)$. Moreover $\H f$ is  characterized by the identity
\begin{align*}
		2\int h\H f&(\nabla f_1, \nabla f_2) \d \mea =\\
		&- \int \la\nabla f,\nabla f_1\ra \text{div}(h\nabla f_1)+\nabla f,\nabla f_2 \text{div}(h\nabla f_2)- \text{div}(h\nabla f) \la\nabla f_1, \nabla f_2\ra  \d \mea,
\end{align*}
for any choice of $h,f_1,f_2 \in \test(\X)$.   The Hessian also induces an $L^0$-linear and continuous map $\H f : L^0(T\X)\to L^0(T\X)$ characterized by
\[
\la\H f (v),w\ra=\H f(v,w), \quad \text{$\mea$-a.e.}, \quad \forall w \in L^0(T\X)
\]
and which satisfies $|\H f(v)|\le |\H f|_{OP}|v|.$ We recall also the following identity, which essentially is contained in \cite[Prop. 3.3.22]{G14},
\begin{equation}\label{eq:hessian gradient}
2\H f (\nabla f)=\nabla |\nabla f|^2, \quad \forall f \in \test(\X).
\end{equation}

Combining \cite[Theorem 5.1]{GPtangent} with the recent  \cite{constancy} we have the following result (we refer to \cite{G14} for the definition of dimension and local base for a normed module).
\begin{theorem}[Constancy of the dimension]\label{thm:constancy}
	Let $\X$ be any ${\rm RCD}(K,N)$  space with $K\in \rr$ and $N \in [1,\infty)$. Then there exists an integer $\dim(\X)\in [1,N]$ such that the tangent module $L^0(T\X)$ has constant dimension equal to $\dim(\X)$. 
\end{theorem} 
A corollary of this result is the existence of a global orthonormal base: $\{e_1,...,e_{\dim(\X)}\}\subset L^0(T\X)$ such that $\la e_i,e_j\ra=\delta_{i,j}$ $\mea$-a.e. for every $i,j=1,...,\dim(\X)$. In particular   for every $v \in L^0(T\X)$ it holds that $v=\sum_{i=1}^{\dim(\X)}v_i e_i$, where $v_i\coloneqq\la v,e_i\ra$. Then, denoted by $(Hf)_{i,j}$ the functions $\H f (e_i, e_j)\in L^0(\mea)$, for $i,j \in \{1,...,\dim(\X)\}$, we can write
\begin{align}
	\H f (v)= \sum_{1\le i,j\le \dim(\X)}(Hf)_{i,j}v_j	\,e_i,\quad \forall v \in L^0(T\X)\label{eq:coordmap}.
\end{align}
Moreover we define the trace and Hilbert-Schmidt norm ${\sf tr}\H f,|\H f|_{HS}  \in L^2(\mea)$ as 
\begin{align}
		&|\H f|_{HS}^2 \coloneqq \sum_{1\le i,j\le \dim(\X)}(Hf)_{i,j}^2, \label{eq:coordHS},\\
	&{\sf tr} \H f\coloneqq \sum_{1\le i\le \dim(\X)}(Hf)_{i,i},\label{eq:coordtr}
\end{align}
which are well defined in the sense that they do not depend on the choice of the base, as can be easily verified by a direct computation. It always holds that $|\H f|_{OP}\le |\H f|_{HS}$ $\mea$-a.e.\ (see \cite[Sec. 3.2]{G14}).

 In view of the above we can restate Theorem 3.3 in \cite{HanRicci} as follows (see also \cite{eks},\cite{AMS} for the ``basic version" of the Bochner inequality.)
\begin{theorem}[Improved Bochner-inequality]
	Let $\X$ be any ${\rm RCD}(K,N)$  space with $K\in \rr$ and $N \in [1,\infty)$. Then for any $f \in \test(\X)$ it holds that $|\nabla f|^2 \in D(\bd)$ and
	\begin{equation}\label{eq:improvedboch}
		\bd \left( \frac{|\nabla f|^2}{2}\right) \ge \left(|\H f|_{HS}^2+K|\nabla f|^2+\la\nabla f,\nabla \Delta f\ra+ \frac{(\Delta f-{\sf tr}\H f)^2}{N-\dim(\X)}\right)\mea,
	\end{equation}
	where $\frac{(\Delta f-{\sf tr}\H f)^2}{N-\dim(\X)}$ is taken to be 0 in the case $\dim(\X)=N.$
\end{theorem}

See \cite{MN} for a proof of the following result.
\begin{prop}[Good cut-off functions]\label{prop:goodcutoff}
	Let $\X$ be an ${\rm RCD}(K,N)$ space, $K \in \rr$ and $N\in (1,\infty)$. Then for every $0<r<R<+\infty$, every compact set $K$ and every open set $U$ containing $K$, such that $\diam(U)<R$ and $\sfd(K,U^c)>r$, there exists a function $\eta \in \test(\X)$ such that
	\begin{enumerate}
		\item $0\le \eta \le1$, $\eta=1$ in $K$ and $\supp \eta \subset U$,
		\item $|\nabla \eta|+|\Delta \eta|\le C(R,r,N,K)$,
	\end{enumerate}
moreover $C$ does not depend on $R$ in the case $K=0.$
\end{prop}
We recall that for any $N \in [1,\infty)$ the Euclidean $N$-cone over a m.m.s. $(Z,\mea_Z,\sfd_Z)$ is defined to be the space $C(Z)\coloneqq[0,\infty)\times Z/ (\{0\}\times Z)$ endowed with the following distance and measure
\begin{align*}
	&\sfd_{C(Z)}((t,z),(s,z'))\coloneqq \sqrt{s^2+t^2-2st\cos(\sfd_Z(z,z')\wedge \pi)},\\
	&\mea_{C(Z)}\coloneqq t^{N-1}\d t\,\otimes\, \mea_Z.
\end{align*}
It was proven in \cite{ket} that 
\begin{equation}\label{eq:ketterer}
	\begin{split}
	&\text{if the Euclidan $N$-cone ($N\ge 2$) over a m.m.s. $Z$ is an ${\rm RCD}(0,N)$ space, }\\
	&\text{then $\diam(Z)\le \pi$ and $Z$ is an ${\rm RCD}(N-2,N-1)$ space.}
	\end{split}
\end{equation}

\subsubsection{Localized Bochner inequality}\label{sec:local bochner}
Along this subsection $\Omega$ is an open subset of $\X$. We define 
\[{\sf{Test}}_\loc(\Omega):=\{u \in \LIP_\loc(\Omega)\cap D(\bd,\Omega)\ | \  \bd u \in  \W_\loc(\Omega)\}. \]
This definition is motivated by the following observation:
\begin{equation}\label{loctotest}
	\text{$\eta \in {\sf{Test}}(X)$ with $\supp \eta \subset \subset \Omega$, $u \in {\sf{Test}}_\loc(\Omega)$ $\implies$ $\eta u \in  {\sf{Test}}(X).$}
\end{equation}
This follows from the fact that  for any  $f \in \LIP(\X)\cap D(\Delta)$ it holds that $|\nabla f|^2 \in \W(\Omega)$, which is essentially a consequence of  the following inequality  (see \cite[Corollary 3.3.9]{G14})
\[
	\int |\H f|_{HS}^2\, \d \mea \le \int (\Delta f)^2\, \d \mea -K\int |\nabla f|^2\, \d \mea, \quad \forall f \in \test \X.
\]
 \eqref{loctotest} allows us to define for every $u \in \test_{\loc}(\Omega)$ the functions $|\H u|_{HS}, {\sf tr} \H u\in L^2_{\loc}(\Omega)$ as 
\begin{align*}
	&|\H u|_{HS}\coloneqq|\H{\eta u}|_{HS},\, \mea\text{-a.e. in $\Omega'$},\\
	&{\sf tr} \H u\coloneqq{\sf tr} \H {\eta u},\, \mea\text{-a.e. in $\Omega'$},
\end{align*}
for every $\eta \in \test(\X)$ with compact support in $\Omega$ and such that $\eta=1$ in $\Omega'\subset \subset \Omega$. This definition is well posed thanks to the locality property of the Hessian (see \cite[Prop. 3.3.24]{G14}). Recall also from Proposition \ref{prop:goodcutoff} that many such functions $\eta$  exist.  The following follows directly from \eqref{eq:improvedboch} and the above definitions.
\begin{prop} [Local improved Bochner Inequality]
	Let $u \in \mathsf{Test}_\loc(\Omega)$, then $|\nabla u|^2 \in D(\bd,\Omega)$ and
	\begin{equation}\label{localbochner}
		\bd\restr{\Omega} (|\nabla u|^2)\ge 2\left (|\H u|_{HS}^2+\frac{(\Delta u-{\sf tr}\H { u})^2}{N-\dim(\X)}+\langle \nabla u, \nabla \Delta u\rangle+K|\nabla u|^2 \right )\mea|_{\Omega },
	\end{equation}
	where $\frac{(\Delta u-{\sf tr}\H u)^2}{N-\dim(\X)}$ is taken to be 0 in the case $\dim(\X)=N.$
\end{prop}
We conclude this subsection observing that
\begin{equation}\label{eq:sobolev gradient}
\text{if $u\in {\sf{Test}}_\loc(\Omega),$ then $|\nabla u|\in \W_{\loc}(\Omega)$ and $\nabla |\nabla u|^2=2|\nabla u| \nabla |\nabla u|$,}
\end{equation}
indeed from \eqref{eq:hessian gradient} it follows that  $|\nabla |\nabla u|^2||\nabla u|^{-1}\le 2|\H u|_{OP},$ $\mea$-a.e. in $\{|\nabla u|>0\}$ and the claim follows applying Lemma \ref{lem:radice sobolev} with $\alpha=1/2$ and $|\nabla u|^2$ in place of $u$.
\subsubsection{(Sub)harmonic functions in RCD spaces}
In this subsection $(\X,\sfd,\mea)$ is an ${\rm RCD}(K,N)$  m.m.s. with $N \in[1,\infty)$.
\begin{prop}\label{prop:maxprinc}
	Let $\Omega$ be an open and bounded subset of  $\X$  and let $u \in D(\bd, \Omega)$ be such that $\bd u \ge 0$. Then 
	\begin{itemize}
		\item weak maximum principle: if $u$ is upper semicontinuous in $\bar \Omega$, then
		\[ {\rm ess} \sup_{\Omega} u \le \sup_{\partial \Omega} u, \]
		\item strong maximum principle:   
		\[\text{if $u \in C(\Omega)$, $\Omega$ is connected and $u(x)=\sup_{\Omega}u$ for some $x \in \Omega$, then $u$ is constant.}\]
	\end{itemize}
\end{prop}
\begin{proof}
	For the first part suppose by contradiction that $\text{ess} \sup_{\Omega} u > \sup_{\partial \Omega} u$. Then there exists $c \in \rr$ such that $\sup_{\partial \Omega} u<c<\text{ess} \sup_{\Omega} $, in particular from the upper semicontinuity and compactness of $\partial \Omega$ we have that $f-\min(f,c)\in \W(\X)$ with compact support in $\Omega$. From this point the proof continue exactly as in  \cite[Thm. 2.3]{grigoni}.
	
	For the second part we apply the strong-maximum principle in \cite[Thm. 2.8]{grigoni} to a ball $B_r(x)\subset \Omega$, obtaining that $u$ is constantly equal to $\sup_{\Omega}u$ in $B_r(x).$ In particular the set $\{u=\sup_{\Omega}u\}\cap \Omega$ is open (and closed in $\Omega$) and thus must coincide with $\Omega$.
\end{proof}

It is known that in  ${\rm RCD}(K,N)$ spaces harmonic functions are continuous. It follows for example from the existence of Harnack inequalities (see Appendix \ref{ap:bjorn}) that they are locally Hölder.  It turns out that they are actually locally-Lipschitz and that the following gradient bound,   analogous to the one due to Cheng and Yau (\cite{chengyau})  in the smooth setting, holds (\cite[Theorem 1.2]{jiang}). We state it only in the case $N=0$ (see also \cite{bobo} for an analogous result).
\begin{theorem}[Gradient estimate]
	For every $N \in [1,\infty)$ there  exists a positive constant $C=C(N)$ such that the following holds. Let $\X$  be  an ${\rm RCD}(0,N)$ space and let $u\in D(\bd,B_{2R}(x))$ be positive and harmonic in $B_{2R}(x)$, then
	\begin{equation} \label{eq:cheng}
		\left \|\frac{|Du|}{u} \right \|_{L^{\infty}(B_R(x))} \le \frac{C}{R}.
	\end{equation}
\end{theorem}

\begin{lemma}\label{ascoliarm}
    Let $\X$ be an ${\rm RCD}(0,N)$ space and let $\{u_i\}$ be a sequence of harmonic and continuous functions in $\Omega$. Suppose that
    \[\sup_\Omega |u_i|<C,  \]
    then there exists a subsequence $u_{i_k}$ that converges locally uniformly to a function $u$ harmonic in $\Omega$.
\end{lemma}
\begin{proof}
    The existence of a (non relabelled) subsequence $u_i$ converging locally uniformly to a continuous function $u$ follows from the gradient estimate  \eqref{eq:cheng} and Ascoli-Arzelà. It remains to prove that $u$ is harmonic in $\Omega.$ Fix $\Omega' \subset \subset \Omega$ and $\eta \in \Lip_c(\Omega)$ with $\eta=1$ in $\Omega'$. Then $\eta u_i\in \W(\X)$ converges to $u\eta $ in $L^2(\Omega').$ Moreover, again by \eqref{eq:cheng}, we have
    \[ \sup_i \||\nabla (\eta u_i)|\|_{L^2(\mea )} <+\infty.\]
    In particular (recall that $\W(\X)$ is Hilbert) up to a subsequence $u_i\eta  \rightharpoonup \eta u$ in $\W(\X)$ and  therefore (from the locality of the gradient) $\int_{\Omega} \la \nabla u, \nabla f\ra\, \d \mea=0$ for every $f \in \LIP_{c}(\Omega')$ (see also \cite[Prop 5.19]{Gappl} for a similar limiting argument). From the arbitrariness of $\Omega'$  we deduce  both that $u\in \W_{\loc}(\Omega)$ and that $u$ is harmonic  in $\Omega.$
\end{proof}

\subsubsection{pmGH-convergence}
We will adopt the following definition of {\emph{pointed-measure Gromov Hausdorff convergence}} of pointed metric measure spaces, which is equivalent to the usual one in the case when $\supp \mea_n=X_n$ and the measures $\mea_n$ are uniformly locally doubling. We refer to  \cite{GMSconv} for a discussion on the various notions of convergence of p.m.m. spaces and their  relations.

\begin{definition}[\emph{pointed-measure Gromov Hausdorff convergence}]\label{def:pmgh}
	We say that the sequence  $(\X_n,\sfd_n,\mea_n,x_n)$ of p.m.m.s. \emph{pointed-measure Gromov Hausdorff}-converges (pmGH-converges in short) to  the p.m.m.s.  $(\X_\infty,\sfd_\infty,\mea_\infty,x_\infty)$, if 
	there are sequences $R_n\uparrow+\infty$, $\eps_n\downarrow 0$ and Borel maps $f_n:X_n\to X_\infty$ such that 
	\begin{itemize}
		\item[1)] $f_n( x_n)=x_\infty$,
		\item[2)] $\sup_{x,y\in B_{R_n}(x_n)}|\sfd_n(x,y)-\sfd_\infty(f_n(x),f_n(y))|\leq\eps_n$,
		\item[3)] the $\eps_n$-neighborhood of $f_n(B_{R_n}( x_n))$ contains $B_{R_n-\eps_n}( x_\infty)$,
		\item[4)] for any $\varphi\in C_{bs}(X_\infty)$ it holds $\lim\limits_{n\to\infty}\int \varphi\circ f_n\,\d\mea_n=\int\varphi\,\d\mea_\infty$.
	\end{itemize}
\end{definition}

It is proven in \cite{GMSconv} that  there exists a distance $\sfd_{pmGH}$  that metrizes the pmGH-convergence for the class of ${\rm RCD}(K,N)$ spaces with $K\in \rr$ and $N<+\infty$ fixed, and more generally for every family of uniformly locally doubling metric measure spaces.	
	
	It will be also useful to recall the so called \emph{extrinsic approach} (see \cite{GMSconv}) to pmGH-convergence, valid  in the case of uniformly doubling and  geodesic m.m.s.:  the pmGH-convergence can be realized by  a proper metric space $(Y,\sfd)$ where $\X_i,\X_\infty$ are subsets of $Y$ such that $\sfd_Y\restr{\X_i\times \X_i}=\sfd_i,\sfd_Y\restr{\X_\infty\times \X_\infty}=\sfd_\infty$, $\sfd_{Y}(x_i,x_\infty)\to 0$, $\mea_i\rightharpoonup\mea_\infty$ in duality with $C_{bs}(Y)$ and  $\sfd^Y_H(B^{\X_n}_R(x_n),B^{\X_\infty}_R(x_\infty))\to 0$ for every $R>0.$

After the works in  \cite{sturm1},\cite{sturm2}, \cite{villani},  \cite{GMSconv}, \cite{AGSrcd} and in view of the Gromov compactness theorem \cite[Sec. 5.A]{gromov}, the following fundamental compactness result for RCD spaces is known.
\begin{prop}\label{prop:compactness}
	Suppose $(\X_n,\sfd_n,\mea_n,x_n)$ are ${\rm RCD}(K_n,N)$ spaces with $N\in[1,\infty)$, $K_n\to K\in \rr$ and $\mea(B_1(x_n))\in[v^{-1},v]$ for some $v>0$. Then there exists a subsequence $(\X_{n_k},\sfd_{n_k},\mea_{n_k},x_{n_k})$ that pmGH-converges to an ${\rm RCD}(K,N)$  space $(X_\infty,\sfd_\infty,\mea_\infty,x_\infty)$.  
\end{prop}

\subsubsection{Stability results under pmGH-convergence}

In this subsection $(X_i,\sfd_i,\mea_i,x_i)$ is a pmGH-converging sequence of ${\rm RCD}(K,N)$ spaces, $N<+\infty$, and $(Y,\sfd)$ is a proper metric space which realizes such convergence through the extrinsic approach (see the previous subsection). 

\begin{definition}[Locally uniform / uniform convergence]\label{def:pointwise conv}
	Let $f_i: \X_i \to \rr$, $f_\infty : \X_\infty\to \rr$. We say that $f_i$ converges locally uniformly to $f_\infty$ if for every $y \in \X_\infty$ and every sequence $y_i\in \X_i$ such that $\sfd_Y(y_i,y)\to 0$ it holds that $\lim_i f_i(y_i)=f_\infty(y)$. 
	
	We say that $f_i$ converges uniformly to $f_\infty$ if for every $\eps>0$  there exists  $\delta>0$ such that $|f_i(y_i)-f_\infty(y)|<\eps$ for every $ i \ge \delta^{-1}$ and every $y_i$ such that $\sfd_Y(y_i,y)<\delta.$
\end{definition}
We point out that in the case of a fixed proper metric space the two notions of convergence in the above definition coincide respectively with the usual locally uniform and uniform convergence.

The following is a version of the Ascoli-Arzelà theorem for varying metric spaces (see also \cite[Prop. 27.20]{oldnew}). The proof can be achieved  arguing as in case of a fixed  (proper) metric  metric space and we will skip it.
\begin{prop}\label{prop:Ascoli}
	Let $f_i: \X_i \to \rr$ be equiLipschitz, equibounded functions with $\supp f_i \subset B_R(x_i)$, then there exists a subsequence  that converges uniformly to a Lipschitz function $f:\X_\infty \to \rr.$
\end{prop}

\begin{definition}[Weak/strong $L^2$-convergence]
	A sequence of functions $f_i \in L^2(\mea_i)$ converges weakly in $L^2$ to a function $f\in L^2(\mea_\infty)$ if $f_i\mea_i\rightharpoonup f\mea_\infty$ in duality with $C_{bs}(Y)$ and $\sup_i \|f_i\|_{L^2(\mea_i)}<+\infty.$
	
	A sequence of functions $f_i \in L^2(\mea_i)$ converges strongly in $L^2$ to a function $f\in L^2(\mea_\infty)$ if it converges weakly in $L^2$ to $f$ and $\lim_i \|f_i\|_{L^2(\mea_i)}=\|f\|_{L^2(\mea_\infty)}.$
\end{definition}

In the following proposition we collect some basic facts about strong and weak $L^2$ convergence.
\begin{prop}\label{prop:l2prop}
	\begin{enumerate}[label=\alph*)]
		\item If $f_i \in L^2(\mea_i) $ converge strongly in $L^2$ to $f \in L^2(\mea_\infty)$ and $f_i$ have uniformly bounded support, then $\phi\circ f_i$ converge strongly in $L^2$ to $\phi \circ f$, for every $\phi \in C(\rr)$ such that  $\phi(0)=0$ and $|\phi(t)|\le C(1+|t|)$ for some positive constant $C>0$.
		\item If $f_i,g_i \in L^2(\mea_i) $ are uniformly bounded in $L^\infty(\mea_i)$ and converge strongly in $L^2$  respectively to $f,g \in L^2(\mea_\infty)$, then $f_ig_i$ converge strongly in $L^2$ to $fg.$
		\item if $f_i : \X_i \to \rr$, with $\supp f_i \subset B^{\X_i}_{R}(x_i)$, converge uniformly to a bounded function $f: \X_\infty \to \rr$ then $f_i$  converge strongly in $L^2$ to $f$.
	\end{enumerate}
\end{prop}
\begin{proof}
	a) follows from the characterization of $L^2$-strong convergence as weak convergence of the graphs (see \cite[Sec. 5.2]{AGSflow}, \cite{GMSconv} and also \cite[Remark 5.2]{AHstab}). Indeed from \cite[(6.6)]{GMSconv} $L^2$-convergence is equivalent to weak convergence of $({\sf id} \times f_i)_*\mea_i$ to $({\sf id} \times f)_*\mea_\infty$ in duality with $\zeta \in C(Y\times \rr)$ satisfying $|\zeta(y,t)|\le \psi(y)+C|t|^2$. Then the claim follows observing that: under our assumptions (since $\phi(0)=0$), testing convergence against such $\zeta$'s is equivalent to test against $\eta\zeta$, where $\eta \in C_{bs}(Y)$ is such that $\eta\equiv 1$ in the supports of $f_i,f_\infty$ and moreover  $|\eta(y)\zeta(y,\phi(t))|\le |\eta(y)|(\psi(y)+C(1+|t|)^2)$ for every $\phi \in C(\rr)$ as in the hypotheses.
	
	The proof of b) can be found in \cite[Proposition 3.3]{AHstab}. 
	
To prove c) we first observe that from the properness of $Y$ it follows that the functions $f_i$ are equibounded, moreover they have uniform bounded support by hypothesis. Therefore we can apply the generalized version of Fatou lemma for varying measure (see  e.g. \cite[Lemma 2.5]{DGnonc}), first to $f_i \phi $ and then to $-f_i\phi$ to obtain that $\int f_i\phi \,\d\mea_i\to \int f \phi\, \d \mea,$ where $\phi$ is an arbitrary function in $C_{bs}(Y).$ Applying the same lemma also the functions $f_i^2,-f_i^2$ we deduce  that $\int f_i^2\, \d \mea_i\to \int f^2\,\d \mea$, concluding the proof.
\end{proof}

We will also need the following result about stability of Laplacian and gradient with respect to strong $L^2$ convergence.  Here $\Delta_i$ (resp. $\Delta_\infty$)  represents the Laplacian operator in $\X_i$ (resp. $\X_\infty$) and  $\nabla_i$ (resp. $\nabla_\infty$) represents the gradient operator in $\X_i$ (resp. $\X_\infty$).
\begin{theorem}[{\cite[Theorem 2.7, Theorem 2.8]{AHlocal}}]\label{thm:l2stab}
	Let $f_i \in D(\Delta_i)$ be such that
	\[
	\sup_i \|f_i\|_{L^2(\mea_i)}+\|\Delta_i f_i\|_{L^2(\mea_i)}<+\infty
	\]
	and assume that $f_i$ converge strongly in $L^2$ to $f$. Then $f \in D(\Delta_\infty)$, $\Delta_i f\to \Delta_\infty f$  weakly in $L^2$  and $|\nabla_i f_i|\to|\nabla_\infty f| $  strongly in $L^2$.
\end{theorem}
We conclude this subsection with the following technical result.
\begin{lemma}[{\cite[Lemma 5.8]{AHstab}}]\label{lem:lsc energy local}
	Let $f_i \in \W(\X_i)$ be such that $\sup_i \||\nabla f_i|\|_{L^2(\mea_i)}<+\infty$ and converging strongly in $L^2$ to $f \in \W(\X_\infty)$, then for any $A\subset Y$ open it holds
	\[
	\int_A |\nabla_\infty f|^2\, \d \mea_\infty \le \liminf_{i} \int_A |\nabla_i f |^2\, \d \mea_i.
	\]
\end{lemma}

\subsection{Regular Lagrangian flows}
In the work of Ambrosio and Trevisan (\cite{AT}) it was extended the theory of flows for Sobolev vector fields (recall \cite{diperna} and \cite{am}) to very general metric measure spaces and in particular for ${\rm RCD}(K,\infty)$ spaces. 

We restate here, using the language we introduced in the previous sections, their main results of existence and uniqueness.

\begin{definition}\label{def:rfl}
	Let $v_t:[0,T]\to L^2(T\X)$ be Borel, we say that a map $F: [0,T]\times X \to X$ is a Regular Lagrangian Flow associated to $v_t$ if the following are satisfied:
	\begin{enumerate}
		\item There exists $C>0$ such that 
		\begin{equation}\label{eq:compr}
			{F_t}_*\mea \le C \mea, \quad \text{ for every } t \in[0,T], 
		\end{equation}	
		\item for $\mea$-a.e. $x\in X$ the function $[0,T]\ni t \mapsto F_t(x)$ is continuous and satisfies $F_0(x)=x.$
		\item for every $f \in \test(X)$ it holds that for $\mea$-a.e. $x\in X$ the function $(0,T)\ni t \mapsto f\circ F_t(x)$ is absolutely continuous and 
		\begin{equation}\label{RFL}
			\frac{\d}{\d t}  f\circ F_t(x)=\langle \nabla f, v_t\rangle\circ F_t(x), \quad \text{ for a.e. } t \in(0,T).
		\end{equation}
	\end{enumerate}
\end{definition}
Notice that in  \eqref{RFL} we are implicitly choosing for every $t \in(0,T)$ a Borel representative of $\langle \nabla f,v_t\rangle$, however \eqref{eq:compr}  ensures that the validity of item 3 in Definition \ref{def:rfl} is independent of this choice.

Observe also that in Definition \ref{def:rfl} we are assuming that the map $F$ is pointwise defined, however the definition is stable under modification in a negligible set of trajectories in the following sense. If $F_t(x)$ is a Regular Lagrangian Flow for $v_t$ (as in Definition \ref{def:rfl}) and for $\mea$-a.e. $x$, $\tilde F_t(x)=F_t(x)$  holds for every $t\in [0,T]$, for some map $\tilde F : [0,T]\times X \to X$ then $\tilde F$ is also a Regular Lagrangian Flow for $v_t$. In any case to avoid technical issues, in our discussion we prefer to fix a pointwise defined representative for the flow map $F$. 

\begin{remark}\label{rmk:accurverfl}
	If $F_t$ is a regular Lagrangian flow for a vector field $v_t$, then for $\mea$-a.e. $x$ it holds that the curve $[0,T] \ni t \mapsto F_t(x) $ is absolutely continuous and its metric speed is given by
	\begin{equation}\label{eq:msflow}
		|\dot {F_t(x)}|=|v_t|\circ F_t(x), \quad \text{ a.e. } t\in [0,T].
	\end{equation}
	This follows from \cite[Lemma 7.4 and 9.2]{AT}. Observe that this statement is independent of the chosen representative of $|v_t|$, thanks to \eqref{eq:compr}.\fr
\end{remark}

We will see  below in Theorem  \ref{thm:uniqrfl} that  the existence and uniqueness of a Regular Lagrangian Flow is  linked to the existence and uniqueness of a solution to the continuity equation (\cite{GHcont}): 

\begin{definition}\label{def:conteq}
	Let $v_t:[0,T]\to L^2(T\X)$ be Borel and $t \mapsto \mu_t \in \prob{\X}$, $t \in [0,T]$ be also a Borel  map. Suppose also  that $\||v_t|\|_{L^2(\mea)}\in L^{1}(0,T)$ and $\mu_t \le C \mea $ for every $t \in[0,T]$ and some positive constant $C$. We say that  $\mu_t$ with is a weak solution of the continuity equation
	\[ \frac{\d}{\d t}\mu_t+\div (v_t \mu_t)=0,\]
	with initial datum $\mu_0,$ if for every $f\in \Lip_{bs}(\X)$ the function $[0,T]\ni t \mapsto \int f \,\d \mu_t  $ is absolutely continuous and
	\begin{equation}\label{eq:conteq}
		\frac{\d}{\d t}\int f\, \d \mu_t = \int \la \nabla f,v_t\ra  \, \d \mu_t\, \quad \quad \text{ for } a.e.\,t\in (0,T).
	\end{equation}
\end{definition}

We refer to \cite[Sec. 3.4]{G14} for the definition of the space $\W_{C}(T\X)$  and the object $\nabla_{sym}v\in L^2(T^{2\otimes}\X)$. For our purposes it sufficient to know that for any $f \in \test(\X)$ we have  $\nabla f \in \W_C(T\X),$ with  $\nabla_{sym}\nabla f=\nabla(\nabla f)= \H f^{\sharp}$ (see \cite[Sec. 3.2]{G14}).

\begin{theorem}[\cite{AT}]\label{thm:uniqrfl}
	Let $v_t:[0,T]\to L^2(T\X)$ be Borel and  such that $v_t \in D(\div)$  for every $t\in[0,T]$. Assume furthermore that $\||v_t|\|_{L^2(\mea)}\in L^{1}(0,T)$, $\|\div(v_t)^-\|_{L^\infty}\in L^\infty(0,T)$ and $\|\nabla_{sym} v_t\|_{L^2(T^{\otimes2}X)}\in L^1(0,T)$. Then 
	\begin{enumerate}
		\item there exists a unique Regular Lagrangian flow  $F_t$ associated to $v_t$,
		\item for every initial datum $\mu_0 \in \prob{\X}$ with $\mu_0\le C \mea$ there exists a unique weak solution  $\mu_t$ to the continuity equation and it is given given by $\mu_t\coloneqq {F_t}_*{\mu_0}.$
	\end{enumerate}
\end{theorem}

We remark that the uniqueness of the Regular Lagrangian Flow in Theorem \ref{thm:uniqrfl} has to be intended up to modification in a negligible set of trajectories, as discussed above.

\begin{remark}\label{rmk:extend}
	Let $v_t$ be as in Theorem \ref{thm:uniqrfl} and autonomous, i.e. $v_t\equiv v $ for some $v \in D(\div)\cap W^{1,2}_{C}(TX)$ with $\div(v)^{-}\in L^\infty(\mea).$ Then, thanks to the uniqueness given by Theorem \ref{thm:uniqrfl}, the Lagrangian flow $F_t$ relative to $v$ can be extended uniquely (up to a set of negligible trajectories) to a map $F: [0,\infty)\times X \to X$  which  satisfies the following group property
	\begin{equation}\label{eq:group}
		F_{s}\circ F_t= F_{s+t}, \quad \mea \text{-a.e}.
	\end{equation}
	for every $s,t \in [0,\infty).$
	
	Moreover if also $\div(v)\in L^\infty(\mea)$, it can be shown (see for example \cite[Lemma 3.18]{GRtorus}) that, denoting by $F^{-v}_{t}$ the Lagrangian flow relative to $-v$ (which exists unique for all times $t\ge 0$, thanks to Theorem \ref{thm:uniqrfl} and the previous observation) 
	\begin{equation}\label{eq:invflow}
		F^{-v}_t\circ F_t = {\sf id}, \quad \mea \text{-a.e},
	\end{equation}
	for every $t\ge 0.$ Hence setting $F_{-t}\coloneqq F^{-v}_t$ we can extend $F$ to $F:(-\infty,+\infty)\times X\to X$, for which \eqref{eq:group} is satisfied for every $s,t \in \rr.$\fr
\end{remark}

\subsubsection{Functions of bounded variation}

We recall the definition of function of bounded variation on a metric measure space. For a detailed treatment of this topic see for example \cite{miranda} and \cite{dimarino}. 
\begin{definition}[Functions of bounded variation]
	We say that function $f \in L^1(\mea)$ belongs to the space ${\sf BV}(\X)$ of functions of  bounded variation if there exists a sequence $f_n\in \Lip_{\loc}(\X)$ such that
	\[
	\limsup_{n \to +\infty } \int \lip f_n \,\d \mea<+\infty,
	\]
	(where $\lip f_n$ was defined in \eqref{eq:local lip}). By localizing this construction we also define 
	\[
	\|Df\|(A)\coloneqq \inf \left \{ \liminf_{n} \int_A \lip f_n \,\d \mea \ : \ f_n\in \Lip_{\loc}(A), \, f_n \to f \text{ in } L^1(A)  \right \},
	\]
	for any  $A\subset \X$ open.
	It is proven in \cite{miranda} (at least for doubling m.m.s.) that this set function  is the restriction to open sets of a finite and positive Borel measure  on $\X$ that we call \emph{total variation} of $f$ and still denote by $\|Df\|.$
\end{definition}
It is proven in \cite[Remark 3.5]{GH}, in the case of proper ${\rm RCD(K,\infty)}$ spaces (and thus also for any ${\rm RCD}(K,N)$, with $N<+\infty$), the equivalence between  the total variation and the weak upper gradient, meaning that if $f \in {\sf BV}(\X)\cap \Lip_{\loc}(\X)$, then  
\begin{equation}\label{eq:totv}
	\|Df\|=|\nabla f|\mea .
\end{equation}

\begin{definition}[Sets of finite perimeter]
	Given a Borel set $E\subset \X$ and any open set $A\subset \X$ we define the perimeter $\Per(A,E)$ as
	\[
	\Per(A,E)\coloneqq \inf \left \{ \liminf_{n} \int_A \lip f_n \d \,\mea \ : \ f_n \in \Lip_{\loc}(A), \, f_n \to \nchi_E \text{ in } L^1_{\loc}(A) \right \}.
	\]
	We say that $E$ has finite perimeter if $\Per(E,\X)<+\infty$. In this case it can be proved that the set function $\Per(E,A)$ is the restriction to open sets of a finite and positive Borel measure on $\X$ that we still denote by $\Per(E,.)$.
\end{definition}

We will need the following variant of the coarea formula, which follows from \eqref{eq:totv}, the standard coarea formula for m.m.s. (see \cite[Remark 4.3]{miranda}) and a simple truncation argument.
\begin{prop}[Coarea formula]\label{prop:coarea}
	Let $\X$ be an ${\rm RCD}(K,N)$ m.m.s with $N<+\infty$ and let $\Omega\subset \X$ open. Let  $u \in \LIP_{\loc}(\Omega)$  be positive and such that $u^{-1}([a,b])$ is  compact in $\Omega$, for every $[a,b] \subset  (0,1)$. Then $\{u<t\}$   has finite perimeter for a.e. $t \in (0,1)$ and  for any  $f : \Omega \to [-\infty,+\infty]$ Borel and in $L^1_{\loc}(\Omega,|\nabla u|\mea \restr{\Omega})$ it holds that
	\begin{equation}\label{eq:coarea}
		\int_{\Omega} \phi(u) f\, |\nabla u| \,\d\mea = \int_{\Omega} \phi(t) \int f \,\d\Per(\{u<t\},.)\, \d t,\quad \forall \phi:[0,1] \to \rr \text{ Borel, with } \supp \phi \subset(0,1).
	\end{equation}
\end{prop}

\section{Nonparabolic RCD spaces}\label{sec:nonpara}
In this section we introduce the notion of nonparabolic ${\rm RCD}(0,N)$ space: this is the natural setting to study problem \eqref{mainpde}, indeed already for smooth manifolds the existence of a solution to \eqref{mainpde} implies that the manifold is nonparabolic  (see for example Theorem 2.3 in \cite{AFM}).

We recall that a (non-compact) Riemannian manifold is said to be nonparabolic if it admits a positive global Green function. It has been proved by Varopouls (\cite{var}) that in the case ${\sf Ric}\ge 0$ the nonparabolicity is equivalent to \eqref{eq:nonparabolic}. This motivates the following.
\begin{definition}[Nonparabolic RCD space]\label{def:nonparabolic}
	Let $(\X,\sfd,\mea)$ be an ${\rm RCD}(0,N)$ space with $N<+\infty.$ We say that $\X$ is \emph{nonparabolic} if
	\begin{equation}\label{eq:nonparabolic}
		\int_1^{+\infty} \frac{s}{\mea(B_s(x))}\, \d s<+\infty, \quad \text{ for every $x \in X.$}
	\end{equation}
	Observe that the above quantity is finite for one $x\in X$ if and only if it is finite for all $x\in X.$
\end{definition}
We point out that condition \eqref{eq:nonparabolic} in the context of ${\rm RCD}$ spaces was already introduced in \cite{constancy}.
	\begin{remark}\label{rmk:N>2}
		It follows immediately from the Bishop-Gromov inequality that if $\X$ is a non-parabolic ${\rm RCD}(0,N)$ space then it is non-compact and $N>2.$\fr
	\end{remark}
In the following two subsections we develop the two main features  of a nonparabolic ${\rm RCD}$ space that we will need in this note: the first is the existence of a Green function, which also provides an explicit solution to \eqref{mainpde}; the second is that  the number of ends is equal to one (see Definition \ref{def:ends}).

\subsection{The Green function}
It turns out that on a nonparabolic RCD space it can be given a notion of positive global Green function. Following \cite{constancy} we  define the Green function $G :\X\times \X \to[0,+\infty]$ as
\begin{equation}\label{eq:green def}
G(x,y)\coloneqq\int_0^\infty p_t(x,y)\, \d \mea.
\end{equation}
We also set $G_x(y)\coloneqq G(x,y).$ For any $\eps>0$ we also define the quasi Green function $G^{\eps} :\X\times \X \to[0,+\infty]$ as
\begin{equation}\label{eq:quasi green def}
	G^{\eps}(x,y)\coloneqq\int_{\eps}^\infty p_t(x,y)\, \d \mea
\end{equation}
and as above we set $G^{\eps}_x(y)\coloneqq G^{\eps}(x,y).$ 
It is prove in \cite[Lemma 2.5]{constancy} that $G^{\eps}_x\in \LIP(\X)\cap D(\bd)$ with $\bd G_x^{\eps}=-p_{\eps}(x,y)\mea $, in particular $G_x^{\eps}$ is superharmonic in the whole $\X$.
\begin{prop}[Estimates for the Green functions, {\cite[Prop. 2.3]{constancy}}, see also \cite{grigor}]
		Let $(\X,\sfd,\mea)$ be a nonparabolic ${\rm RCD}(0,N)$ m.m.s. Then there exits a 	constant $C=C(N)>1$ such that
	\begin{equation}\label{eq:greenestimates}
	\frac{1}{C}\int_{\sfd(x,y)}^\infty \frac{s}{\mea(B_s(x))} \,\d s  \le	G(x,y)\le C\int_{\sfd(x,y)}^\infty \frac{s}{\mea(B_s(x))} \,\d s, \quad \forall x,y \in \X. 
	\end{equation} 
\end{prop}

	\begin{prop}\label{prop:green harmonic}
		Let $\X$ be a nonparabolic ${\rm RCD}(0,N)$. Then $G_x$ is positive, continuous and harmonic in $\X\setminus\{x\},$ for every $x\in \X.$	
	\end{prop}
	\begin{proof}
		Fix $R,\delta>0$ with $R>\delta$ and define $A_{\delta,R}=B_R(x)\setminus \bar B_{\delta}(x).$ It enough to prove that $G_x$ is harmonic on $A_{\delta,R}$.
		Recall that $G^{\eps}_x\in \LIP(\X)\cap D(\bd)$ with $\bd G_x^{\eps}=-p_{\eps}(x,y)\mea $ and that $G_x^{\eps}\to G_x$ in $L^1_{\loc}(\X)$. 
		We now observe that from the heat kernel bounds \eqref{eq:kernelestimate} we have $\sup_{t\in(0,1)} \|p_t(x,.)+|\nabla p_t(x,.)|\|_{L^\infty(A_{\delta,R})}<+\infty $ and that $\sup_{\eps>0} \|G^{\eps}_x\|_{L^\infty(A_{\delta,R})} <+\infty.$ From this, following the arguments in the proof of \cite[Lemma 2.5]{constancy}, we can prove that the sequence $G_x^{\eps}$ is Cauchy in $\W(A_{\delta,R}).$ In particular $G_x^{\eps}\to G_x$ in $\W(A_{\delta,R})$ and, since $p_t(x,.)\to 0$ uniformly in $A_{\delta,R}$, we deduce that $G_x$ is harmonic in $A_{\delta,R}$ (cf. with Lemma \ref{lem:techconv}).
		
		We pass to the continuity of $G_x$ in $\X\setminus\{x\}$. We first observe that  from \eqref{eq:kernelestimate} we have
		$$\||\nabla p_t(x,.)|\|_{L^\infty(\X \setminus  B_{\delta}(x))}\le C(N)t^{-1/2} \mea(B_{\sqrt t}(x))^{-1}e^{\frac{-\delta^2}{5t}}\eqqcolon \beta(t,\delta)$$
		 and that, thanks to the Bishop-Gromov inequality and the nonparabolicity assumption, $\int_0^\infty \beta(t,\delta)\, \d t<+\infty,$ for every $\delta>0.$	 Therefore from the continuity of $p_t$ and Proposition \ref{prop:sob to lip} we deduce that
		$$\limsup_{z\to y}|G_x(z)-G_x(y)|\le \left(\int_0^\infty \beta(t,\delta)\, \d t\right)\limsup_{z\to y} \sfd(z,y)=0, \quad \forall y \in \X\setminus \bar B_{\delta}(x_0),$$
		from which the claimed continuity follows.
\end{proof}

As anticipated, the Green function provides a solution to \eqref{mainpde}, in particular we have the following:
\begin{cor}\label{cor:green sol}
		Let $\X$ be a nonparabolic ${\rm RCD}(0,N)$ and let $\Omega \subset \X$ be open, unbounded with $\partial \Omega$ bounded. Then for every $x_0 \in \Omega^c$ there exists $\lambda >0$ such that  $\lambda G_{x_0}$ is a solution to \eqref{mainpde}.
\end{cor}
\begin{proof}
	From Proposition \ref{prop:green harmonic} it follows that $G_{x_0}$ is harmonic and continuous in $\X \setminus \{x_0\}$, moreover from \eqref{eq:greenestimates}  we have that $G_{x_0}(x)\to 0$ as $\sfd(x,x_0)\to +\infty$. Therefore, since $\partial \Omega$ is bounded and $G_{x_0}$ is positive, we can simply take $\lambda=(\min_{\partial \Omega} G_{x_0})^{-1}.$
\end{proof}

\subsection{Number of ends}
Let us introduce the notion of \emph{ends} for a metric space. The definition is usually given for manifolds, however, since the definition is purely metric, it carries over verbatim to metric spaces.
\begin{definition}[Number of ends of a metric space]\label{def:ends}
	Let $(X,\sfd)$ be a metric space and $k \in \mathbb{N}$. We say that $\X$ has $k$ ends if both the following are true:
	\begin{enumerate}
		\item  for any $K$ compact, $\X \setminus K$ has at most $k$ unbounded connected components,
		\item there exists  $K'$ compact such that $\X \setminus K'$ has  exactly $k$ unbounded connected components.
	\end{enumerate}
\end{definition}

The following result generalizes to the non smooth setting a well known  result for Riemannian manifolds.
\begin{prop}\label{prop:ends}
	Suppose $(\X,\sfd,\mea)$ is a non compact ${\rm RCD}(0,N)$ space, $N\in[1,\infty)$. Then exactly one of the following holds:
	\begin{enumerate}[label=(\roman*)]
		\item $\X$ is a cylinder, meaning that $X$ is isomorphic to the product $(\rr\times \X',\sfd_{Eucl}\times \sfd',\mathcal{L}^1\otimes \mea')$, where $(X',\sfd',\mea')$ is a compact ${\rm RCD}(0,N-1)$ m.m.s. if $N\ge 2$  and a single point if $N\in [1,2)$,
		\item for every $C$ bounded subset of $X$, there exists $R>0$  such that  the following holds: for every couple of points $x,y \in X$ satisfying $\sfd(x,C),\sfd(y,C)>R$  there exists $\gamma \in \Lip([0,1],\X)$ such that  $\gamma(0)=x,\, \gamma(1)=y$, $\gamma \subset X\setminus C$ and $\Lip (\gamma)\le 5\sfd(x,y).$
	In particular $\X$ has one end.
	\end{enumerate}
\end{prop}
\begin{proof}
	We closely follow \cite[Prop. 2.10]{AFM}.
	Suppose that $ii)$ does not hold, it follows that there exists a bounded set $C$, two sequences of points $(x_k),(y_k) \subset X$ and geodesics $(\gamma_k)$ between $x_k$ and $y_k$ such that $\sfd(x_k,C),\sfd(y_k,C)\to+\infty$ and   $\gamma_k$ intersects $C$ for all $k$. Since $\X$ is proper and $C$ is compact, with a compactness argument (assuming all the $\gamma_k$ parametrized by arc lenght)  we deduce that $\X$ contains a line. In particular, begin $\X $ an ${\rm RCD}(0,N)$ space, by the splitting theorem \cite{split} we infer that $\X$ is is isomorphic to the product $(\rr\times \X',\sfd_{Eucl}\times \sfd',\mathcal{L}^1\otimes \mea')$, where $(X',\sfd',\mea')$ is an ${\rm RCD}(0,N-1)$ space if $N\ge 2$  and a single point if $N\in [1,2)$. It remains to prove that $\X'$ is bounded.

	Suppose it is not. We claim that this would imply the validity of $ii)$ and thus a contradiction. Indeed suppose that $C\subset X$ is bounded, then  $C\subset B_R(p)$  for some $R>0$ and $p\in\X$.   It is enough to show that for every couple of points $x_0,x_1 \in B_{2R}(p)^c$ there exists $\gamma \in \Lip([0,1],\X)$  joining them and with image contained in  $B_R(p)^c.$ In the case $\sfd(x_0,x_1)\le R$ we conclude immediately by taking a geodesic between $x_0$ and $x_1$, hence we can suppose  that $\sfd(x_0,x_1)>R$. Identifying $\X$ with $\rr\times \X'$ we have that $p=(\bar t,x'),$ $x_i=( t_i,x'_i)$ $i=0,1,$ for some $\bar t,t_0,t_1 \in \rr$ and $x', x_0', x_1'\in \X'$. Hence $I \times B'\subset B_{2R}(p)$, where $I\coloneqq[\bar t-R,\bar t+R]$ and $B'\coloneqq B_R(x').$ In particular $x_i \in (I \times B')^c$, $i=0,1$, i.e. for every $i=0,1$ either $t_i \notin I$ or $x_i \notin B'$. We claim that it is  sufficient to deal with the case $t_0\in \rr \setminus I$ and $x_1'\in\X'\setminus B',$ indeed the other cases follow from this one by concatenating two paths of this type as follows: if $t_0,t_1\notin I$ and $x_0',x_1'\in B'$ we choose $y' \notin B'$ (which exists since $\X'$ is unbounded) and we join $(t_0,x_0')$ to $(t_0,y')$ and then  $(t_0,y')$ to $(t_1,x_1')$; if $t_0,t_1\in I$ and $x_0',x_1' \notin B'$ we pick $s \notin I$ and we  join $(t_0,x_0')$ to $(s,x_0')$ and then  $(s,x_0')$ to $(t_1,x_1')$. 
	
	Hence we can assume that $t_0\in \rr \setminus I$ and $x_1'\in\X'\setminus B'.$ To build the required path, consider a geodesic $\eta :[0,1]\to X'$ going from $x_0'$ to $x_1'$ and define the function $s:[0,1]\to \rr$ as $s(t)=t_1\,t+(1-t)\,t_0$. Then the curve 
	\[\gamma(t)=\begin{cases}
		(t_0,\eta(t)), & t \in [0,1),\\
		(s(t),x_1'),  & t \in [1,2],
	\end{cases}\]
	is Lipschitz and has image contained in $(I \times B')^c\subset B_R(p)^c$, hence (up to a reparametrization) satisfies all the requirements.
	
	To estimate $\Lip (\gamma)$ we observe that, up to a reparametrization, we can assume that $\Lip (\gamma )=L(\gamma)$, hence it is sufficient to bound the length of $\gamma$.  In the case $t_0\in \rr \setminus I$ and $x_1'\in\X'\setminus B'$ it is sufficient to  observe that $L(\gamma)=\sfd'(x_0',x_1')+|t_0-t_1|\le 2( \sfd_{Eucl}\times\sfd((t_0,x_0'),(t_1,x_1')))$, where  $\gamma$ is the curve constructed above. The general case follows concatenating  two paths as described above, where we pick $y'$ and $s$ so  that $\sfd'(y',x')<2R$, $|\bar t-s|<2R $. Indeed it can be easily checked these  two  resulting paths have respectively length not grater  than $2R$ and $2R+\sfd(x_0,x_1).$ Since we are assuming $\sfd(x_0,x_1)\ge R$, this  concludes the proof.
\end{proof}
\begin{cor}\label{cor:ends}
		If $\X$ is a nonparabolic ${\rm RCD}(0,N)$ space, then it is not a cylinder and in particular item ii) of Proposition \ref{prop:ends} holds and $\X$ has only one end.
\end{cor}
\begin{proof}
	Suppose by contradiction  $\X$ is a cylinder $\rr \times \X'$. Then  for any $r>0$ and any $(x,t) \in \X'\times \rr$ we have
	$$\mea(B_r((x,t)))=(\mathcal{L}^1\otimes\mea' )(B_r((x,t)))\le (\mathcal{L}^1\otimes\mea') (\X'\times [t-r,t+r] )=\mea'(\X')\, 2r,$$
	which clearly contradicts the fact that $\X$ is nonparabolic.
\end{proof}

\section{New estimates for harmonic functions}\label{sec:positive div}

\subsection{Preliminary calculus rules}
	In this subsection we collect and prove some versions of the chain and Leibniz rule for the Laplacian and the divergence operator, that will be used in the following  subsection. These results are essentially variants of  the ones already contained in \cite{Gappl} and \cite{G14}. However, since one of the main obstacles in the proof of Theorem \ref{thm:harm estimates} and Corollary \ref{cor:div estimate} is making the computations rigorous  and justifying the derivatives taken, we decided to state and prove in details the results we need.

In everything that follows $\Omega$ is an open subset of a proper and infinitesimally Hilbertian m.m.s. $(\X,\sfd,\mea).$
\begin{prop}[Leibniz rule for $\bd$]\label{prop:leiblapl}
	Let $u \in D(\bd,\Omega)$ and let $g\in \LIP_\loc(\Omega)\cap D(\bd,\Omega)$ be such that $\bd g \in L^2(\Omega).$ Then $ug \in D(\bd,\Omega)$ and 
	\begin{equation}\label{leiblap}
		\bd(ug)=g\bd u+u \bd g+2\langle \nabla u, \nabla g \rangle \mea . 
	\end{equation}
\end{prop}
\begin{proof}
	Let $f \in \LIP_c(\Omega).$ Then using the Leibniz rule for the gradient
	\[\int \langle \nabla(ug), \nabla f\rangle \, \d\mea =\int \langle \nabla u, \nabla (fg)\rangle \, \d\mea + \int  \langle \nabla g, \nabla (uf)\rangle \, \d\mea - 2\int f \langle \nabla u, \nabla g \rangle \, \d\mea.    \]
	Since $fg\in \LIP_c(\Omega)$ and $fu \in \W(X)$ with support compact in $\Omega$, we conclude  from Remark \ref{extension1}.
\end{proof}

\begin{prop}[Leibniz rule for ${\bf{div}}$]\label{prop:leibdiv}
	Let $v \in D({\bf{div}},\Omega)$. Then
	\begin{enumerate}
		\item if $g \in \LIP_\loc(\Omega)$, then $gv \in D({\bf{div}},\Omega) $   and
		\begin{equation}\label{leibdiv}
			{\bf{div}}(gv)=\la\nabla g,v\ra\mea + g {\bf{div}}(v) . 
		\end{equation}
		\item if $g \in \W_\loc(\Omega)$ and ${\bf{div}}(v) \in L^2_{\loc }(\Omega)$, then $gv \in D({\bf{div}},\Omega) $  and \eqref{leibdiv} holds.
	\end{enumerate}
\end{prop}
\begin{proof}
	Let $f \in \LIP_\c(\Omega).$ Using the Leibniz  rule for the gradient we get
	\[
	-\int \la \nabla f, gv \ra  \d \mea = -\int\la \nabla (fg),v\ra  \d \mea +\int f  \la \nabla g, v \ra  \d \mea.
	\]
	The conclusion follows in the first case observing that $fg\in \LIP_c(\Omega)$ and in the second case observing that $fg \in \W(\X)$ with compact support in $\Omega$ and recalling Remark \ref{extensiondiv}.
\end{proof}

\begin{prop}[Chain rule for $\bd$]\label{prop:chainlapl}
	Let $u \in D(\bd,\Omega)$ and let $\varphi \in C^2(I)$, where $I$ is an open interval such that \eqref{eq:chain rule hyp} holds. Then 
	\begin{enumerate}
		\item if $u \in \LIP_{\loc}(\Omega)$ then $\phi(u)\in  D(\bd,\Omega)$ and
		\begin{equation}\label{chainlap}
			\bd\restr{\Omega}(\varphi(u))=\phi'(u)\bd\restr{\Omega} u+\phi''(u)|\nabla u|^2\mea\restr{\Omega}. 
		\end{equation}
		\item if $\bd u\ge g \mea\restr{\Omega} $ for some  $g \in L^1_{\loc}(\Omega)$ then $\phi(u)\in  D(\bd,\Omega)$ and
		\begin{equation}\label{chainlapv2}
			\bd\restr{\Omega}(\varphi(u))\ge  \left (\phi'(u)g +\phi''(u)|\nabla u|^2\right )\mea\restr{\Omega}. 
		\end{equation}
	\end{enumerate}	
\end{prop}
\begin{proof}
	Let $f \in \LIP_c(\Omega),$ then from the chain rule and Leibniz rule for the gradient
	\[ -\int \langle \nabla (\varphi(u)), \nabla f \rangle \d \mea = -\int \langle \nabla (\varphi'(u)f), \nabla u \rangle \d \mea + \int f\la \nabla \phi'(u), \nabla u	\ra\d \mea. \]
	In the first case we conclude from the fact that $\phi'(u)f \in \Lip_c(\Omega)$. In the second case we assume also $f \ge 0$ and
	observe that $\phi'(u)f \in \W(X)\cap L^\infty(\mea)$ is nonnegative with compact support in $\Omega$, hence from Remark \ref{extension1} it follows that
	\[ -\int \langle \nabla (\varphi(u)), \nabla f \rangle \d \mea \ge  \int \phi'(u)f \, g\mea + \int f \phi''(u)|\nabla u|^2\d \mea. \]
	The conclusion follows applying Proposition \ref{prop:laplineq}.
\end{proof}

\subsection{Second order  estimates for harmonic functions}

Before starting the discussion, let us say that this part is independent of all the rest of the note and we only assume that
\[
\text{$\X$ is an ${\rm RCD}(K,N)$ with $N \in [2,+\infty)$ and $\Omega$ an open subset of $\X$.}
\]
The goal of this subsection is to prove the following two results, which can be reinterpreted as a generalization of the well know fact that in a Riemannian manifold with nonnegative Ricci curvature the square norm of the gradient of an harmonic function is subharmonic.

\begin{theorem}\label{thm:harm estimates}
Let $\X$ be an ${\rm RCD}(K,N)$ space with $N \in [2,\infty)$, let $u$ be harmonic in $\Omega$ and   $ \beta >\frac{N-2}{N-1}$. Then $|\nabla u|^{\beta/2}\in\W_{\loc}(\Omega)$, $|\nabla u|^\beta \in D(\bd,\Omega)$ and
	\begin{equation}\label{eq:harmonic estimates}
		\bd(|\nabla u|^\beta)\ge C_{\beta,N}\, |\nabla |\nabla u|^{\frac{\beta}{2}}|^2\, \mea\restr{\Omega}+2\beta K |\nabla u|^2\mea \restr\Omega,
	\end{equation}
	where $C_{\beta,N}=\frac{4}{\beta} \left (\beta-\frac{N-2}{N-1}\right )$. Moreover  $|\nabla u|^\beta \in D(\bd,\Omega)$ also for $\beta=\frac{N-2}{N-1}$ with \eqref{eq:harmonic estimates} holding without the term containing $C_{\beta,N}.$
\end{theorem}
\begin{cor}\label{cor:div estimate}
	Let $\X$ be an ${\rm RCD}(K,N)$ space  with $N\in(2,\infty)$. Suppose $u$ is positive and harmonic in $\Omega$, set $v=u^{\frac{-1}{N-2}}$ and let $\beta > \frac{N-2}{N-1}$. Then $|\nabla v|^{\beta/2}\in \W_{\loc}(\Omega)$, $u^2\nabla|\nabla v|^{\beta/2} \in D({\bf{div}},\Omega)$ and 
	\begin{equation}\label{eq:div estimate}
		{\bf div}(u^2\nabla|\nabla v|^{\beta})\ge   C_{\beta,N}\, u^2 |\nabla |\nabla v|^{\frac{\beta}{2}}|^2\, \mea\restr{\Omega}+2\beta K u^2|\nabla v|^\beta \, \mea \restr{\Omega},
	\end{equation}
	where $ C_{\beta,N}=\frac{4}{\beta} \left (\beta-\frac{N-2}{N-1}\right )$.  Moreover $u^2\nabla|\nabla v|^{\beta/2} \in D({\bf{div}},\Omega)$ also for $\beta=\frac{N-2}{N-1}$ with \eqref{eq:div estimate} holding without the term containing $C_{\beta,N}.$
\end{cor}

Let us remark that the results on the rest of this note rely only on Corollary \ref{cor:div estimate}. However we decided to isolate Theorem \ref{thm:harm estimates}, since it contains estimates for general harmonic functions that appear to be useful and interesting on their own.

The main ingredient for the proof of Theorem \ref{thm:harm estimates} is the following new Kato-type inequality. Observe that if $\dim(\X)=N$, in which case (from \cite{HanRicci}) ${\sf tr}\H u =\Delta u$, letting $t\to 0$ we recover the standard refined  Kato-inequality for harmonic functions.
\begin{lemma}[Generalized refined Kato inequality]\label{lem:kato}
Let $\X$ be an ${\rm RCD}(K,N)$space, with $N\in[1,+\infty)$, and let $n=\dim(\X).$ Then for any 	$u \in \test_{\loc} (\Omega)$ (see Sec. \ref{sec:local bochner}) it holds
	\begin{equation}\label{eq:Kato}
		\frac{t+n }{t+n-1}|\nabla |\nabla u||^2\le  |\H u|^2_{HS}+\frac{({\sf tr}\H u)^2}{t}, \quad \text{$\mea$-a.e. in $\Omega$,\, }\forall t>0.
	\end{equation}

\end{lemma}
\begin{proof}
	Observe that it is enough to prove \eqref{eq:Kato} for $u \in \test(\X)$.
	
	We make the following preliminary observation. Let $A$ be any symmetric $n\times n$ real matrix, then the following inequality holds for every $t>0$
	\begin{equation}\label{eq:katomatrix}
			\frac{t+n }{t+n-1}|A\cdot v|^2\le \frac{|v|^2({\sf tr} A)^2}{t} + |v|^2|A|^2, \quad \forall v \in \rr^n, \quad \forall N\ge n.
	\end{equation}
	where $|A|$ is the Hilbert-Schmidt norm of $A$. To prove it we can assume that $A$ is diagonal with diagonal entries $\lambda_1,...,\lambda_n$, where $\lambda_n\ge \lambda_i$, $i\le n$ and also that $|v|\le 1.$   Applying twice  Cauchy-Schwartz we obtain
	\begin{align*}
		\frac{(\lambda_1+...+\lambda_n)^2}{t}+\lambda_1^2+...+\lambda_{n-1}^2+\lambda_n^2&\ge \frac{(\lambda_1+...+\lambda_n)^2}{t} +\frac{(\lambda_1+...+\lambda_{n-1})^2}{n-1}+\lambda_n^2\\
		&\ge \frac{\lambda_n^2}{t+n-1}+\lambda_n^2\ge\frac{t+n }{t+n-1} |A\cdot v|^2,
	\end{align*}
	which proves \eqref{eq:katomatrix}.

	Let $e_1,...,e_n \in L^0(T \X)$ be a global orthonormal base which exists thanks to Theorem \ref{thm:constancy}. Consider the $n\times n$ real matrix $A : \X \to L^0(\mea)^{n^2}$ defined by $\{A_{i,j}(x)\}_{i,j}=\{\H{u}(e_i,e_j)(x)\}_{i,j}$. Define also the vector $v: \X \to L^0(\mea)^n$ as $v_i(x)\coloneqq \la \nabla u, e_i\ra(x) $. From \eqref{eq:hessian gradient} and the formula \eqref{eq:coordmap} we have that
	\[
	2|\nabla u|\nabla |\nabla u|=\nabla |\nabla u|^2=2\sum_{i=1}^n\sum_{j=1}^n A_{i,j}v_j e_i, \quad \text{ in $L^0(T\X)$}.
	\]
	Taking the square pointwise norm on both sides we obtain
	\[
	4|\nabla u|^2 |\nabla |\nabla u||^2= 4 |A\cdot v|^2, \quad \mea\text{-a.e..}
	\]
	Since from \eqref{eq:coordHS} and \eqref{eq:coordtr} we have that $|\H u|_{HS}^2=|A|_{HS}^2$ and ${\sf tr}\H u={\sf tr}A$, the conclusion follows combining the above identity with \eqref{eq:katomatrix}.
\end{proof}

The second ingredient for the proof  is the following simple technical lemma (cf. with \cite[Prop. 4.15]{Gappl})

\begin{lemma}\label{lem:techconv}
	Let $\X$ be an ${\rm RCD}(K,N)$space, with $N<+\infty$. Let $(u_n)\subset D(\bd,\Omega)$ and $u \in \W_{\loc}(\Omega)$ be such that $|\nabla u_n-\nabla u|^2\to 0$ in $L^1_{\loc}(\Omega)$. Moreover assume that $\bd u_n \ge g_n \mea$ for some $g_n \in L^1_{\loc}(\Omega)$ such that $\int g_nf \,\ \d \mea \to \int gf\, \d \mea+ \int hf \, \d\mea $, for every $f \in \LIP_{c}(\Omega)$ with $f \ge 0$, for some fixed functions $g\in L^0(\Omega,\mea)$ and $h \in L^1_{\loc}(\Omega)$, with $g \ge 0$ $\mea$-a.e.. Then $g \in L^1_{\loc}(\Omega)$, $u \in D(\bd,\Omega)$ and $\bd u \ge (g+h) \mea.$
\end{lemma}
\begin{proof}
	The assumptions guarantee that $\int \la \nabla u_n,\nabla f\ra \d \mea\to \int \la \nabla u,\nabla f \ra  \d \mea$, for every $f \in \LIP_{c}(\Omega)$, therefore we can pass to the limit on both sides of $-\int \la \nabla u_n,\nabla f\ra \d \mea \ge \int g_nf\, \d \mea$ to obtain that
	\[
	-\int \la \nabla u,\nabla f\ra \d \mea \ge \int gf\, \d \mea+\int hf\, \d \mea, \quad \forall f \in \LIP_{c}(\Omega), \text{ with } f \ge 0.
	\]	
	From this it follows that $g \in L^1_{\loc}(\Omega)$, indeed we can take for any  $K$ compact in $\Omega$ a function $f \in \LIP_c(\Omega)$ such that $f \ge 0$ and $f=1$ in $K$ and then bring $\int hf\, \d \mea$ to the other side of the inequality.
	The conclusion then follows applying Proposition \ref{prop:laplineq}.
\end{proof}

\begin{proof}[Proof of Theorem \ref{thm:harm estimates}]
The proof is based on a inductive bootstrap argument.
We make the following claim
\[
\text{if $\beta > \frac{N-2}{N-1}$ is such that $|\nabla u|^{\beta}\in\W_{\loc}(\Omega),$ then $|\nabla u|^{\beta/2}\in\W_{\loc}(\Omega)$ and \eqref{eq:harmonic estimates} holds.}
\]
Observe that, since we already know that $|\nabla u|^{\beta}\in \W_{\loc}(\Omega)$ for every $\beta \ge 1$ (recall \eqref{eq:sobolev gradient}), the first part of the conclusion follows iterating the above statement.

We  pass to the proof of the claim, hence we fix $\beta > \frac{N-2}{N-1}$ such that $|\nabla u|^{\beta}\in\W_{\loc}(\Omega)$. Since $u \in \testloc(\Omega)$, from the local Bochner inequality \eqref{localbochner} combined with the Kato inequality \eqref{eq:Kato} (with $t=N-\dim(\X)$ if $\dim(\X)<N$ and letting $t\to0$ if $N=\dim(\X)$, since in the latter case $0=\Delta u={\sf tr}\H u$) we have $|\nabla u|^2 \in D(\bd,\Omega)$ and
\begin{equation}\label{eq:conto1}
	\bd (|\nabla u|^2) \ge \left(\frac{2N}{N-1}|\nabla |\nabla u||^2+2K|\nabla u|^2\right)\mea\restr{\Omega}, 
\end{equation}
Hence from the chain rule for the Laplacian (second version in Proposition \ref{prop:chainlapl},  applied with $\phi\in C^2(\rr)$ as $\phi(t)=(t+\eps)^{\frac \beta2}$)  we have that $(|\nabla u|^2+\eps)^{\frac{\beta}{2}} \in D(\bd,\Omega)$ and
\begin{equation}\label{eq:betapiccolo}
	\begin{split}
			\bd((|\nabla u|^2+\eps)^{\frac{\beta}{2}}) \ge\bigg[\beta &(|\nabla u|^2+\eps)^{\frac{\beta}{2}-1}|\nabla |\nabla u||^2 \left (\frac{N}{N-1}+\frac{(\beta-2)|\nabla u|^2}{|\nabla u|^2+\eps}\wedge (\beta-2)\right )\\
			&+2K\beta (|\nabla u|^2+\eps)^{\frac{\beta}{2}} \frac{|\nabla u|^2}{|\nabla u|^2+\eps}\bigg]\mea\restr{\Omega},
	\end{split}
\end{equation}
for every $\eps>0$.
Setting $v_{\eps} \coloneqq \nabla(|\nabla u|^2+\eps)^{\frac{\beta}{2}}$ it is easy to see using dominated convergence that $|v_{\eps}-\nabla (|\nabla u|^\beta)|^2\to 0$ in $L^1_{\loc}(\Omega)$ as $\eps\to 0^+$. Moreover for every $\beta> \frac{N-2}{N-1}$,  denoting by $g_{\beta,\eps}\in L^1_{\loc}(\Omega)$ the function  on the right hand side of \eqref{eq:betapiccolo}, we have that   $\int_{\Omega} g_{\beta,\eps} f \d \mea \to \int_{\Omega} g_{\beta} f\d\mea+ 2K\beta \int_{\Omega}  |\nabla u|^\beta f\d\mea $, for every $f \in \Lip_c(\Omega)$ with $f \ge 0$,  where $g_\beta\in L^0(\Omega,\mea)$ is given by
\[
g_\beta \coloneqq \beta \left (\beta-\frac{N-2}{N-1}\right ) \nchi_{|\nabla u|>0}|\nabla u|^{\beta-2}|\nabla |\nabla u||^2.\\
\]
This can be seen applying dominated convergence for the second term in \eqref{eq:betapiccolo} and using respectively  dominated convergence in the case $\beta\ge 2$ and monotone convergence  in the case $\frac{N-2}{N-1}< \beta<2$, to deal with the first term.
 We are therefore in position to apply Lemma \ref{lem:techconv} and deduce both that $g_\beta \in L^1_{\loc}(\Omega)$ and that $|\nabla u|^{\beta}\in D(\bd,\Omega)$ with $\bd(|\nabla u|^\beta)\ge (g_\beta+2K\beta |\nabla u|^\beta) \mea\restr{\Omega}$. Moreover the fact that  $g_\beta \in L^1_{\loc}(\Omega)$ together with  Lemma \ref{lem:radice sobolev}  implies that $|\nabla u|^{\beta/2}\in \W_{\loc}(\Omega)$. This shows the claim and thus concludes the proof of the first part.
 
 We pass to  the case $\beta=\frac{N-2}{N-1}$. From the previous part we know that $|\nabla u|^\beta\in \W_{\loc}(\Omega)$, hence we can repeat the above argument and observe that in this case $g_\beta=0,$ from which the conclusions follows.
\end{proof}

\begin{proof}[Proof of Corollary \ref{cor:div estimate}]
We start observing that from the positivity and local Lipschitzianity of $u$ follows that $u^{-1}\in \Lip_{\loc}(\Omega)$ and thus $u^{\alpha }\in \W_{\loc}(\Omega)$ for every $\alpha\in \rr.$ Moreover, from the chain rule for the Laplacian (first version in Proposition \ref{prop:chainlapl}, applied with $u$ and $\phi(t)=t^{\alpha}$, $\alpha \in \rr$) and by the harmonicity of $u$, we deduce that $u^{\alpha}\in D(\bd,\Omega)$ with $\bd(u^{\alpha})=\alpha(\alpha-1)u^{\alpha-2}|\nabla u|^2 \mea\restr{\Omega}$.  
Hence from the Leibniz rule for the Laplacian (Proposition \ref{prop:leiblapl}) and Theorem \ref{thm:harm estimates} we have that $|\nabla u|^{\beta}u^\alpha \in D(\bd,\Omega)$ with
\[
\bd(|\nabla u|^{\beta}u^\alpha )\ge \left(u^\alpha g_\beta + \alpha(\alpha-1)u^{\alpha-2}|\nabla u|^{\beta+2}+2\alpha u^{\alpha-1} \la \nabla |\nabla u|^\beta,\nabla u\ra \right)\mea\restr{\Omega}+2K\beta u^\alpha|\nabla u|^\beta \, \mea \restr{\Omega},
\]
for every $\beta \ge \frac{N-2}{N-1}$ and every $\alpha \in \rr$, where $g_\beta$ is the same as in the proof of Theorem \ref{thm:harm estimates}.

Applying the Leibniz rule for the divergence (Proposition \ref{prop:leibdiv}) we deduce that $u^2\nabla(|\nabla u|^\beta u^{\alpha}) \in D({\bf div},\Omega)$ with
\begin{equation}\label{eq:div starting point}
	\begin{split}
		{\bf div}(u^2\nabla (|\nabla u|^\beta u^{\alpha}))\ge& \bigg (u^{\alpha +2} g_\beta +\alpha(\alpha+1) u^\alpha |\nabla u|^{\beta+2}+2(\alpha+1) u^{\alpha+1}  \la \nabla |\nabla u|^\beta,\nabla u\ra \bigg)\mea\restr{\Omega}\\
		&+2K\beta u^{\alpha+2}|\nabla u|^\beta \, \mea \restr{\Omega},
	\end{split}
\end{equation}
for every $\beta \ge\frac{N-2}{N-1} $ and every $\alpha \in \rr$.

We now assume that $\beta > \frac{N-2}{N-1}$. Since $|\nabla v|=(N-2)^{-1}u^{\frac{1-N}{N-2}}|\nabla u|$ and since from Theorem \ref{thm:harm estimates} we have $|\nabla u|^{\beta/2}\in \W_{\loc}(\Omega)$, it follows that $|\nabla v|^{\beta/2}\in \W_{\loc}(\Omega)$.

To see \eqref{eq:div estimate} we just take $\alpha=-\beta\frac{N-1}{N-2}$ in \eqref{eq:div starting point}. Then a direct computation gives that the right hand side of \eqref{eq:div starting point} equals the right hand side of \eqref{eq:div estimate}.

Finally choosing $\beta = \frac{N-2}{N-1}$, $\alpha=-1$ in \eqref{eq:div starting point} and recalling that in this case $g_\beta=0$, shows also the second part of the statement, thus finishing the proof.
\end{proof}

\section{The monotonicity formula}\label{sec:monotonicity}

\subsection{Decay estimates}\label{sec:estimates for u}
Throughout this section $(\X,\sfd,\mea)$ is a nonparabolic ${\rm RCD}(0,N)$ space (recall from Remark \ref{rmk:N>2} that in this case $N>2$), $\Omega\subset \X$ is open, unbounded,  with $\partial \Omega$ bounded, $x_0\in \Omega^c$ is fixed and $u$ is a solution to \eqref{mainpde}.
It follows from the maximum principle that
\[
0<u<\|u\|_{L^\infty}, \quad \text{in $\Omega$}.
\]
Moreover from Corollary \ref{cor:ends} we must have that $\Omega^c$ is bounded.

\begin{prop}\label{prop:gradbound}
	Set $R_0\coloneqq3\diam(\Omega^c)+1,$ then
	\begin{equation}
		\frac{|\nabla u|(x)}{u(x)} \le \frac{C}{\sfd(x,x_0)}, \quad \text{ for $\mea$-a.e. } x \in B_{R_0}(x_0)^c,
	\end{equation}
	where $C=C(N)$ is a positive constant  depending only on $N$.
\end{prop}
\begin{proof}
	Immediate from the gradient estimate \eqref{eq:cheng} with $R=\sfd(x,x_0)/4.$
\end{proof}

\begin{prop}\label{prop:twosidedbound}
	For every  $D,M>0$  there exists a positive constant $C_2=C_2(N,D,M)$ such that the following holds. Let $u$, $\Omega$ $x_0 \in \Omega^c$ as above with $\diam(\Omega^c)\le D$ and $\|u\|_{L^\infty}\le M$, then setting $\delta \coloneqq \sfd(x_0,\{u\le 1/2\})\wedge 1$ we have
	\begin{equation}\label{eq:twosidebound}
		\frac{\delta^{N-2}}{2}\sfd(x_0,x)^{2-N}\le u(x)\le C_2 \int_{\sfd(x,x_0)}^{+\infty} s\frac{\mea(B_1(x_0))}{\mea(B_s(x_0))}\, \d s, \quad \forall x  \in B_{R_0}(x_0)^c,
	\end{equation}
	where  and $R_0\coloneqq3\diam(\Omega^c)+1$ and the first inequality actually holds in $\Omega\cap B_{\delta}(x_0)^c$.
\end{prop}
Before passing to the proof we notice that from the assumption $\liminf_{y\to \partial \Omega}u(y)\ge 1$ we have $\sfd(x_0,\{u\le 1/2\})>0$, in particular the left inequality in \eqref{eq:twosidebound} is nontrivial.

\begin{proof}[Proof of Proposition \ref{prop:twosidedbound}]
	We start with the first inequality.
	
	From Laplacian comparison \eqref{eq:laplacian comparison} we know that $d_{x_0}^2 \in D(\bd)$ and $\bd d_{x_0}^2\le 2N \mea.$ Moreover from \eqref{eq:grad dist} $|\nabla d_{x_0}^2|^2=4d_{x_0}^2.$ Define now the function $h=d_{x_0}^{2-N}$, then from the chain rule for the Laplacian we have that $h\in D(\bd,\X\setminus\{x_0\})$, and
	\[ \bd h =\frac{2-N}{2} d_{x_0}^{-N}\bd d_{x_0}^{2}+\frac{2-N}{2}\left(\frac{2-N}{2}-1\right)4d_{x_0}^{2}d_{x_0}^{-N-2}=\frac{N-2}{2}d_{x_0}^{-N}(2N\mea -\bd d_{x_0}^2)\ge 0.\]
	Hence $h$ is subharmonic in $\X\setminus \{x_0\}$.  Moreover we have that $\lambda h\le 1/2$ in $B_{\delta}(x_0)^c$, where $\lambda\coloneqq\delta^{N-2}/2$. Finally  from the assumption $\sfd(x_0,\{u\le 1/2\})>\delta$ we have $u\ge \lambda h$ in $\Omega\cap B_{\delta}(x_0)^c$. Fix now  $r>0$ and define the open set $\Omega^r\coloneqq\{x \in \Omega \ : \ \sfd(x,\Omega^c)>r\}$. Observe that the function $ \lambda h-u$ is subharmonic in $\Omega^r$. Therefore from the weak maximum principle (see Proposition \ref{prop:maxprinc}) we deduce that
	\begin{align*}
		\sup_{\Omega^r\cap B_R(x_0)\cap B_{\delta}(x_0)^c} \, (\lambda h-u)&\le  \max_{\partial \Omega^r\cap B_{\delta}(x_0)^c}(\lambda h-u)\vee  \max_{ \partial B_{\delta}(x_0)\cap \Omega^r}(\lambda h-u)\vee \max_{ \partial B_{R}(x_0)}(\lambda h-u)\\
		&\le   \max_{\partial \Omega^r}(1/2-u)\vee 0\vee \max_{ \partial B_{R}(x_0)}(\lambda h-u)\, ,
	\end{align*}
	for every $R>R_0$. Sending $R$ to $+\infty$ and $r$ to 0, recalling that both $h$ and $u$ vanish  at infinity (since $N>2)$ and that $\liminf_{x\to \partial \Omega}u(x)\ge 1,$ we conclude that $\lambda h \le u$ in $\Omega\cap B_{\delta}(x_0)^c$. This proves the first inequality in \eqref{eq:twosidebound}.
	
	
	We now pass to the second inequality in \eqref{eq:twosidebound}. We argue by comparison with the quasi Green function $G^1(x)\coloneqq G^1(x_0,x)$ (recall its definition in \eqref{eq:quasi green def}).
		Recall that $G^1$ is superharmonic in $\X$. Moreover, using the upper bound for the Green function  and the estimates of the heat kernel we deduce that
	\begin{equation}\label{eq:twosidesboundG}
		c_1^{-1} \int_{1}^{+\infty} \frac{e^{-\frac{\sfd(x_0,x)^2}{3s}}}{\mea(B_{\sqrt s}(x_0))}\, \d s \overset{\eqref{eq:kernelestimate}  }{\le} G^1(x)\le G(x,x_0) \overset{\eqref{eq:greenestimates}}{\le}  c_1 \int_{\sfd(x,x_0)}^{+\infty} \frac{s}{\mea(B_s(x_0))}\, \d s,
	\end{equation}
	for every $x \in X\setminus\{x_0\},$ for some positive constant $c_1=c_1(N)>1$. From Bishop-Gromov inequality and using the change of variable $s=t\sfd(x_0,x)^2$ we obtain 
	\[
	G^1(x)\ge c_1^{-1}\frac{\sfd(x_0,x)^{2-N}}{\mea(B_1(x_0))}\int_{\frac{1}{\sfd(x_0,x)^2}}^{+\infty} \frac{e^{-\frac{1}{3t}}}{t^{\frac N2}}\, \d t
	\ge C_1 \frac{\sfd(x_0,x)^{2-N}}{\mea(B_1(x_0))}, \quad \forall x \in B_1(x_0)^c,\]
	for some constant $C_1$ depending only on $N$.
	Therefore, taking $\lambda\coloneqq M\mea(B_1(x_0))\frac{R_0^{N-2}}{C_1}$, we have  $\lambda G^1\ge M\ge  u$ in $\partial B_{R_0}(x_0).$ Hence, since $\lambda G^1-u$ is superharmonic in $\Omega$, from the weak maximum principle it follows that for every $R>R_0$
	\[ \inf_{B_{R}(x_0)\cap B_{R_0}(x_0)^c} \, (\lambda G^1-u)= \min_{\partial B_{R}(x_0) \cup \partial B_{R_0}(x_0)} (\lambda G^1-u)\ge \min(0,\min_{ \partial B_{R}(x_0)}(\lambda G^1-u)\, ).\]
	Sending $R$ to $+\infty$  and recalling that both $G^1$ and $u$ go to 0 at infinity we conclude that $u\le \lambda G^1$ in $B_{R_0}(x_0)^c,$ which combined with the second bound in \eqref{eq:twosidesboundG} gives the second inequality in \eqref{eq:twosidebound}.
\end{proof}

\subsection{Monotonicity}\label{sec:actual monotonicity}
As in the previous section $(\X,\sfd,\mea)$ is a nonparabolic ${\rm RCD}(0,N)$ space, $\Omega\subset \X$ is open, unbounded,  with $\partial \Omega$ bounded and $u$ is a solution to \eqref{mainpde}.

We start with the following simple remark, which  allows to define $U_\beta$ and will be needed to justify the many applications of the  coarea formula along all this section.
\begin{remark}\label{rmk:coarea}
Since $u$ is locally Lipschitz (recall Proposition \ref{prop:gradbound}), satisfies  $\liminf_{x\to \partial \Omega}u(x)\ge 1$ and vanishes at infinity, it follows that $u$ satisfies the hypotheses needed to apply the coarea formula \eqref{eq:coarea} in $\Omega$. In particular for every $f \in L^1_{\loc}(\Omega)$ with $f \mea\restr{\Omega}\ll|\nabla u|\mea \restr\Omega$ we have 
\begin{equation}\label{eq:coarea applicazioni}
	\int_0^1 \phi(t)\int g \, \d \Per(\{u<r\})\,\d r= \int_{\Omega} \phi(u) f \, \d \mea< +\infty \ , \quad \forall \phi:[0,1] \to \rr \text{ Borel, with } \supp \phi \subset(0,1),
\end{equation}
where $g$ is any Borel representative of the function $\frac{\d (f\mea\restr{\Omega })}{\d (|\nabla u|\mea\restr{\Omega})}$. Therefore, by the arbitrariness of  $\phi$, we also deduce that:
\begin{equation}\label{eq:well def}
	\begin{split}
		&\text{for any $f \in L^1_{\loc}(\Omega)$ with $f \mea\restr{\Omega}\ll|\nabla u|\mea \restr\Omega$  and for any Borel representative $g$ of $\frac{\d (f\mea\restr{\Omega })}{\d (|\nabla u|\mea\restr{\Omega})}$,} \\
		&\text{the function $(0,1)\ni r\mapsto \int g \, \d \Per(\{u<r\})$ is in $L^1_{\loc}(0,1)$ and (its a.e. equivalence class)}\\
		& \text{does not depend on the choice  of the representative $g$.}
	\end{split}
\end{equation}\fr
\end{remark}

Choosing in the above remark $f=\frac{|\nabla u|^{\beta+2}}{u^{\beta \frac{N-1}{N-2}}} \in L^1_{\loc}(\Omega)$, with $\beta > -2$, and observing that $\supp( \Per(\{u<t\})) \subset \{u=t\}$ we deduce that, fixed a Borel representative of $|\nabla u|,$ the function
\begin{equation}\label{eq:defU}
U_{\beta}(t)\coloneqq\frac{1}{t^{\beta \frac{N-1}{N-2}}}\int |\nabla u|^{\beta+1}\, \d \Per(\{u<t\}) \, \in L^1_{\loc}(0,1),
\end{equation}
is well defined and independent of the representative chosen for $|\nabla u|.$ It worth to recall that  after the work of \cite{GPcap} a canonical choice for the representative of $|\nabla u|$ could be its quasi-continuous representative (see \cite{GPcap} for the precise definition). An interesting point would be to investigate  the relation between the representative of $U_\beta$ given by this canonical choice and the continuous representative of $U_\beta$, which exists thanks to Theorem \ref{thm:monotonicity}. We will not investigate this point in the present paper.

We are ready to state our main result regarding monotonicity.
\begin{theorem}\label{thm:monotonicity}
	Let $\X$ be a nonparabolic ${\rm RCD}(0,N)$ space and let $\Omega\subset \X$ be open, unbounded,  with $\partial \Omega$ bounded. Suppose that $u$ is a solution of \eqref{mainpde} and let $U_{\beta}$, with $\beta \ge \frac{N-2}{N-1}$, be the function defined in \eqref{eq:defU}. Then $U_{\beta}\in W^{1,1}_{\loc}(0,1)$, $U'_{\beta} \in {BV_{\loc}}(0,1)$ and
	\begin{equation}\label{eq:U'positive}
		U^{'-}_{\beta}(t)\ge \frac{C_{\beta,N}}{t^2}\int_{\{u<t\}} \, u^2|\nabla |\nabla u^{\frac{1}{2-N}}|^{\frac{\beta}{2}}|^2 \d \mea,\quad \forall \, t \in (0,1],
	\end{equation}
	(recall that $|\nabla u^{\frac{1}{2-N}}|^{\frac{\beta}{2}} \in \W_{loc}(\Omega)$ for every $\beta>\frac{N-2}{N-1}$ by Corollary \ref{cor:div estimate}) where $ C_{\beta,N}=\frac{4}{\beta}\left(\beta-\frac{N-2}{N-1}\right)$, $U^{'-}_{\beta}$ is the left continuous representative of $U'_{\beta}$ and where the left hand side is taken to be 0 if $\beta=\frac{N-2}{N-1}$. In particular $U_{\beta}$ is non-decreasing.
\end{theorem}

To prove Theorem \ref{thm:monotonicity} we start computing the first derivative of $U_\beta$ (which does not evidently carry a sign).
\begin{prop}\label{prop:Uac}
	With the same assumptions as in Theorem \ref{thm:monotonicity}, the function $U_{\beta}$ belongs to $W^{1,1}_{\loc}(0,1)$ and its derivative is given by
	\begin{equation}\label{eq:Uprimo}
		U_\beta'(t)=  \int \la \frac{\nabla u}{|\nabla u|},\nabla \left(\frac{|\nabla u|^{\beta}}{u^{\beta \frac{N-1}{N-2}}}\right)\ra \d \Per(\{u<t\}), \quad \text{a.e. }t\in(0,1),
	\end{equation}
where the right hand side has to be intended as in \eqref{eq:well def} with $f=\la \nabla u,\nabla \left(\frac{|\nabla u|^{\beta}}{u^{\beta \frac{N-1}{N-2}}}\right)\ra$.
\end{prop}
\begin{proof}
	Consider the vector field $v\coloneqq\frac{\nabla u|\nabla u|^{\beta}}{u^{\beta \frac{N-1}{N-2}}}\in L^0(T\X)\restr{\Omega}$ for $\beta \ge \frac{N-2}{N-1}$ and observe that from the Leibniz rule for the divergence (second version in Proposition \ref{prop:leibdiv}) $v \in D(\div,\Omega)$ with
	\[ \div(v)=\la \nabla u,\nabla \left(\frac{|\nabla u|^{\beta}}{u^{\beta \frac{N-1}{N-2}}}\right)\ra \in L^1_{\loc}(\Omega),\]
	thanks to the harmonicity of $u$. In particular $\div(v)\mea \ll |\nabla u|\mea$, hence recalling \eqref{eq:coarea applicazioni} and integrating by parts we have
	\[
	\int_0^1 \int  \frac{\div(v)}{|\nabla u|} \d \Per(\{u<t\})\phi(t) \d t\overset{\eqref{eq:coarea applicazioni}}{=}\int \div(v)\phi(u) \d \mea=-\int \frac{|\nabla u|^{\beta+2}}{u^{\beta \frac{N-1}{N-2}}} \phi'(u)\d \mea\overset{\eqref{eq:coarea}}{=} -\int_0^1 U_{\beta}(t)\phi'(t) \d t,
	\]
	for every $\phi \in C^1_c(0,1) $, where in the last step we used that $\supp (\Per(\{u<t\},.))\subset \{u=t\}$ and  with $\frac{\div(v)}{|\nabla u|}$ denoting any Borel representative of $\frac{\d(\div(v)\mea\restr\Omega)}{\d(|\nabla u|\mea\restr\Omega))}$.  The conclusion follows.
\end{proof}

To prove that $U'_\beta$ is nonnegative we need to push our analysis to  the second order and in particular to compute the derivative of $U_{\beta}'(t)t^2$. The reason for the term $t^2$ is that the key vector field with nonnegative divergence of Corollary \ref{cor:div estimate} presents a term $u^2.$

\begin{prop}\label{prop:U'bv}
	With the same assumptions as in Theorem \ref{thm:monotonicity}, the function $U_{\beta}(t)'t^2$ belongs to ${BV}_{\loc}(0,1)$ and
	\begin{equation}\label{eq:Usecondo}
		(U_{\beta}'(t)t^2)'\ge C_{\beta,N} \left( \int\frac{u^2|\nabla |\nabla u^{\frac{1}{2-N}}|^{\frac{\beta}{2}}|^2}{|\nabla u|} \d \Per (\{u<t\}) \right)\,\mathcal{L}^1\restr{(0,1)}\ge 0,
	\end{equation}
	where $C_{\beta,N}=\frac{4}{\beta}\left(\beta-\frac{N-2}{N-1}\right)$ and where the right hand side has to be intended as in \eqref{eq:well def} with $f=u^2|\nabla |\nabla u^{\frac{1}{2-N}}|^{\frac{\beta}{2}}|^2$ when $\beta>\frac{N-2}{N-1}$  (recall also that from Corollary \ref{cor:div estimate} $|\nabla u^{\frac{1}{2-N}}|\in\W_{\loc}(\Omega)$), and  identically 0 in the case $\beta=\frac{N-2}{N-1}$.
\end{prop}
\begin{proof}
	Consider any  nonnegative $\phi\in C^1_c(0,1)$.  Applying formula \eqref{eq:Uprimo} and the coarea formula \eqref{eq:coarea applicazioni}
	\[
	\int_0^1 (U_{\beta}'(t)t^2)\phi'(t) \d t =\int_0^1 \int \la \frac{\nabla u}{|\nabla u|},u^2\nabla \left(\frac{|\nabla u|^{\beta}}{u^{\beta \frac{N-1}{N-2}}}\right)\ra  \phi'(u)\d \Per(\{u<t\})\overset{\eqref{eq:coarea applicazioni}}{=}\int\la\nabla (\phi(u)), u^2\nabla \left(\frac{|\nabla u|^{\beta}}{u^{\beta \frac{N-1}{N-2}}}\right)\ra\, \d \mea,
	\]
	observing that $\phi(u)\in\LIP_{c}(\Omega)$ and recalling from Corollary \ref{cor:div estimate} that $u^2\nabla \left (\frac{|\nabla u|^{\beta}}{u^{\beta \frac{N-1}{N-2}}}\right ) \in D({\bf{div}},\Omega)$, we obtain
	\[
	-\int_0^1 (U_{\beta}'(t)t^2)\phi'(t) \d t = \int \phi(u) \,\d{\bf{div}}\left(u^2\nabla \left (\frac{|Du|^{\beta}}{u^{\beta \frac{N-1}{N-2}}}\right )\right).
	\]
	We now plug in \eqref{eq:div estimate} and (when $\beta>\frac{N-2}{N-1}$) apply the coarea formula \eqref{eq:coarea applicazioni} (observe that $||\nabla |\nabla  u^{\frac{1}{2-N}}|^{\frac{\beta}{2}}|^2 |\mea\ll |\nabla u|\mea$ ) to obtain
	\begin{equation}\label{eq:bv}
		-\int_0^1 (U_{\beta}'(t)t^2)\phi'(t) \d t \ge \tilde  C_{\beta,N}\, \int_0^1 \int  u^2|\nabla u|^{-1}|\nabla |\nabla  u^{\frac{1}{2-N}}|^{\frac{\beta}{2}}|^2\, \mea\restr{\Omega} \d \Per (\{u<t\}) \phi(t) \, \d t \ge  0,
	\end{equation}
 with $u^2|\nabla u|^{-1}|\nabla |\nabla  u^{\frac{1}{2-N}}|^{\frac{\beta}{2}}|^2$ denoting any Borel representative of $\frac{\d\left (u^2|\nabla |\nabla  u^{\frac{1}{2-N}}|^{\frac{\beta}{2}}|^2\mea\restr\Omega\right )}{\d(|\nabla u|\mea\restr\Omega))}$. 
	The proof is concluded observing that \eqref{eq:bv} gives at once that the distributional derivative of $U'(t)t^2$ is a locally finite measure  (it is positive) and that \eqref{eq:Usecondo} holds.
\end{proof}
Justified by Proposition \ref{prop:Uac},  from this point on  we will  identify $U_{\beta}$ with its continuous representative. 
Moreover Proposition \ref{prop:U'bv} guarantees that $U'_{\beta} \in  {BV}_{\loc}(0,1)$, thus we will denote by  $U^{'-}_{\beta}$ its representative which  is left-continuous in $(0,1]$ (notice that $U^{'-}_{\beta}$ might take value $+\infty$ at $t=1$).  We observe also that  \eqref{eq:Usecondo} implies that
\begin{equation}\label{eq:increasing}
	(0,1] \ni t \mapsto U^{'-}_{\beta}(t)t^2 \text{ is a non decreasing function.}
\end{equation}
To prove Theorem \ref{thm:monotonicity} we aim to integrate \eqref{eq:Usecondo}, however to do so we still need to know that $U_\beta$ is bounded close to 0. In particular we prove the following:
\begin{prop}
	With the same assumptions as in Theorem \ref{thm:monotonicity},
	\begin{equation}\label{eq:Ubounded}
	U_\beta \in L^\infty(0,1/2 ).
	\end{equation}
\end{prop}
\begin{proof}
It is enough to show that 
\begin{equation}\label{eq:scontro}
\left |\int_0^{\frac 12} U_{\beta}\phi \, \d t \right | \le C \int_0^{\frac 12} |\phi|, \quad  \forall \phi \in C^1_c (0,1/2 ),
\end{equation}
for some positive constant $C$ independent of $\phi.$

We start observing that, integrating by parts and applying the coarea formula \eqref{eq:coarea applicazioni},
\[ 0=\int_{\Omega} \Delta u \phi(u)\, \d \mea=-\int_{\Omega}|\nabla u|^2\phi'(u), \d \mea=-\int_0^1 \int |\nabla u| \, \d \Per(\{u<r\}) \phi'(r)\d r,\quad  \forall \phi \in C^1_c  (0,1 ), \]
in particular $\int |\nabla u| \, \d \Per(\{u<r\})=D$ for a.e. $r\in(0,1)$, for some constant  $D$. Therefore using again the coarea formula
\begin{align*}\left |\int_0^{\frac 12} U_{\beta} \phi \, \d t\right |&\le \int u^{\beta \frac{1-N}{N-2}} |\nabla u|^{\beta+2}|\phi(u)|\, \d\mea  \le \left  \||\nabla u|^{\beta }u^{\beta \frac{1-N}{N-2}} \right \|_{L^{\infty}(\{u\le 1/2\})}  \int |\nabla u|^2|\phi(u)| \d \mea \\
&\overset{\eqref{eq:coarea applicazioni}}{=} \left  \||\nabla u|^{\beta }u^{\beta \frac{1-N}{N-2}} \right \|_{L^{\infty}(\{u\le 1/2\})}  \int_0^{\frac 12} \int |\nabla u| \, \d \Per(\{u<r\}) |\phi(r)|\, \d r \\
&= D\left   \||\nabla u|^{\beta }u^{\beta \frac{1-N}{N-2}} \right \|_{L^{\infty}(\{u\le 1/2\})}  \int_0^{\frac 12}|\phi|, \quad \forall \phi \in C^1_c (0,1/2).
\end{align*}
Therefore to prove \eqref{eq:scontro} it remains to show that  $\left   \||\nabla u|^{\beta}u^{\beta \frac{1-N}{N-2}} \right \|_{L^{\infty}(\{u\le 1/2\})}<+\infty$.
 Let $R_0$ be as in Proposition \ref{prop:twosidedbound} and observe that  $u$, being positive and satisfying $\liminf_{x\to \partial \Omega}u(x)\ge 1$, is bounded away from zero in ${B_{R_0}(x_0)\cap \Omega}$. Moreover again thanks to $\liminf_{x\to \partial \Omega}u(x)\ge 1$  we have  $\sfd(\partial \Omega ,\{u\le1/2\})>0.$ Therefore, since $u \in \LIP_{\loc}(\Omega)$,  we have
	\[\left \||\nabla u|^{\beta }u^{\beta \frac{1-N}{N-2}} \right \|_{L^{\infty}(\{u\le1/2\}\cap B_{R_0}(x_0))}<+\infty . \]
	Moreover combining  Proposition \ref{prop:gradbound} and the lower bound in \eqref{eq:twosidebound} we obtain
	\[ \||\nabla u|^{\beta }u^{\beta \frac{1-N}{N-2}}\|_{L^\infty(\Omega\cap B_{R_0}(x_0)^c)}  \le \left \|\left(\frac{C(N)}{\sfd(.,x_0)^{N-2}u}	\right)^\beta \right \|_{L^\infty(\Omega\cap B_{R_0}(x_0)^c)}<+\infty.  \]
	Combining the two estimates we conclude.
	\end{proof}

We are now ready to prove the main monotonicity result.

\begin{proof}[Proof of Theorem \ref{thm:monotonicity}]
We start observing that, thanks to \eqref{eq:Usecondo}, setting $\mu\coloneqq(U_{\beta}(t)'t^2)'\ge 0$ we have 
	\begin{equation}
	U_{\beta}^{'-}(t)t^2-U_{\beta}^{'-}(s)s^2=\mu([s,t))\ge \int_s^t\int g \d \Per (\{u<r\}) \, \d r\overset{\eqref{eq:coarea applicazioni}}{=}\int_{\{s<u<t\}} \Phi(u) \d \mea,
	\end{equation}
	for every $0<s<t\le 1$, with $\Phi(u)= C_{\beta,N}u^2|\nabla |\nabla u^{\frac{1}{2-N}}|^{\frac{\beta}{2}}|^2$ when $\beta>\frac{N-2}{N-1}$, $\Phi(u)=0$ when $\beta=\frac{N-2}{N-1}$ and where $g$ is a Borel representative of $\frac{\d( \Phi(u) \mea\restr\Omega )}{\d(|\nabla u|\mea\restr\Omega))}$.
	Therefore to conclude it is enough to prove that there exists a sequence $s_n\to 0^+$ such that $U_{\beta}^{'-}(s_n){s_n^2}\to 0.$
	
To achieve this we first prove that $U^{'-}_{\beta}(t)\ge 0$ for every $t \in (0,1)$. We assume by contradiction that there exists $T\in(0,1)$ such that $U^{'-}_{\beta}(T)<0$.  From \eqref{eq:increasing} 
	\[
U^{'-}_{\beta}(s)\le U^{'-}_{\beta}(T)\frac {T^2}{s^2}, \quad \forall s<T,
	\]
	from which integrating with respect to $s$  on the interval $(t,T)$
	\[
	U_{\beta}(T)-U_{\beta}(t)\le U^{'-}_{\beta}(T)T^2 \left( \frac 1t-\frac 1T\right).
	\]
	Sending $t\to 0^+$ and recalling that $U^{'-}_{\beta}(T)<0$ we obtain $U_\beta(t)\to +\infty$ as $t\to 0^+$, which  however contradicts \eqref{eq:Ubounded}.

	Since $U^{'-}_{\beta}(t)\ge 0$ we have that $U_{\beta}$ is non-decreasing and also non-negative, hence it admits a limit as $t\to 0^+.$ In particular $U^{'-}_{\beta} \in L^1(0,\frac 12),$ therefore
	\begin{equation*}
	a_n\coloneqq\int_{2^{-(n+1)}}^{2^{-n}} U_{\beta}^{'-}(t) \d t\to 0, \text{  as $n\to +\infty$.}
	\end{equation*}
	Moreover from Markov inequality we have $|\{U_{\beta}^{'-}>a_n2^{n+2} \}\cap(2^{-(n+1)},2^{-n})|\le \frac 12 2^{-(n+1)}$, thus for   every $n$ we can find $s_n \in (2^{-(n+1)},2^{-n})$ such that $U_{\beta}^{'-}(s_n)\le a_n2^{n+2}$. Therefore
	$$U_{\beta}^{'-}(s_n)s_n\le a_n2^{n+2}s_n<4a_n\to 0$$
	and the proof is complete.
\end{proof}

\section{Functional versions of the rigidity and almost rigidity}\label{sec:coni} 

\subsection{From outer functional  cone to  outer metric cone}\label{sec:functioncone}
The following result is a variant of the ``from volume cone to metric cone" theorem for ${\rm RCD}$ spaces (see \cite{volcon}). The two main differences with the work in \cite{volcon} are that here we start from a function satisfying an equation (instead that from a condition on the measure), from which we deduce a conical structure on the complement of a bounded set (instead that on a ball). Related to this type of results we mention also \cite{volumebounds} and \cite{boundary} where  functional versions of the splitting theorem in non-smooth setting have been obtained.  Finally we recall that the almost-splitting theorems in the smooth setting proved in \cite{CC} were also based on a functional formulation similar to the one we are considering here.

\begin{theorem}\label{thm:functional cone}
	Let $(X,\sfd,\mea)$ be an ${\rm RCD}(0,N)$ space with $N \in [2,\infty)$  and  $U \subset X$ be  open  with $\partial U$ bounded. Suppose there exists a positive and continuous function  $\bu \in D(\bd,U)$ such that $\Delta \bu =N $ $\mea$-a.e. in $U$, $|\nabla \sqrt{2\bu}|^2=1$ $\mea$-a.e. in $U$, $\bu_0\coloneqq \limsup_{x\to \partial U}\bu(x)<+\infty$ and $\{\bu >\bu_0\}\neq\emptyset$.
	Then
	\begin{itemize}
		\item [i)] there exists unique an ${\rm RCD}(N-2,N-1)$ space $(Z,\sfd_Z,\mea_Z)$ with $\text{diam}(Z)\le \pi$ and a bijective measure preserving local isometry  $S:\{\bu>\bu_0\}\to Y\setminus \overline{B_{r}}(O_Y)$, with $r\coloneqq\sqrt{2 \bu_0}$ and where $(Y,\sfd_Y,\mea_Y)$ is the Euclidean $N$-cone built over $Z$ with vertex $O_Y,$
		\item [ii)]
		\begin{itemize}
			\item if $D_Z\coloneqq\diam (Z)<\pi$ then local isometry of point i) is an isometry between $Y\setminus \overline{B_{r_Z}}(O_Y)$ and $\left \{\bu >r_{Z}^2/2\right \}$, where $r_Z\coloneqq r(1-\sin{D_Z/2})^{-1}>r,$\\
			\item if $\diam(Z)=\pi$, then $(X,\sfd,\mea)$ isomorphic to $(Y,\sfd_Y,\mea_Y)$,
		\end{itemize}
		\item [iii)] the function $\bu$ has the following explicit form
		\begin{equation}\label{eq:explicit}
			\bu(x)=\frac 12\sfd_Y(S(x),O_Y)^2=\frac 12(\sfd(x,\partial \{\bu > \bu_0\})+\sqrt{2\bu_0})^2, \quad \forall x \in \{\bu>\bu_0\},
		\end{equation}
		in particular the level set $\{u=\frac{t^2}{2}\}$, for every $t>\bu_0$, is Lipschitz-path connected and isometric (with its induced intrinsic distance) to $(Z,t\sfd_Z)$.
	\end{itemize}
\end{theorem}
We observe that the uniqueness part  of Theorem \ref{thm:functional cone} is an immediate consequence of the rest of the statement. Indeed, from the last part of $iii)$ we deduce that the metric space $(Z,\sfd_Z)$ (and thus  $(Y,\sfd_Y)$) is uniquely determined up to isometries. Moreover, since $S$ is measure preserving, the measure $\mea_Y$ is uniquely determined as well, hence from the definition of the measure in an $N$-cone we obtain that also $\mea_Z$ is uniquely determined.

As already said, the proof of the above Theorem is mainly an adaptation of the proof in \cite{volcon}. However some parts will require new arguments. The first main point is that the in \cite{volcon} the starting point is the  gradient flow of the distance function $\sfd$, which is used to deduce analytical information on $\sfd$. Here instead we start from an analytical information, i.e. a PDE, and we want to build a flow. This will be done through the tool of Regular Lagrangian Flows. One of the main tools we need to develop in this regard is an a-priori estimate of local type, which seems to be missing in literature and does not follows immediately from the standard global a-priori estimates in \cite{AT}.

The second main difference is that here the analysis takes place in the complementary of a bounded set, while in \cite{volcon} all the work is done inside a fixed ball. Among other things, this difference will mainly affect the way in which we deduce that the cone is itself an ${\rm RCD}(0,N)$ space. Indeed in \cite{volcon} this follows from the fact that a whole ball with center $x_0$ is isometric to a ball centered at the tip of the cone, therefore any blow up of the space at $x_0$ will converge to the said cone. Then from the closedness of the ${\rm RCD}(0,N)$ condition the conclusion follows. However in our case  the same argument cannot be applied, indeed our isometry is by nature far from the tip of the cone. This issue will be overcome noticing that our isometry is almost global, meaning that the space is isometric to the cone outside a bounded set. This allow us to deduce that any blow down of the space will converge to the cone, which gives the conclusion again by the closedness of the ${\rm RCD}(0,N)$ class.

Since they are interesting on their own and  independent  of the rest of the proof, we isolate the two ingredients that we just described in the following two subsections. The remaining part of the argument will be outlined in Appendix \ref{ap:cone}.

\subsubsection{The blow down argument}

\begin{prop}\label{prop:blowdown}
	Let $(\X,\sfd_{\X},\mea_{\X})$ be a m.m.s. and let $U\subset \X$ be closed and bounded. Suppose that there exists an Euclidean $N$-cone $(Y,\sfd_Y,\mea_Y)$  over a m.m.s. $Z$,  $N \in[1,\infty)$, with tip $O_Y$ and a bijective local isometry $T:U^c \to Y\setminus B_R(O_Y)$, which is measure-preserving, i.e. $T_*\mea_{\X}\restr{U^c}=\mea_Y\restr{B_R(O_Y)}$. Then for every $x_0\in\X$ and every sequence $r_n\to +\infty$ it holds that
	\[
	(\X,r_n^{-1}\sfd_{\X},r_n^{-N}\mea_{\X},x_0)\overset{{\text pmGH}}{\longrightarrow}(Y,\sfd_Y,\mea_Y,O_Y).
	\]
	In particular if $\X$ is an ${\sf{RCD}}(0,N)$ space, then $Y$ is an ${\sf{RCD}}(0,N)$ space as well.
	
	Finally if $\X$ is ${\sf{RCD}}(0,N)$ and $\diam(Z)=\pi$  then $\X$ is isomorphic to $Y$ as m.m.s..
\end{prop}
\begin{proof}
	Fix $x_0 \in \X$ and observe that, up to increase $R$ and  enlarge $U$, it is not restrictive to assume  that $x_0\in U$.
	
	Set $D\coloneqq\diam(U)$, $\delta_n\coloneqq \frac{1}{4r_n}(D+R)$ and define $\X_n\coloneqq(\X,r_n^{-1}\sfd_{\X},r_n^{-N}\mea_{\X},x_0)$. Without loss of generality we will assume that $r_n\ge 1.$
	
	Since $Y$ is an Euclidean $N$-cone it follows that   $Y_n\coloneqq(Y,r_n^{-1}\sfd_Y,r_n^{-N}\mea_Y)$ is isomorphic to $(Y,\sfd_Y,\mea_Y)$ via the  map $i_n:Y_n \to Y$, defined as $i_n(t,z)\coloneqq(t/r_n,z)$ in polar coordinates, which satisfies $i_n(B_R(O_Y))=B_{r_n^{-1}R}O_Y$ for every $R>0.$ Observe that in particular ${i_n}_*\mea_Y=r_n^N\mea_Y.$
	
	We  extend  $T$ to the whole $\X$ by setting $T(x)=O_Y$ for every $x \in U$ and we denote this new map again by $T$.  It is straightforward to check that
	\begin{equation}\label{eq:smartbound}
		|\sfd_X(x_1,x_2)-\sfd_Y(T(x_1),T(x_2))|\le 2(R+D), \quad \forall\, x_1,x_2 \in X.
	\end{equation}
	Define now the map $T_n : \X_n \to Y$ as $T_n=i_n\circ T.$  It follows from \eqref{eq:smartbound} and the properties of $i_n$ that $T_n$ is a $\delta_n$-isometry. Moreover it can be readily checked that $B^Y_{R-D}(O_Y)\subset T(B^{\X}_{R}(x_0))$, for any $R>D.$ In particular it follows that $B^Y_{R-\delta_n}(O_Y)\subset T_n(B^{\X_n}_{R}(x_0))$, for every $R>D$. Finally we let $\phi \in C_b(Y)$ be of bounded support,  since $T$ is measure preserving we have
	\begin{align*}
		r_n^{-N}\int \phi\circ T_n \d \mea&=r_n^{-N}\int_U \phi\circ T_n \d \mea+r_n^{-N}\int_{Y\setminus B_R(O_Y)} \phi\circ i_n \d \mea_Y\\
		&=r_n^{-N}\int_U \phi\circ \d {T_n}_*\mea+\int_{Y\setminus B_{r_n^{-1}R}(O_Y)} \phi \d \mea_Y.
	\end{align*}
	Passing to the limit, observing that the first term on the right hand side vanishes as $r_n\to +\infty$, we obtain $r_n^{-N}\int \phi\circ T_n \d \mea\to \int \phi \d \mea_Y$ as $r_n\to +\infty$. This concludes the first part.
	
	The second part follows immediately from the closedness of the ${\rm RCD}(0,N)$ condition under pmGH-convergence.
	
	Suppose now that $\X$ is an ${\rm RCD}(0,N)$ space and $\diam(Z)=\pi$. Then $Y$ must contain a line. Therefore, since from the previous part $Y$ is an ${\rm RCD}(0,N)$ space, it follows from the splitting theorem (\cite{split}, \cite{split2}) that $Y$ is isomorphic to $(\rr\times Y',\sfd_{Eucl}\times \sfd',\mathcal{L}^1\otimes \mea_Y')$ for some m.m.s. $(Y',\sfd',\mea_{Y'})$. In particular  $O_Y=(\bar t, \bar y)$ for some $\bar t \in \rr$ and $\bar y \in Y'$ and $\mea_{Y}(B_r(O_Y))=\mea_{Y}(B_r(s,\bar y)),$ for any $r>0$ and any $s\in \rr.$

	 Therefore taking $s$ big enough we have that $O'\coloneqq(s,\bar y)\in Y$ satisfies $O'\in\{\sfd_Y(.,O_Y)>R+1\}.$ Therefore, since $ T|_{U^c}$ is a measure preserving local isometry, 
	$\mea_Y(B_r(O'))=\mea_X(B_r(T^{-1}(O')))$ holds for every $r \in(0,1).$ Hence
	\[
	\lim_{r\to 0^+} \frac{\mea(B_r(T^{-1}(O')))}{r^N}=\lim_{r\to 0^+}\frac{\mea_Y(B_r(O'))}{r^N}=\lim_{r\to 0^+}\frac{\mea_Y(B_r(O_Y))}{r^N}=\mea_Y(B_1(O_Y))\eqqcolon\theta,
	\]
	since $O_Y$ is the vertex of $Y$.
	On the other hand, since $\X_n\coloneqq(\X,r_n^{-1}\sfd_{\X},r_n^{-N}\mea_{\X},x_0)\overset{{\text pmGH}}{\longrightarrow}(Y,\sfd_Y,\mea_Y,O_Y)$ we have 
	\[
	\lim_{r_n\to +\infty} \frac{\mea(B_{r_n}(T^{-1}(O')))}{{r_n}^N}=\lim_{r_n\to+\infty}\frac{\mea_X(B_{r_n}(x_0))}{r_n^N}=\lim_n \mea_{X_n}(B_1(x_0))=\mea_Y(B_1(O_Y))=\theta.
	\]
	From Bishop-Gromov inequality we deduce that $\frac{\mea(B_r({T^{-1}}(O')))}{r^N}=\theta$ for every $r>0$ and from \cite[Thm. 1.1]{volcon} we must have that $\X$ is a cone, which must evidently  coincide with $Y.$
\end{proof}

\subsubsection{Local a priori estimate for Regular Lagrangian Flows in RCD spaces}

The following local version of the a priori estimates in \cite[Prop. 4.6]{AT} will be crucial for the argument of Appendix \ref{ap:cone} to work (see Proposition \ref{prop:measureannuli}).
\begin{prop}\label{prop:localest}
	Let $\{v_t,\mu_t\}_{t \in[0,T]}$ be as in Theorem \ref{thm:uniqrfl}. Assume additionally that $\{\mu_t\}_{t\in[0,T]}$ are all concentrated in a common bounded Borel set $B$. Then setting $\rho_t \coloneqq\frac{\d \rho_t}{\d \mea}$, $t \in [0,T]$, it holds that
	\begin{equation}\label{localest}
		\sup_{t \in (0,T)} \|\rho_t\|_{L^{\infty}}\le \|\rho_0\|_{L^{\infty}} e^{\int_0^T \| \div (v_t)^-\|_{L^{\infty}(B_t) }\, \d t}\,,
	\end{equation}
for any family $\{B_t\}_{t\in[0,T]}$  of Borel sets  such that $\rho_t=0$ $\mea$-a.e.\ in $B_t^c$ and the map $(x,t)\mapsto \nchi_{B_t}(x)$ is Borel.
\end{prop}

For the proof of Proposition \ref{prop:localest} we will need the following :
\begin{lemma}[Commutator estimate {\cite[Lemma 5.8]{AT}}]\label{commutatorlemma}
	Let $\X$ be an ${\rm RCD}(K,\infty)$ m.m.s., then there exists a positive constant $C=C(K)>0$ such that the following holds.  
	Let $v \in W^{1,2}_{C}(TX)$ with $\div(v)\in L^\infty(\mea)$,  then
	\begin{equation}\label{commuataor}
		\int \la \nabla h_t (f),v\ra g \d \mea + \int  f \div(h_t(g)v_t)   \d \mea \le C \left (\|\nabla v\|_{L^2(T^{\otimes 2}X)}+\|\div(v)\|_{L^\infty} \right ) \|f\|_{L^2\cap L^4}\|g\|_{L^2\cap L^4},
	\end{equation}
	for every $f,g \in L^2(\mea)\cap L^4(\mea)$ and every $t>0$. In particular for fixed $g$ the left hand side of \eqref{commuataor} defines a functional in $(L^2(\mea) \cap L^4(\mea))^*=L^2(\mea)+L^{4'}(\mea)$, denoted by $\mathscr{C}^t(g,v)$, which satisfies
	\begin{equation}\label{commutatornorm}
		\|\mathscr{C}^t(g,v)\|_{L^2(\mea)+L^{4'}(\mea)}\le  C \left (\|\nabla_{sym} v\|_{L^2(T^{\otimes 2}X)}+\|\div(v)\|_{L^\infty} \right ) \|g\|_{L^2\cap L^4}.
	\end{equation}
	Moreover it holds that
	\begin{equation}\label{vanishcommutator}
		\|\mathscr{C}^t(g,v)\|_{L^2(\mea)+L^{4'}(\mea)}\to 0,  \quad \quad \text{ as } t \to 0^+.
	\end{equation}
\end{lemma}

We can now pass to the proof of the local a priori estimate.
\begin{proof}[Proof of Proposition \ref{localest}]
	We start with the preliminary observation that,  combining the fact that $\mu_t={F_t}_*\mu_0$ with \eqref{eq:compr} and recalling that $B$ is bounded, we have
	\begin{equation}\label{eq:alreadybdd}
		\sup_{t \in [0,T]} \|\rho_t\|_{L^q(\mea)}<+\infty, \quad  \forall \,q\in[1,\infty].
	\end{equation}
	To conclude it is sufficient to prove that for every $p>1$ the function $[0,T]\ni t \mapsto \int (\rho_t)^p \d \mea$ is absolutely continuous and 
	\begin{equation}\label{gronwall}
		\frac{\d}{\d t} \int (\rho_t)^p \d \mea \le (p-1)\int (\rho_t)^p \div(v_t)^- \d \mea, \quad \text{ for a.e. } t\in (0,T).  
	\end{equation}
	Indeed \eqref{localest} would follow first applying Gronwall Lemma (noticing that $\rho_t=0$ $\mea$-a.e. outside $B_t$) and then letting $p \to +\infty.$ 
	
	So we fix $p>1.$ Pick  a sequence $s_n \downarrow 0$ and define $\rho_t^n:= h_{s_n}\rho_t.$ From the fact that $\rho_t $ is a solution of the continuity equation and the selfadjointedness of the heat flow, we obtain that for every $f \in \LIP_{bs}(\X),$ the function $t\mapsto \int f\rho_t^n \d \mea $ is absolutely continuous and
	\begin{equation}\label{pde}
		\frac{\d}{\d t}\int f \rho_t^n \d \mea = \int \la \nabla  h_{s_n}f,v_t\ra \rho_t \, \d \mea=\int f[\mathscr{C}^{s_n}(\rho_t,v_t)-\text{div}(\rho_t^n v_t)] \d \mea,  \quad \text{ for a.e. } t\in (0,T), 
	\end{equation}
	where $\mathscr{C}^{s_n}(\rho_t,v_t)$ is defined as in Lemma \ref{commuataor}. Set now $\eta_t^n:=\mathscr{C}^{s_n}(\rho_t,v_t)-\div(\rho_t^n v_t)$. From the Leibniz rule and the $L^\infty$-to Lipschitz regularization  of the heat flow  \eqref{eq:BE} we have that
	\[ \|\div(\rho_t^n v_t)\|_{L^2}\le \|\rho_t\|_{L^2} \|\div(v_t)\|_{L^\infty}+c(K)\frac{\|\rho_t\|_{L^\infty}}{\sqrt{s_n}}\||v_t|\|_{L^2}.\]
	This bound together with \eqref{commutatornorm}, \eqref{eq:alreadybdd} and the hypotheses on $v_t$, guarantees that $\eta_t^n \in L^1((0,T),L^2(\mea)+L^{4'}(\mea)).$ Denote by $V$ the Banach space $L^2(\mea)+L^{4'}(\mea)$ and observe that  $L^2\cap L^4=V^*$. Then \eqref{pde} can be restated as:  for a weakly*-dense set of $\phi \in V^*$ the function $[0,T] \ni t \mapsto \phi(\rho_t^n) $ is absolutely continuous and $\frac{\d}{\d t}\phi(\rho_t^n)=\phi(\eta_t^n).$ It follows (see e.g. Remark 4.9 in \cite{AT}) that  $\rho_t^n$ is absolutely continuous in $L^1((0,T),V)$ and  strongly differentiable a.e. with $\frac{\d}{\d t}\rho_t^n=\eta_t^n.$
	
	Pick a convex function $\beta : [0+\infty) \to [0+\infty)$ such that $\beta(t)=t^p$ for every $t \le 2\sup_{t \in (0,T)} \|\rho_t\|_{L^\infty}<+\infty$ and such that $\beta'$ is globally bounded. In particular from the maximum principle for the heat flow we have that $\beta(\rho_t^n)=(\rho_t^n)^p$ for every $t$ and $n$. Moreover, since $\rho_t$ are uniformly bounded in $L^q(\mea)$ for every $1\le q \le \infty$, from the contractivity of the heat flow we have also that $\rho_t^n$ are bounded  in $L^q(\mea)$ for every $1\le q \le \infty$,  uniformly in $t$ and $n$. Finally, observing that $\beta'(t)/t$ is globally bounded, we deduce that $\beta'(\rho_t^n)$ are again bounded  in $L^q(\mea)$ for every $1\le q <\infty$,  uniformly in $t$ and $n$. 
	
	Observe now that from the convexity of $\beta$ we have that
	\begin{equation}\label{convexrho}
		\int \beta(\rho_t^n)-\beta(\rho_s^n) \d \mea \le 
		\int \beta'(\rho_t^n) (\rho_t^n-\rho_s^n) \d \mea, \quad \forall t,s \in[0,T].   
	\end{equation}
	This in turn gives
	\[ \int \beta(\rho_t^n)-\beta(\rho_s^n) \d \mea \le \sup_{t \in [0,T]}\|\beta'(\rho_t^n)\|_{L^2\cap L^4(\mea )} \|\rho_t^n-\rho_s^n \|_{L^2+L^{4'}}
	\le \sup_{t \in [0,T]}\|\beta'(\rho_t^n)\|_{L^2\cap L^4} \int_s^t \|\eta_r^n\|_{L^2+L^{4'}} \d r, \]
	for every $t,s \in [0,T],$ with $s\le t.$
	Hence the function $\int \beta(\rho_t^n)\d \mea$ is absolutely continuous in $[0,T]$ and from \eqref{convexrho} we deduce that
	\[ \frac{\d}{\d t}\int \beta(\rho_t^n)\d \mea\le \int \beta'(\rho_t^n) \eta_t^n \d \mea, \quad \text{ for a.e. } t\in (0,T). \]
	Then from the definition of $\eta_t^n$ and $\beta$ and integrating by parts we obtain
	\begin{align*} 
		\frac{\d}{\d t}\int \beta(\rho_t^n)\d \mea &
		\le- \int [\beta'(\rho_t^n)\rho_t^n-\beta(\rho_t^n)]\div(v_t) \d \mea+ \int \mathscr{C}^{s_n}(\rho_t,v_t) \beta'(\rho_t^n) \d \mea\\
		&\le  (p-1)\int (\rho_t^n)^p \div(v_t)^- \d \mea+ p\int \mathscr{C}^{s_n}(\rho_t,v_t) (\rho_t^n)^{p-1} \d \mea,
	\end{align*}
	for a.e. $t \in (0,T).$
	Observe  now that combining  \eqref{commutatornorm} with \eqref{vanishcommutator}, an application of dominated convergence gives that $ \int_0^T\|\mathscr{C}^{s_n}(\rho_t,v_t)\|_{L^2+L^{4'}} \to 0 $ as $s_n \to 0$.  Then, recalling that $\div(v_t)^- \in L^\infty((0,T),L^\infty(\mea ))$ and \eqref{eq:alreadybdd},  we can let $s_n \to 0$ in the above and obtain at once the absolute continuity of $\int (\rho_t)^p  \d \mea$ together with \eqref{gronwall}.
\end{proof}

\subsection{From almost outer functional cone to almost outer metric cone}\label{sec:almcone}

Our aim in this section is to prove the following.
\begin{theorem}\label{thm:almconev1}
	For every  $\eps\in(0,1/3)$, $R_0>0$, $\gamma>\frac12\frac{N-2}{N-1}$,  $N\in [2,\infty)$, $L>0$  there exists $0<\delta=\delta(\eps,\gamma,N,R_0,L)$ such that the following holds. Let  $(\X,\sfd,\mea,x_0)$ be a pointed  ${\rm RCD}(-\delta,N)$ m.m.s. with $\mea(B_1(x_0))\in(\eps,\eps^{-1})$. Let  $U\subset \X$ be open with $U^c\subset B_{R_0}(x_0)$ and $v\in D(\bd,U)\cap C(U)$ be positive and  such that $\limsup_{x\to \partial U}v(x)\le 1$,  $\Delta  v=N|\nabla \sqrt {2v}|^2$, $v \ge 1+\eps$ in $B_{R_0}(x_0)^c\neq\emptyset$  and $\||\nabla \sqrt v|\|_{L^\infty(U)}\le L$. Suppose furthermore that
	\begin{equation}\label{eq:small int}
		\int_{U} \frac{1}{v^{N-2}} \left |\nabla |\nabla \sqrt v|^\gamma \right |^2\, \d \mea<\delta,
	\end{equation}
where it is intended that $v^{N-2}\equiv 1$ in the case $N=2$.  

	Then there exists a pointed ${\rm RCD(0,N)}$ space $(\X',\sfd',\mea',x')$ such that 
	\[
	\sfd_{pmGH}((\X,\sfd,\mea,x_0),(\X',\sfd',\mea',x'))<\eps
	\]
	and $(\X',\sfd',\mea',x')$ is a truncated cone outside a compact set $K\subset B_{2R_0}(x'),$ i.e. there exists an  ${\rm RCD(0,N)}$ $N$-cone $Y$, with vertex $O_Y$, over an ${\rm RCD}(N-2,N-1)$ space $Z$, and a measure preserving local isometry $T: \X'\setminus K\to Y\setminus \bar B_{r}(O_Y)$, for some $r>L^{-1}.$
\end{theorem}
Observe that \eqref{eq:small int} makes sense thanks to Corollary \ref{cor:div estimate}.

Notice also that the assumption $u \ge 1+\eps$ in $B_{R_0}(x_0)^c$ is necessary as the function $v\equiv 1$ shows.
We will also prove another version of the above result, which is Theorem \ref{thm:almconev2}. Before passing to its statement we need some definitions and notations.

For any couple of compact metric spaces $(X_1,\sfd_1),\, (X_2,\sfd_2)$ we define their \emph{Gromov Hausdorff distance} as
\[
\begin{split}
\sfd_{GH}((X_1,\sfd_2),(\X_2,\sfd_2))\coloneqq\inf \{&\eps>0 : \exists f: X_1\to X_2 \ \text{ such that} \ \text{$f(\X_1)$ is $\eps$-dense in $\X_2$}\\
&\text{ and } |\sfd_1(x,y)-\sfd_2(f(x),f(y))|\leq\eps, \, \forall x,y \in X_1\}.
\end{split}
\]
For a sequence of compact metric spaces $(X_n,\sfd_n)$ converging to $(\X,\sfd)$ in the GH-sense we say that a sequence of maps $f_n : \X_n \to \X$ realizes the convergence if there exists a seuqence $\eps_n \to 0$ such that $f(\X_n)$ is $\eps_n$-dense in $\X$ and $|\sfd_n(x,y)-\sfd(f_n(x),f_n(y))|\leq\eps_n$,  $\forall x,y \in X_n$.

Moreover we recall the notion of \emph{Sturm distance} $\mathbb{D}$  for compact m.m.s. which was first introduced in \cite{sturm1} in the case of renormalized spaces (see also \cite{GMSconv}):
\[
\mathbb{D}((X_1,\sfd_2,\mea_2),(\X_2,\sfd_2,\mea_1))\coloneqq \inf \left | \log \frac{\mea_1(\X_1)}{\mea_2(\X_2)}\right |+W_2\left({\iota_1}_*\tilde \mea_1,{\iota_2}_*\tilde \mea_2\right),
\]
where $\tilde \mea_i=\frac{\mea_i}{\mea(\X_i)} ,<$ $i=1,2$ and the infimum is taken among all the complete and separable metric spaces $(Y,\sfd)$ and isometric embeddings $X_i \overset{\iota_i}{\hookrightarrow}Y$, $i=1,2.$
Observe that $\mathbb{D}$ is well defined since we are  assuming that $\supp (\mea_i)=\X_i.$

It is important to recall that an upper bound on $\mathbb{D}$ does not imply in general an upper bound on  $\sfd_{GH}$, indeed this holds only if we restrict ourselves to a family of uniformly doubling metric measure spaces. However we will need to apply $\mathbb{D}$ to spaces which are not a priori uniformly doubling, for this reason we will need to work both with $\sfd_{GH}$ and with $\mathbb{D}.$

Finally, for any $(\X,\sfd)$  complete metric space  and  $A\subset \X$ we define the \emph{intrinsic metric on $A$} to be the distance function 
$\sfd^{A}: A \to [0,+\infty]$ defined by 
\[
\sfd^A(x,y)\coloneqq \inf_\gamma \, L(\gamma),\quad \forall x,y \in  A,
\]
where the infimum is taken among all curves $\gamma \in AC([0,1],\X)$ with values in $A$ and such that $\gamma(0)=x$, $\gamma(1)=y.$

If we also assume that $A$ is relatively compact, then for every $x,y \in A$ such that $\sfd(x,y)<+\infty$ there exists an absolutely continuous curve $\gamma$ with values in $\bar A$ such that $\sfd^A(x,y)=L(\gamma).$ 

We are ready to state the second version of Theorem \ref{thm:almconev1}, which is more in the spirit of the `volume annulus implies metric annulus' theorem  of Cheeger and Colding, which is also based on a functional formulation similar to the present one (see \cite[Theorem 4.85]{CC}).

\begin{theorem}\label{thm:almconev2}
	For every $\eps\in(0,1/3)$, $R_0>0$,  $\gamma> \frac 12\frac{N-2}{N-1}$, $N\in [2,\infty)$, $L>0$, $\eta\in(0,\eps^{-1})$  there exists $0<\delta=\delta(\eps,\gamma,N,R_0,L,\eta )$ such that, given $\X$, $U $ and $v$ as in Theorem \ref{thm:almconev1},
	 there exists an ${\rm RCD(0,N)}$ $N$-cone $(Y,\sfd_Y,\mea_Y)$ with vertex $O_Y$, over an ${\rm RCD}(N-2,N-1)$ space $Z$, and a constant $\lambda\in (0,L)$ such that the following holds. 
	 
	 For every $1+\eps+\eta<t_1\le t_2<\eps^{-1}$ satisfying $\{ \sqrt v\le t_2+2\eta\}\subset B_{R_0}(x_0)$, it holds
	\begin{equation}\label{eq:closeannuli1}
	\sfd_{GH}\left((\{t_1\le \sqrt v\le t_2\}, \sfd_{\X}^{\eta}),\left(\{t_1\le \lambda \sfd_{O_Y}\le t_2\}, \sfd_{Y}^{\eta}\right)  \right)<\eps,
	\end{equation}
where $\sfd_{O_Y}\coloneqq\sfd_Y(.,O_Y)$ and $\sfd_{\X}^{\eta}$ and $\sfd_{Y}^{\eta}$ denote the  intrinsic metrics on  $\{t_1-\eta< \sqrt v< t_2+\eta\}$ and on $\{t_1-\eta<\lambda \sfd_{O_Y}< t_2+\eta\}$  (see definition above). Moreover, provided that $t_1+\eps<t_2$, 
\begin{equation}\label{eq:closeannuli2}
	\mathbb{D}\left(\big(\{t_1\le \sqrt v\le t_2\}, \sfd_{\X}^{\eps},\mea\restr{\{t_1\le\sqrt v\le t_2\}}\big),\big(\{t_1\le \lambda \sfd_{O_Y}\le t_2\}, \sfd_{Y}^{\eps},\mea_Y\restr{\{t_1\le \lambda \sfd_{O_Y}\le t_2\}}\big)  \right)<\eps.
\end{equation}
Moreover the cone $Y$ can be taken so that the conclusion of Theorem \ref{thm:almconev1} holds (with the same $\eps$, $R_0$  and $L$) with  $Y$ and for some ${\rm RCD}(0,N)$ space $\X'$.
\end{theorem}
We point out that in general we cannot say anything better than $\lambda \le L$. This is immediately seen by taking $v=L^2|x|^2$ in $\rr^n$ and $U= \rr^n\setminus \bar B_{1/L}(0).$

It is important to notice  that the information in \eqref{eq:closeannuli1} is not contained in \eqref{eq:closeannuli2}, indeed,  as said above, it is not clear to us if, fixed $\eps,\gamma,N,R_0,L$ as in Theorem \ref{thm:almconev2}, the metric measure spaces $(\{\sqrt{t_1}\le \sqrt v\le \sqrt{t_2}\}, \sfd_{\X}^{\eta},\mea\restr{\{t_1\le\sqrt v\le t_2\}})$, for arbitrary $v,t_1,t_2$ as in  the hypotheses, satisfy some uniform doubling condition.

For the proof of Theorem \ref{thm:almconev2} we will need the following elementary lemma. The proof is a direct consequence of the definition of distance in a cone and will be omitted.
\begin{lemma}\label{lem:conelemma1}
		Let $(Y,\sfd)$ be an Euclidean cone of vertex $O_Y$ and for any $0<a<b$ let $\sfd_{a,b}$ be the intrinsic metric on $\{a<d(.,O_Y)<b\}$. Then for every $0<\eps<a<b$ it holds
	\[
	\sfd_{a-\eps,b+\eps}\le \sfd_{a,b}\le \frac{a}{a-\eps}\sfd_{a-\eps,b+\eps}, \quad \text{ in $\{a<d(.,O_Y)<b\}$.}
	\]
	Moreover for any two sequences $(a_n),(b_n)$ such that $a_n \to a$, $b_n \to b$ it holds
	\[
	(\{a_n\le d(.,O_Y)\le b_n\},\sfd_{a_n-\eps,b_n+\eps})\overset{GH}{\rightarrow}	(\{a\le d(.,O_Y)\le b\},\sfd_{a-\eps,b+\eps})
	\]
	and the map $f_n(t,z)\coloneqq \left(\frac{(t-a_n)(b-a)}{(b_n-a_n)}+a,z \right) $ (in polar coordinates) realizes such convergence.
\end{lemma}

Finally  for proof of Theorem \ref{thm:almconev2} we will need the following result (see below for the definition of mGH-convergence):
\begin{prop}[{\cite[Prop. 3.30]{GMSconv}}]\label{prop:GH to D}
	Let $(X_n,\sfd_n,\mea_n)$ be compact m.m.\ spaces mGH-converging to a compact m.m.s. $(\X_\infty,\sfd_\infty,\mea_\infty)$. Then $$\mathbb{D}((X_n,\sfd_n,\mea_n),(\X_\infty,\sfd_\infty,\mea_\infty))\to0.$$
\end{prop}

\begin{definition}[\emph{measure Gromov Hausdorff convergence}]\label{def:mgh}
	We say that the sequence  $(\X_n,\sfd_n,\mea_n)$ of m.m.s. \emph{measure Gromov Hausdorff}-converges (mGH-converges in short) to  a compact m.m.s.  $(\X_\infty,\sfd_\infty,\mea_\infty)$, if 
	there are  Borel maps $f_n:X_n\to X_\infty$ such that 
	\begin{itemize}
		\item[1)] $\sup_{x,y\in B_{R_n}(x_n)}|\sfd_n(x,y)-\sfd_\infty(f_n(x),f_n(y))|\leq\eps_n$,
		\item[2)] $f_n(\X_n)$ is $\eps_n$-dense in $\X_\infty$,
		\item[3)] ${f_n}_*\mea_n\rightharpoonup \mea_\infty$ in duality with $C(\X_\infty)$.
	\end{itemize}
\end{definition}
Notice that if $(\X_n,\sfd_n,\mea_n) \to (\X_\infty,\sfd_\infty,\mea_\infty)$ in the mGH-sense, then $\sfd_{GH}((\X_n,\sfd_n,) ,(\X_\infty,\sfd_\infty))\to0.$

\begin{proof}[Proof of Theorem \ref{thm:almconev1} and Theorem \ref{thm:almconev2}]
	
	The proof of Theorem \ref{thm:almconev1} is essentially the same as  Theorem \ref{thm:almconev2}, except that it stops earlier. For this reason we will prove both theorems together. The reader interested only in the proof of the first result can ignore the second half of the argument.

\bigskip
	\textbf{Proof of Theorem \ref{thm:almconev1}:}
	
	We argue by contradiction. Suppose that there exist numbers $\eps\in(0,1/3),N\in [2,\infty)$, $R_0>0$, $\gamma \ge \frac12\frac{N-2}{N-1}$, $L>0$,  a sequence $\delta_n \to 0,$ a sequence
	$(X_n,\sfd_n,\mea_n,x_n)$  of  ${\rm RCD}(-\delta_n,N)$ m.m.s., a sequence of open sets $U_n\subset X_n$, with $U_n^c \subset B_{R_0}(x_n)$, functions $v_n\in D(\bd,U_n)$ satisfying $\Delta  v_n=2N|\nabla \sqrt {v_n}|^2 $  such that:
	\begin{enumerate}
		\item[$a)$]  $\limsup_{x\to \partial U_n}v_n(x)\le 1$,  $v_n \ge 1+\eps$ in $B_{R_0}(x_0)^c\neq \emptyset$, $\,\||\nabla \sqrt v_n|\|_{L^\infty(U_n)}\le L,$ 
		\item[$b)$] $\mea_n(B_1(x_n))\in(\eps,\eps^{-1})$, 
		\item[$c)$] \eqref{eq:small int} holds (with  $v_n$, $\mea_n$, $\gamma $ and $\delta=\delta_n$),
		\item[$d)$] for every $n$ the conclusion of Theorem \ref{thm:almconev1} does not hold.
	\end{enumerate}
	
	We first observe that, since $v_n \ge 1+\eps$ in $B_{R_0}(x_0)^c$, removing the set $\{v_n\le1\}$ from $U_n$ does not effect neither the hypotheses of the theorems nor their conclusions, therefore it is not restrictive to assume that $v_n>1$ in $U_n$.
	
	Moreover by compactness, up to passing to a non relabelled subsequence, we can assume that the p.m.m. spaces $(X_n,\sfd_n,\mea_n,x_n)$  pmGH-converge  to a  to an ${\rm RCD(0,N)}$ pointed m.m.s.  $(X_\infty,\sfd_\infty,\mea_\infty,x_\infty).$ 
	
	Passing to the extrinsic approach, we consider a proper metric space $(Y,\sfd_Y)$ that realizes such convergence, in particular we identify the metric spaces $X_n$ and $\X_\infty$ with the corresponding subsets of $Y$ such that $\sfd_Y(x_n,x_\infty)\to 0$, $\mea_n \rightharpoonup \mea_\infty$ in duality with $C_{bs}(Y)$ and $\sfd^Y_H(B^{\X_n}_R(x_n),B^{\X_\infty}_R(x_\infty))\to 0$ for every $R>0.$ In particular for every $x \in \X_\infty$ there exists a sequence $y_n \in \X_n$ such that $\sfd_{Y}(x,y_n)\to0$ and conversely for every $R>0$ and every sequence $y_n  \in B^{\X_n}_{R}(x_n)$ there exists a subsequence converging to a point $x \in \X_\infty$. This two facts will be used repeatedly in the proof without further notice.
	
Define the compact sets $K_n=U_n^c\subset \X_n$ and observe that, since $K_n\subset B_{R_0}(x_n)$, they are all contained in a common ball in $Y$ centered at $x_\infty.$
	Hence from the metric version of Blaschke’s theorem (see \cite[Theorem 7.3.8]{burago}) there exists a compact set $K_\infty\subset \X_\infty$  such that, up to a subsequence, $\sfd^Y_H(K_\infty ,K_n)\to 0.$
	
	Define the open  (in the topology of $\X_\infty$) set  $U_\infty\coloneqq X_\infty \setminus K_\infty$ and for every $r>0$, define the open sets $U_n^{<r}=\{x \in X_n \ : \ \sfd_n(x,K_n)<r\}$ and $U^{<r}_\infty=\{x \in X_\infty \ : \ \sfd_\infty(x, K_\infty)<r\}$. Analogously we define the sets $U_n^{>r},U_n^{\le r},U_n^{\ge r}$ and the corresponding ones for $n=\infty.$

	From the assumptions  $\limsup_{x\to \partial U_n}v_n(x)\le1 $ and $\||\nabla \sqrt v_n|\|_{L^\infty(U_n)}\le L$, applying Proposition \ref{prop:sob to lip} it follows that
	\begin{equation}\label{eq:safe}
		v_n \le (1+RL)^2, \text{ in $B_R(x_n)\cap U_n$}, \quad \{v_n\le 1+\eps/4\}  \text{ in } U_n^{\le 4\rho}\cap U_n,
	\end{equation}
for every $R>0$ and for some small constant $0<\rho=\rho(\eps,L)<\eps$, independent of $n$.

	For every $n$ and every  $k \in \mathbb N$ with $k\ge R_0+100$, thanks to Proposition \ref{prop:goodcutoff}, there exists a cut off function  $\eta_k^n \in \test(X)$, $0\le\eta_k^n\le 1$, such that $\supp \eta_n^k\subset U_n^{>\rho/2}\cap B_{k+2}(x_n)$, $\eta_n^k=1$ on $ U_n^{\ge \rho}\cap B_{k+1}(x_n)$ and $\Lip \eta_n^k+|\Delta \eta_n^k|\le C$, for some constant $C$ depending only on $\eps,N,L$. Observe  also that we can choose $\eta_n^k$ so that $\eta_n^k=\eta_n^{k+1}$ in $B_{k+1}(x_n)$, for every $k$. 
	
	Define the functions  $v_n^k\coloneqq v_n\eta_n^k, \, \tilde v_n^k\coloneqq \sqrt{v_n}\eta_n^k$ and observe that  from \eqref{eq:safe} and the assumption $\||\nabla \sqrt v_n|\|_{L^\infty(U_n)}\le L$ they  are equi-Lipschitz, equi-bounded in $n$ and all supported on $B_{k+2}(x_n).$ Hence by Ascoli-Arzelà (see Prop. \ref{prop:Ascoli}), up to a subsequence, as $n \to +\infty$ they converge uniformly to   functions   $v_\infty^k,\tilde v_\infty^k\in C(X_\infty)$ with support in $\bar B_{k+2}(x_\infty)$.
	
	From Proposition \ref{prop:l2prop} it follows that  $v_n^k,\tilde v_n^k$ converge also strongly in $L^2$ respectively to $v_\infty^k,\tilde v_\infty^k$.
It is clear from the construction and the uniform convergence that $v_\infty^k=v_\infty^{k+1}$ on  $B_{k}(x_\infty)$. Therefore the assignment $v_\infty\coloneqq v_\infty^k$ in $B_{k}(x_\infty)$ for every $k$,  well defines a function $v_\infty \in C(\X_\infty).$ Analogously we can define $\tilde v_\infty \in C(\X_\infty)$  and we observe that $\tilde v_\infty=\sqrt{v_\infty}$ in $U_\infty^{\ge 2\rho}$. 

	We make the following  two claims:
	\begin{enumerate}
		\item[({\bf A})] $v_\infty \le 1+\eps/4$ in $U_\infty^{\le 3\rho}$,
		\item[({\bf B})]$\{v_\infty>1+\eps/4 \}\neq \emptyset$, $v_\infty \in D(\bd)$   and (up to multiplying $v_\infty$ by a positive constant $C_0$) it holds that $ \Delta v_\infty=N$, $|\nabla v_\infty|^2=2v_\infty$ $\mea$-a.e. in $\{v_\infty>1+\eps/4 \}$.
	\end{enumerate} 
We start with  claim ({\bf A}).
It is clearly enough prove that
\[
	v_\infty^k\le 1+\eps/4, \quad \text{ in $ {U_\infty^{\le 3\rho}}$}
\]
for any $k$.
Pick  any $y \in U_\infty^{\le  3\rho}$, then there exists a sequence $y_n \in X_n$ such that $y_n\to y$ in $Y$ and by uniform convergence $v_n^k(y_n)\to v_\infty^k(y)$. Moreover it must hold that $\sfd_Y(y_n,K_n)<4\rho$ if $n$ is big enough. If $y_n \notin U_n$, then $v_n^k(y_n)=0$ by construction. If instead $y_n \in U_n$ from the second in \eqref{eq:safe} and the fact that $\eta_k^n\le 1$ we deduce that $v_n^k(y_n)\le v_n(y_n)\le 1+\eps/4.$ Combining these two observations we get claim ({\bf A}).

We pass to the proof of claim ({\bf B}). It is easy to check, since $v_n\ge 1$  in $U_n$ (recall the observation made at the beginning of the proof), that  $\sup_n \|\Delta v_{n}^k\|_{L^\infty(X_n)}<+\infty$ and $\sup_n \|\Delta \tilde v_{n}^k\|_{L^\infty(X_n)}<+\infty.$  Moreover by Bishop-Gromov inequality we have $\sup_n\mea_n(B_R(x_n))<+\infty$, for every $R\ge 1$. Therefore $\sup_n \|\Delta v_{n}^k\|_{L^2(X_n)},$ $\sup_n\|\Delta \tilde v_{n}^k\|_{L^2(X_n)}<+\infty$ and applying Theorem \ref{thm:l2stab} we deduce that $v_\infty^k,\tilde v_\infty^k \in D(\Delta) $, that $\Delta v_n^k$  converges to $\Delta v_\infty^k$ weakly in $L^2$ and that $|\nabla \tilde v_n^k|\to |\nabla \tilde v_\infty^k|$ strongly in $L^2$.  From the locality of Laplacian follows that $v_\infty \in D(\bd).$
Additionally, since $L'\coloneqq\sup_n \||\nabla \tilde v_n^k|\|_{L^\infty}<+\infty$, applying a) of Proposition \ref{prop:l2prop} (with $\phi(t)=(t\wedge L')^{\alpha}$) we also deduce that $|	 \nabla \tilde v_n^k|^{\alpha}\to|\nabla \tilde v_\infty^k|^{\alpha} $ strongly in $L^2,$ for every $\alpha>0.$ We make the intermediate claim that
	\begin{equation}\label{eq:harmoniclimit}
		\Delta v_\infty= N |\nabla \sqrt {2v_\infty}|^2, \quad \mea\text{-a.e. in $U_{\infty}^{>2\rho}\cap B_{k}(x_\infty)$}.
	\end{equation}
In particular from Corollary \ref{cor:div estimate} this implies that $|\nabla \sqrt{v_\infty}|^{\gamma} \in \W_{\loc}(U_{\infty}^{>2\rho}\cap B_{k}(x_\infty))$.
	To prove \eqref{eq:harmoniclimit} pick any $\phi \in \LIP_c(U_{\infty}^{>2\rho}\cap B^{\X_\infty}_{k}(x_\infty))$. Consider also a function $\eta\in \LIP(Y)$ such that $\eta\equiv 1$ in $\supp \, \phi$, $\sfd_Y(\supp \eta, K_\infty)>2\rho$ and $\supp \eta \subset B^Y_k(x_\infty)$.
	Moreover, since $\sfd_H^Y(K_n,K_\infty )\to 0,$ for $n$ big enough  we have $\{y \ : \ \sfd_Y(y,K_\infty )>2\rho \}\subset \{y \ : \ \sfd_Y(y,K_n)>\rho \}$ and analogously, since $x_n\to x_\infty$ in $Y$, for $n$ big enough $B^Y_{k}(x_\infty)\subset B^Y_{k+1}(x_n).$ Therefore $\supp \eta \cap X_n \subset U_n^{>\rho}\cap B^{\X_n}_{k+1}(x_n)$ for $n$ big enough. We now extend $\phi$ to a function $\phi'\in \LIP(Y)$ and define $\bar \phi=\eta \phi' \in \LIP_{bs}(Y).$ Since by the locality of the Laplacian and gradient we have $\Delta v_n^k=\Delta v_n=N|\nabla \sqrt{2 v_n}|^2=2N|\nabla \tilde  v_n^k|^2$ $\mea_n$-a.e. in $\supp \bar \phi$, we can compute 
	\begin{align*}
		&\int \phi \Delta v_\infty^k \d \mea_\infty=\int \bar \phi \Delta v_\infty^k \d \mea_\infty =\lim_n \int \bar \phi \Delta v_n^k \d \mea_n=\lim_n \int  \bar \phi2N|\nabla \tilde  v_n^k|^2\d \mea_n\\
		&=\int \bar \phi 2 N |\nabla \tilde v_\infty^k|^2\d \mea_\infty=\int \bar \phi 2 N |\nabla \sqrt v_\infty|^2\d \mea_\infty.
	\end{align*}
This and the locality of the Laplacian prove \eqref{eq:harmoniclimit}.

For every $n$ and every $k$ as above we consider a cut-off function $\xi_{n}^k\in \LIP(\X_n)$ analogous to $\eta_n^k$ but with smaller support, more precisely we require that  $0\le\xi_k^n\le 1$,  $\supp \xi_n^k\subset U_n^{>\rho}\cap B_{k+2}(x_n)$, $\xi_n^k=1$ on $ U_n^{\ge 2\rho}\cap B_{k+1}(x_n)$ and $\Lip \xi_n^k\le C'$, for some constant $C'$ depending only on $\eps,N,L$. Up to a subsequence, from Ascoli-Arzelà we have that $\xi_n^k \to \xi_\infty^k$ uniformly, for some $\xi_\infty^k\in \LIP(\X_\infty)$ satisfying $\xi_\infty^k=1$ in $U_\infty^{>2\rho}\cap B^{\X_\infty}_{k}(x_\infty).$ In particular the same convergence holds also strongly in $L^2$. 

We set $w_{n,k}\coloneqq \xi_{n}^k|\nabla \sqrt{v_n}|^{\gamma}\in \W(\X_n)$ and $w_{\infty,k}\coloneqq \xi_{\infty}^k|\nabla \tilde v_\infty^k|^{\gamma} \in \W(\X_\infty)$ and observe that by construction and the locality of the gradient $w_{n,k}=\xi_{n}^k|\nabla \tilde v_n^k|^\gamma$ $\mea_n$-a.e.. In particular, since we proved that $|\nabla \tilde v_n^k|^\gamma \to |\nabla \tilde v_\infty^k|^{\gamma} $ strongly in $L^2$ and recalling $\sup_n \||\nabla \tilde v_n^k|\|_{L^\infty}<+\infty$, we have from Proposition \ref{prop:l2prop}  that $w_{n,k}\to w_{\infty,k}$ strongly in $L^2$.

Combining \eqref{eq:small int} with the first in \eqref{eq:safe} and $\Lip \xi_{n,k}\le C'$ we deduce that $\sup_n \||\nabla w_{n,k}|\|_{L^2(\mea_n)}<+\infty$. We now apply Lemma \ref{lem:lsc energy local} with the open set $A=\{\sfd_Y({K_\infty },.)>3\rho \}\cap B^Y_{k}(x_\infty)$  that, combined with the observation that $A\cap X_n \subset   U_n^{\ge 2\rho}\cap B_{k+1}(x_n)$ for $n$ big enough, gives
	\begin{align*}
		&\int_{U_{\infty}^{>3\rho}\cap B_{k}(x_\infty)} |\nabla |\nabla \tilde v_\infty^k|^{\gamma} |^2\, \d \mea_\infty \le \liminf_n \int_{U_n^{\ge 2\rho}\cap B_{k+1}(x_n)}  
		|\nabla |\nabla \sqrt{ v_n}|^{\gamma} |^2\, \d \mea_n\\
		&\overset{\eqref{eq:safe}}{\le} (1+RL)^{2N-4}\liminf_n \int_{U_n}  
		v_n^{2-N}|\nabla |\nabla \sqrt{ v_n}|^{\gamma} |^2\, \d \mea_n\overset{\eqref{eq:small int}}{=}0.
	\end{align*}
	Therefore, from the locality of the gradient and the arbitrariness of $k$ we obtain that $|\nabla |\nabla \sqrt{ v_\infty}|^{\gamma} |=0 $ $\mea_\infty$-a.e. in $U_{\infty}^{>3\rho}.$ 
	Consider now the open  set $\{v_\infty >1+\eps/4\}\subset U_\infty^{>3\rho}$. Observe that from the assumption $v_n \ge 1+\eps$ in $B_{R_0}(x_n)^c\neq \emptyset $  and \eqref{eq:safe} we deduce that for every $n$ there exists $y_n \in \X_n\cap B_{2R_0}(x_n)\cap U_n^{>4\rho}$ such that $v_n^k(y_n)=v_n(y_n)\ge 1+\eps$ for every $k$, therefore by compactness and uniform convergence we deduce that $\{v_\infty >1+\eps/4\}\neq \emptyset$. From ({\bf A}) it holds $\partial \{v_\infty >1+\eps/4\}=\{v_\infty=1+\eps/4\}$, in particular since $v_\infty ^{(2-N)/2}$ ( $\ln (v_\infty^{-1/2})$ if $N=2$) is harmonic in $U_\infty^{>2\rho}$ (recall \eqref{eq:harmoniclimit}), from the maximum principle it follows that  the connected components of $\{v_\infty >1+\eps/4\}$ are unbounded. Let $U'$ be one of these components. It follows that $|\nabla \sqrt{v_\infty}|\equiv C$ $\mea$-a.e. in $U'$ for some constant $C$, that must be positive. Indeed if $C=0$, we would have that $v_\infty$ is constant in $U'$, but since $\partial U'\subset \{v_\infty=1+\eps/4\}$, $v_\infty$ should be constantly equal to $1+\eps/4$, which contradicts $U'\subset \{v_\infty >1+\eps/4\}.$ Finally the assumption $v_n \ge 1+\eps$ in $B_{R_0}(x_n)^c$  ensures that $\X_\infty\setminus U'\subset B_{2R_0}(x_\infty)$.  It follows that the function $(2C^2)^{-1}v_\infty\restr{U'}$ satisfies the hypotheses of Theorem \ref{thm:functional cone} with $U=U'$.
	In particular $\X_\infty$ has Euclidean volume growth and from Corollary \ref{cor:ends} it has one end, from which we deduce that $\{v_\infty >1+\eps/4\}$ is connected. Therefore repeating the above argument for $U'=\{v_\infty >1+\eps/4\}$ proves claim ({\bf B}) with $C_0\coloneqq(2C^2)^{-1}$.

	Combining ({\bf A}) and ({\bf B}), from Theorem \ref{thm:functional cone} we deduce the existence of an ${\rm RCD}(0,N)$ $N$-cone $(Y',\sfd_{Y'},\mea_{Y'})$ with vertex $O_{Y'}$ and a bijective measure preserving local isometry $T:  \{1+\eps/4<v_\infty \}\to \{ \ r<\sfd_{Y'}(O_{Y'},.) \}  $, for some $r>0$ which also satisfies (recall \eqref{eq:explicit})
	\begin{equation}\label{eq:vgood}
		\sqrt{v_\infty}(x)=\lambda \sfd_{Y'}(O_{Y'},T(x)), \quad \text{ for every }x \in \{	v_\infty > 1+\eps/4\},
	\end{equation}
where $\lambda\coloneqq (2C_0)^{-1/2}$ ($C_0$ being the constant in ({\bf B})).
	We claim that $\lambda\le \sup_n \||\nabla \sqrt{v_n}|\|_{L^\infty}$, which in particular gives that $r\ge\lambda^{-1}(1+\eps/4)\ge L^{-1}.$
Indeed from \eqref{eq:vgood}, the fact that $Y'$ is geodesic and the fact that $T$ is a local isometry we deduce that for every $x \in \{	v_\infty > 1+\eps/4\}$ and every $r'>0$ small enough, there exists $y \in B^{\X_\infty}_{r'}(x)$ such that $|\sqrt{v_\infty(x)}-\sqrt{v_\infty(y)}|=\lambda \sfd_\infty(x,y),$. The claim now follows from uniform convergence and the Sobolev to Lipschitz property.
	 Since, as observed above, $\X_\infty\setminus\{1+\eps/4<v_\infty \}^c\subset B_{2R_0}(x_\infty)$, the conclusion of Theorem \ref{thm:almconev1} holds for every sufficiently large $n$, which contradicts item $d)$ above.
	\textbf{This concludes the proof of Theorem \ref{thm:almconev1}}.
	
	\bigskip
	\textbf{Proof of Theorem \ref{thm:almconev2}:}
	We argue by contradiction exactly as above, except that we substitute the assumption $d)$ with the following:
	\begin{enumerate}
		\item[$d')$] there exist a number $\eta>0$ and two sequences $(t_1^n),(t_2^n)\subset (1+\eps+\eta ,\eps^{-1})$ with $t_1^n\le t_2^n$, such that $\{\sqrt v_n\le t_2^n+2\eta\}\subset B_{R_0}(x_n)$ and  the conclusion of Theorem \ref{thm:almconev2} is false  with $t_1^n,t_2^n$,  for every $n$.
	\end{enumerate}
	Since assumption $d)$ was not used until the very end of the proof of Theorem \ref{thm:almconev1} above, we can, and will, repeat all the first part of the proof up to this point together with all the constructions and objects introduced along the argument.
	
	 Up to passing to a subsequence we can  assume that $t_i^n \to t_i^\infty \in [1+\eps+\eta,\eps^{-1}]$, $i=1,2.$
	
	It will be useful later to remark that

	\begin{equation}\label{eq:vn}
		\sqrt{v_n}=\tilde v_n^k\, \quad \text{ in } \{1+\eps/2\le \sqrt{v_n}\le t_2^n+2\eta\}, \text{   for every $k,$}
	\end{equation}
which follows from the  second in \eqref{eq:safe} and the assumption $\{\sqrt{v_n}\le t_2^n+2\eta\}\subset B_{R_0}(x_n).$
	
	We claim that
	\begin{equation}\label{eq:levelset conv}
		a_n\coloneqq\sup_{s \in(t_1^\infty-\eta,t_2^\infty+\eta )} \sfd_H^Y(\{\sqrt{v_n}=s\},\{\sqrt{v_\infty}=s\})\to 0, \quad \text{ as }n \to +\infty. 
	\end{equation}
(We point out that this does not follows form the uniform convergence, indeed we need first to prove some regularity of the level sets of $\sqrt{v_\infty}.$)
	The key observation is that for every $\eps'>0$ there exists $\delta'>0$ such that for every $t \in [1+\eps,2\eps^{-1}]$ it holds
	\begin{equation}\label{eq:good level set}
		B^{X_\infty}_{\eps'}(x)\cap \{\sqrt{v_\infty}=t'\}\neq \emptyset, \quad \forall x \in \{\sqrt{v_\infty}=t\}, \forall  t' \in [t-\delta',t+\delta'].
	\end{equation}
This is an immediate consequence of  \eqref{eq:vgood}, the fact that $T$ is a local isometry and the fact that, since  $Y'$ is a cone, for every $y' \in Y'$ there exists a ray emanating from $O_{Y'}$ and passing through $y'$.

	Suppose now that \eqref{eq:levelset conv} does not hold. Then, up to a passing to non relabelled subsequence, there exists a sequence $(s_{n})\subset (t_1^\infty-\eta,t_2^\infty-\eta )$ and $\eps'>0$ such that  $s_{n}\to \bar s \in [t_1^\infty-\eta,t_2^\infty+\eta] \subset [1+\eps,2\eps^{-1}]$ and
	$$
	\sfd_H^Y(\{\sqrt{v_n}=s_{n}\},\{\sqrt{v_\infty}=s_{n}\})>\eps', \quad \forall n.
	$$
	Therefore, up to passing to a further subsequence, there exists either a sequence $y_{n}\in  \{\sqrt v_{n}=s_{n}\}$ such that $\sfd_Y(y_{n},\{\sqrt{v_\infty}=s_{n}\})>\eps'$, for all $n$, or  a sequence $y_{n}\in \{\sqrt{v_\infty}=s_{n}\}$ such that $\sfd_Y(y_{n},\{v_{n}=s_{n}\})>\eps'$, for all $n$. In the first case, since by  assumption  $\{\sqrt v_n=s_n\}\subset \{t_1^n-2\eta \le \sqrt{v_n}\le t_2^n+2\eta\}\subset B_{R_0}(x_n)$ for $n$ big enough, up to passing to a further subsequence we have that $y_{n} \to y_\infty \in \X_\infty$ and by uniform convergence (recall \eqref{eq:vn}) that $\sqrt{v_\infty}(y_\infty)=\bar s$. In particular $\sfd(y_\infty,\{\{\sqrt{v_\infty}=s_{n}\}\})>\eps'/2$ for every $n$ big enough, which contradicts \eqref{eq:good level set}. In the second case, again up to  a subsequence and from the continuity of $\sqrt{v_\infty}$, we have that $y_{n} \to y_\infty \in \{\sqrt{v_\infty}=\bar s\}$. 
	Moreover from \eqref{eq:good level set} it follows the existence of a $\delta'>0$ such that there exist  $y_{\infty}^+,y_\infty^-\in B^{\X_\infty}_{\eps'/4}(y_\infty)$ such that $\sqrt{v_\infty}(y_{\infty}^\pm)=\bar t\pm\delta'$. Finally, from uniform convergence (recall again \eqref{eq:vn}), for every $n$ big enough there exist $y_{n}^+, y_{n}^-\in \X^{n}$ such that $\sfd_Y(y_{n}^\pm,y_{\infty}^\pm)<\eps'/4 $ and $|v_{n}(y_{n}^\pm)-(\bar s\pm \delta')|<\delta'/2.$ In particular by the continuity of $v_{n}$, for every $k$ big enough, there exists $z_{n}$ which lies on a geodesic connecting $y_{n}^+$ and $ y_{n}^-$ such that $z_{n}\in\{v_{n}=s_{n}\}$. From the triangle inequality it follows that $\sfd(z_n,y_n)<\eps'$ if $n$ is big enough, which is a contradiction since $	y_n \in \{\sqrt{v_\infty} =s_{n}\}$. 

	From \eqref{eq:levelset conv} it follows that 
		$\sfd_H^Y(\{t_1^n\le \sqrt{v_n}\le t_2^n\},\{t_1^n\le \sqrt{v_\infty}\le t_2^n\})\to 0$ as $n \to +\infty$.
	Moreover it is clear that $\sfd_H^Y(\{t_1^n\le \sqrt{v_\infty}\le t_2^n\},\{t_1^\infty\le \sqrt{v_\infty}\le t_2^\infty\})\to 0$  as $n \to +\infty$  (recall \eqref{eq:vgood}), therefore 
	\begin{equation}\label{eq:sublevel conv}
		b_n\coloneqq\sfd_H^Y(\{t_1^n\le \sqrt{v_n}\le t_2^n\},\{t_1^\infty\le \sqrt{v_\infty}\le t_2^\infty\})\to 0, \quad \text{ as $n \to +\infty$.}
	\end{equation}
In particular, since both sets are compact,  we can build a Borel map $f_n : \{t_1^n\le \sqrt{v_n}\le t_2^n\}\to \{t_1^\infty\le \sqrt{v_\infty}\le t_2^\infty\}$ that has $b_n$-dense image and such that $\sfd_Y(x,f_n(x))\le 2b_n$ for all $x \in \{t_1^n\le \sqrt{v_n}\le t_2^n\}.$

We claim that
\begin{equation}\label{eq:meas conv}
	{f_n}_*\left(\mea_n \restr{\{t_1^n\le \sqrt v_n\le t_2^n\}}\right)\rightharpoonup \mea_\infty \restr{\{t_1^\infty\le \sqrt v_\infty\le t_2^\infty\}}, \quad \text{in duality with }C(\{t_1^\infty\le \sqrt v_\infty\le t_2^\infty\}).
\end{equation}
From the fact that $\sfd_Y(.,f_n(.))\le 2b_n$ and using dominated convergence  it is enough to show that 
$$\mea_n \restr{\{t_1^n\le \sqrt v_n\le t_2^n\}}\rightharpoonup \mea_\infty \restr{\{t_1^\infty\le \sqrt v_\infty\le t_2^\infty\}}, \quad \text{in duality with }C_{bs}(Y).$$
To prove the above we first define for every $\delta>0$ the closed set $C_{\delta}\coloneqq\{y \in Y \ : \ \sfd_Y(y,\{\sqrt{v_\infty}=t_1^\infty\}\cup \{\sqrt{v_\infty}=t_2^\infty\})\le\delta\}$ and observe that for every $\eps'>0$ there exists $\delta'$ such that
\begin{equation}\label{eq:vanish annulus}
			\mea_{\infty}(C_{\delta'})<\eps'.
\end{equation}
This can be seen using the fact that $T$ is a measure preserving local isometry, the Bishop-Gromov inequality and formula \eqref{eq:vgood}.

We also define for any $\delta>0$ the sets  $A_\delta\coloneqq\{y \in Y \ : \ \sfd_Y(y,X_\infty\setminus \{t_1^\infty\le \sqrt{v_\infty}\le t_2^\infty\})\ge \delta\}$ and $B_\delta\coloneqq \{y\in Y\ : \ \sfd_Y(y, \{t_1^\infty\le \sqrt{v_\infty}\le t_2^\infty\}\le \delta)\}$.  We claim that  $ B_{\delta_1}\setminus A_{\delta_2} \subset C_{2\delta_1+2\delta_2},$ for every $\delta_1,\delta_2>0.$
To see this let $y \in B_{\delta_1}\setminus A_{\delta_2}$, which implies $\sfd(y,\{t_1^\infty\le \sqrt{v_\infty}\le t_2^\infty\})\le \delta_1,\sfd(y,X_\infty\setminus \{t_1^\infty\le \sqrt{v_\infty}\le t_2^\infty\})< \delta_2$. Taking two points $y_1,y_2\in \X_\infty$ which realize these two distances we must have that $\sfd_\infty(y_1,y_2)=\sfd_Y(y_1,y_2)\le \delta_1+\delta_2$. Moreover  any geodesics in $\X_\infty$ from $y_1$ to $y_2$, by the continuity of $v_\infty$, must intersects $\{\sqrt{v_\infty}=t_1^\infty\}\cup \{\sqrt{v_\infty}=t_2^\infty\},$ from which the claim follows.

We finally fix $\phi \in C_{bs}(Y)$ and $\eps'>0$ arbitrary. Let $\delta'=\delta'(\eps')$ be the one given by \eqref{eq:vanish annulus} and pick   $\eta\in C_b(Y)$ such that $0\le \eta\le 1$, $\eta=1$ in $A_{\delta'/4}$ and $\supp(\eta)\subset A_{\delta'/8}.$ Observe that by uniform convergence (recall also \eqref{eq:vn}), for $n$ big enough we have that $	A_{\delta'/8}\cap\supp(\phi)\cap  X_n\subset \{t_1^n\le \sqrt{v_n}\le t_2^n\}\subset B_{\delta'/4},$ therefore
\begin{align*}
\limsup_{n} &\left |\int \phi\,\d \mea_n\restr{\{t_1^n\le \sqrt{v_n}\le t_2^n\}}-\int \phi\,\d \mea_\infty\restr{\{t_1^\infty\le \sqrt{v_\infty}\le t_2^\infty\}} \right |\le \limsup_n  \left |\int \phi\eta \,\d \mea_n-\int \phi\eta\,\d \mea_\infty \right |+\\
&+ \limsup_n \|\phi\|_\infty \mea_n(B_{\delta'/4}\setminus A_{\delta'/4})+\|\phi\|_\infty \mea_\infty (\{t_1^\infty\le \sqrt{v_\infty}\le t_2^\infty\}\setminus A_{\delta'/4})\\
&\le  \limsup_n \mea_n(C_{\delta'})+\mea_\infty(C_{\delta'})\overset{\eqref{eq:vanish annulus}}{\le}2\eps'.
\end{align*}
From the arbitrariness of $\eps'$ and $\phi \in C_{bs}(Y)$ the convergence in \eqref{eq:meas conv} follows.  

We now pass to the study of the behaviour of $f_n$ with respect to the intrinsic metrics. More precisely for every $\tau>0$  we set $A^\tau_n\coloneqq\{t_1^n-\tau< \sqrt{v_n}<t_2^n+\tau\}$, $A^\tau_\infty\coloneqq\{t_1^\infty-\tau< \sqrt{v_\infty}<t_2^\infty+\tau\}$, $A^\tau_{\infty,n}\coloneqq\{t_1^n-\tau< \sqrt{v_\infty}<t_2^n+\tau\}$ and denote by $\sfd^\tau_n, \sfd^\tau_\infty,\sfd^\tau_{\infty,n}$ the intrinsic metrics on $A_n^\tau,\, A_\infty^\tau,A^\tau_{\infty,n}$  respectively  (see the beginning of this section). It is clear that the metrics $\sfd^\tau_n, \sfd^\tau_\infty,\sfd^\tau_{\infty,n}$  induce on the sets $\{t_1^n\le \sqrt{v_n}\le t_2^n\},\{t_1^\infty\le \sqrt{v_\infty}\le t_2^\infty\},\{t_1^n\le \sqrt{v_\infty}\le t_2^n\}$ the same topology induced by the metrics $\sfd_n,\sfd_\infty$.

Notice also that, from \eqref{eq:vgood} and since $T$ is a local isometry on $\{\sqrt v_\infty>1+\eps/4\}$,
\begin{equation}\label{eq:linkcone}
	(\{s\le \sqrt{v_\infty}\le t\},\sfd_\infty^{s,t,\tau})  \text{ is isometric to } (\{s\le \lambda \sfd_{O_{Y'}}\le t\},\sfd_{Y'}^{s,t,\tau}), \,\,\, \forall \tau \in (0,\eta),\, \forall\,t\ge s > 1+\eps+\eta ,
\end{equation}
where $\sfd_\infty^{s,t,\tau}$ and $\sfd_{Y'}^{s,t,\tau}$ are the intrinsic metrics respectively  on $\{s-\tau< \sqrt{v_\infty}<t+\tau\}$ and  on $\{s-\tau< \lambda \sfd_{O_{Y'}}<t+\tau\}$, the isometry being $T$ itself, which also measure preserving. In particular there exists a constant $D>0$ such that for every $\tau\in(0,\eta)$ it holds $\diam(\{t_1^\infty\le \sqrt{v_\infty}\le t_2^\infty\},\sfd_\infty^{\tau})\le D.$

Observe that from \eqref{eq:vn} we  deduce that the functions $\sqrt{v_n}, \sqrt{v_\infty}$ are equi-Lipschitz on $\{t_1^n-\eta\le \sqrt{v_n}\le t_2^n+\eta\}$, $\{t_1^\infty-\eta\le \sqrt{v_\infty}\le t_1^\infty+\eta\}$  and we fix $M\ge 2$ a bound on their Lipschitz constant. 

Putting $\eps_n \coloneqq2\max(b_n,a_n)$ (where $a_n,b_n$ are the ones in \eqref{eq:levelset conv} and \eqref{eq:sublevel conv}) it is not restrictive to assume both that $\sqrt{\eps_n}<\eta/(2M)$ and that $|t_1^n-t_1^\infty|,|t_2^n-t_2^\infty|<\eps_n$, for every $n$.

Pick any $x_0,x_1 \in \{t_1^n\le \sqrt{v_n}\le t_2^n\}$ and set $y_i=f_n(x_i)\in\{t_1^\infty\le \sqrt{v_\infty}\le t_2^\infty\}$, $i=0,1$, where $f_n$ was defined above and recall that $\sfd_Y(x_i,f_n(x_i))\le \eps_n$. Consider an absolutely continuous curve $\gamma: [0,1]\to \overline{A_\infty^{\eps-2M\sqrt{\eps_n}}}$ such that  $\gamma(i)=y_i$, $i=0,1$ and $L(\gamma)= \sfd_{\infty}^{\eta-2M\sqrt{\eps_n}}(y_0,y_1)\le D.$ Letting $N_n\coloneqq\lfloor 2D/\sqrt{\eps_n}\rfloor$, there exist $0=t_0<t_1<...<t_{N_n}=1$ such that $\sfd_\infty(\gamma(t_i),\gamma(t_{i+1})\le L(\gamma\restr{[t_i,t_{i+1}]}))\le  L(\gamma)/N_n,$ for every  $i=0,...,N_n-1$. Thanks to \eqref{eq:levelset conv} and the $M$-Lipschitzianity of $\sqrt{v_n}$  there exist
 points  $x_i\in A_n^{\eta-M\sqrt{\eps_n}}$ $i=1,...,N_n-1$, such that  $\sfd_Y(x_i,\gamma(t_i))<\eps_n$, $i=1,...,N_n-1$, and in particular   $|\sfd_n(x_i,x_{i+1})-\sfd_\infty(\gamma(t_i),\gamma(t_{i+1})|\le 2\eps_n,$ for every $i=0,...,N_n.$ Therefore  $\sfd_n(x_i,x_{i+1})< L(\gamma)/N_n+2\eps_n\le \sqrt{\eps_n}$, and thus any geodesic (in $\X_n$) $\gamma_i$ from $x_i$ to $x_{i+1}$  has image  contained in $A_n^{\eta}$. We define $\bar \gamma:[0,1]\to A_n^{\eta}$ as the concatenation of all the geodesics $\gamma_i$ (appropriately reparametrized), in particular 
\begin{equation}\label{eq:bound1}
\sfd_n^{\eta}(x_0,x_1)\le L(\bar \gamma)\le N_n \left(\frac{L(\gamma)}{N_n}+2\eps_n\right)\le \sfd_\infty^{\eta-2M\sqrt{\eps_n}}(f_n(x_0),f_n(x_1))+4D\sqrt{\eps_n}.
\end{equation}
Conversely pick  an absolutely continuous curve $\bar \gamma: [0,1]\to \overline {A_n^{\eta}}$ such that $\bar \gamma(i)=x_i$, $i=0,1$ and $L(\bar \gamma)= d_n^{\eta}(x,y)\le 2D,$ which exists thanks to \eqref{eq:bound1}. Arguing exactly as above we can construct an absolutely continuous curve $ \gamma: [0,1]\to A_\infty^{\eta+2M\sqrt{\eps_n}}$ such that $\gamma(i)=y_i$, $i=0,1$ and 
$$\sfd_\infty^{\eta+2M\sqrt{\eps_n}}(f_n(x_0),f_n(x_1))\le L( \gamma)\le \sfd_n^{\eta}(x_0,x_1)+4D\sqrt{\eps_n}.$$

Recalling \eqref{eq:linkcone} we are in position to apply Lemma \ref{lem:conelemma1} and deduce that $\sfd_\infty^{\eta\pm 2M\sqrt{\eps_n}}(f_n(x_0),f_n(x_1))\to \sfd_\infty^{\eta}(f_n(x_0),f_n(x_1))$ as $n\to +\infty$, uniformly in $x_0,x_1 \in \{t_1^n\le \sqrt{v_n}\le t_2^n\}$. Moreover again from Lemma \ref{lem:conelemma1} we have that the image of $f_n$ is $cb_n$-dense in $\{ t_1^\infty\le \sqrt v_\infty\le t_2^\infty\}$ w.r.t. the metric $\sfd_\infty^{\eta}$, for some constant $c$ independent of $n$.

 Combing this with the above inequalities and \eqref{eq:meas conv} we obtain
\begin{equation}\label{eq:mgh1}
\left (\{t_1^n\le \sqrt{v_n}\le t_2^n\},\sfd_n^{\eta},\mu_n  \right)\overset{mGH}{\rightarrow} \left ( \{ t_1^\infty\le \sqrt{v_\infty}\le t_2^\infty\},\sfd_{\infty}^{\eta},\mu_\infty  \right),
\end{equation}
with   $\mu_n\coloneqq\mea_n\restr{\{t_1^n\le \sqrt v_n\le t_2^n\}}$, $\mu_\infty\coloneqq\mea_\infty\restr{\{t_1^\infty\le \sqrt v_\infty\le t_2^\infty\}}$ and where, if $t_1^\infty=t_2^\infty$, the convergence is intended only in the GH-sense.
Finally from Lemma \ref{lem:conelemma1} and recalling \eqref{eq:linkcone} we have that $\left (\{t_1^n\le\sqrt{v_\infty}\le t_2^n\},\sfd_{\infty,n}^{\eta}  \right)\overset{GH}{\rightarrow} \left( \{ t_1^\infty\le \sqrt{v_\infty }\le t_2^\infty\},\sfd_{\infty}^{\eta}  \right),$ and that such convergence can be realized by  a map  $g_n:  \{t_1^n\le \sqrt{v_\infty }\le t_2^n\} \to \{ t_1^\infty\le \sqrt{v_\infty }\le t_2^\infty\} $ that (if $t_1^n\neq t_2^n$) also satisfies ${g_n}_*\left(\mea_\infty\restr{\{t_1^n\le\sqrt{v_\infty }\le t_2^n\}}\right)=\frac{(t_2^\infty)^N-(t_1^\infty)^N)}{((t_2^n)^N-(t_1^n)^N)}\mea_\infty \restr{\{ t_1^\infty\le \sqrt{v_\infty }\le t_2^\infty\}}$.  In particular
\begin{equation}\label{eq:mgh2}
\left (\{t_1^n\le\sqrt{v_\infty}\le t_2^n\},\sfd_{\infty,n}^{\eta},\mu_{\infty,n}  \right)\overset{mGH}{\rightarrow} \left( \{ t_1^\infty\le \sqrt{v_\infty }\le t_2^\infty\},\sfd_{\infty}^{\eta},\mu_\infty  \right),
\end{equation}
with $\mu_{\infty,n}\coloneqq\mea_\infty\restr{\{t_1^n\le \sqrt {v_\infty}\le t_2^n\}}$ and where in the case $t_1^\infty=t_2^\infty$ the convergence is intended only in the GH-sense. 

If $t_1^\infty=t_2^\infty$, combining \eqref{eq:mgh1} with \eqref{eq:mgh2}, and recalling \eqref{eq:linkcone}, we obtain that \eqref{eq:closeannuli1} holds for $n$ big enough. Therefore, since we are assuming $d')$ (see above), up to a subsequence, we must have that  either \eqref{eq:closeannuli2} is  false every $n$ big enough or the last claim about $Y$ in Theorem \ref{thm:almconev2} is false.
The latter cannot happen, indeed in the first half of the proof we proved precisely that Theorem \ref{thm:almconev1} holds with the same $Y$ and with the same $\eps,L, R_0.$
Hence we must be  n the first case  and in particular $t_1^\infty+\eps<t_2^\infty$ and we can apply Proposition \ref{prop:GH to D} together with \eqref{eq:mgh1} and \eqref{eq:mgh2} and  obtain that \eqref{eq:closeannuli2} holds for $n$ big enough, which is a contradiction.  \textbf{This concludes the proof of Theorem \ref{thm:almconev2}}.
\end{proof}

\section{Rigidity and almost rigidity from the monotonicity formula}\label{sec:rigidity from monotonicty}

\subsection{Rigidity}
The following rigidity result follows almost immediately combining the explicit lower bound on the derivative of $U_{\beta}'$ in \eqref{eq:U'positive} and Theorem \ref{thm:functional cone} about ``from outer functional cone to outer metric cone".
\begin{theorem}\label{thm:rigidity}
	Let $\X,\, \Omega,\, u , U_\beta$, with $\beta >\frac{N-2}{N-1}$, be as in Theorem \ref{thm:monotonicity} and  	suppose that  $U^{'-}_{\beta}(t_0)=0$ for some $t_0\in (0,1]$.Then the hypotheses of Theorem \ref{thm:functional cone},  and in particular also its conclusions, are satisfied choosing $\bu=Cu^\frac{2}{2-N}$, $\bu_0=Ct_0^\frac{2}{2-N}$ and $U=\{u<t_0\}$, for some constant $C>0$.
\end{theorem}

\begin{proof}
	
	Suppose $U_{\beta}^{'-}(t_0)=0$ for some $t_0\in(0,1]$ and observe that, thanks to \eqref{eq:U'positive}, since $ C_{\beta,N}>0$, we must have that $|\nabla |\nabla u^{\frac{1}{2-N}}|^{\beta/2}|=0$ $\mea$-a.e. in $\{u<t_0\}$. We claim that $\{u<t_0\}$ is connected. Indeed, if $t_0<1$, from the continuity of $u$ follows that $ \partial\{u<t_0\}\subset \{u=t_0\}$, hence from the strong maximum principle we deduce that all the connected components of $\{u<t_0\}$ are unbounded. Moreover $\partial \{u<t_0\}$ is bounded and thus from Corollary \ref{cor:ends} it follows that that  $\{u<t_0\}$ is connected. If $t_0=1,$ we conclude observing that $\{u<1\}$ is the union of the sets $\{u<t\}$ with $t<1.$ Therefore we have that $|\nabla u^{\frac{1}{2-N}}|^2\equiv C$ $\mea$-a.e. in $\{u<t_0\}$, for some constant $C$. We now claim that $C>0$. Indeed if $C=0$ we would have that $\nabla u=-(2-N)u^{\frac{N-1}{N-2}}\, \nabla u^{\frac{1}{2-N}} =0$ $\mea$-a.e. in $\{u<t_0\}$ and therefore  $u$ would be constant in  $\{u<t_0\}$ (recall \eqref{eq:loc constant}). However $u$ goes to 0 at $+\infty$ and $\{u<t_0\}$ is unbounded, therefore $u\equiv 0$ in $\{u<t_0\}$, but this violates the positivity of $u$.
	Setting $v=u^{\frac{1}{2-N}}$, by the chain rule for the Laplacian, the harmonicity of $u$ and by locality we have
	\[
	\Delta \frac{v^2}{2}=\frac12 \Delta (u^{\frac{2}{2-N}})=\frac{N}{(2-N)^2}\frac{|\nabla u|^2}{u^{2\frac{N-1}{N-2}}}=CN, \quad \text{$\mea$-a.e. in $\{u<t_0\}$.}
	\]
	Moreover $|\nabla v^2/2|^2=v^2|\nabla v|^2=2C\frac{v^2}{2}.$
	Therefore the function $\bu=C^{-1}v^2/2$ satisfies the hypotheses of  Theorem \ref{thm:functional cone} with $U=\{u<t_0\}$ and $\bu_0=C^{-1}t_0^{2/(2-N)}/2$. This concludes the proof.
\end{proof}

\subsection{Almost rigidity}

The goal of this subsection is to prove the following.
\begin{theorem}\label{thm:almrigv1}
	For all numbers $\eps\in(0,1/3)$, $R_0>0$, $\beta> \frac{N-2}{N-1}$, $N\in (2,\infty)$ and for every function $f:(1,+\infty)\to \rr^+$  in $L^1(1,+\infty)$ there exists $0<\delta=\delta(\eps,\beta,N,f)$ such that the following holds. Let $(\X,\sfd,\mea,x_0)$ be a pointed  ${\rm RCD}(0,N)$ m.m.s. such that $\mea(B_1(x_0))\le \eps^{-1}$ and $\frac{s}{\mea(B_s(x_0))}\le f(s)$ for $s\ge \eps^{-1}$. Let $u$ be a solution to \eqref{mainpde} with  $\|u\|_{L^\infty(\Omega)}\le \eps^{-1}$ and such that there exists $t\in(\eps,1]$ satisfying  $\diam(\{u>t-\eps t\}^c)<R_0$, $\sfd(x_0,\{u\le t\})>\eps$, $\||\nabla u|\|_{L^\infty(\{u<t\})}\le \eps^{-1}$  and
	\begin{equation}\label{eq:small der}
		{U^-_\beta}'(t)<\delta.
	\end{equation}
	Then there exists a pointed ${\rm RCD(0,N)}$ space $(\X',\sfd',\mea',x')$ such that 
	\[
	\sfd_{pmGH}((\X,\sfd,\mea,x_0),(\X',\sfd',\mea',x'))<\eps
	\]
	and $(\X',\sfd',\mea',x')$ is a truncated cone outside a compact set $K\subset B_{2R_0}(x'),$ i.e. there exists an  ${\rm RCD(0,N)}$ Euclidean $N$-cone $Y$, with vertex $O_Y$,  over an ${\rm RCD}(N-2,N-1)$ space $Z$ and a measure preserving local isometry $T: \X'\setminus K\to Y\setminus \bar B_{r}(O_Y)$, for some $r>0.$
\end{theorem}

We will also prove the following alternative version of the above statement (see Section \ref{sec:almcone} for the definition of $\mathbb{D}, \sfd_{GH}$ and of intrinsic metric).
\begin{theorem}\label{thm:almrigv2}
	For all numbers $\eps\in(0,1/3)$, $R_0>0$ $\beta> \frac{N-2}{N-1}$, $N\in (2,\infty)$, $\eta>0$ and for every function $f:(1,+\infty)\to \rr^+$  in $L^1(1,+\infty)$ there exists $0<\delta=\delta(\eps,\beta,N,f,\eta)$ such that, given $\X$, $v$ and $t$ as in Theorem \ref{thm:almconev1}, there exists an ${\rm RCD(0,N)}$ Euclidean $N$-cone $(Y,\sfd_Y,\mea_Y)$, with vertex $O_Y$, over an ${\rm RCD}(N-2,N-1)$ space $Z$ and a constant $\lambda>0$ such that the following holds. For every $(1+\eps+\eta)t^{\frac{1}{2-N}}<t_1<t_2<\eps^{-1}t^{\frac{1}{2-N}}$  it holds
	\begin{equation*}
		\sfd_{GH}\left((\{t_1\le u^{\frac{1}{2-N}}\le t_2\}, \sfd_{\X}^{\eta}),\left(\{t_1\le \lambda \sfd_{O_Y}\le t_2, \sfd_{Y}^{\eta}\right)  \right)<\eps,
	\end{equation*}
	where $\sfd_{O_Y}\coloneqq\sfd_Y(.,O_Y)$ and $\sfd_{\X}^{\eta}$ and $\sfd_{Y}^{\eta}$ denote the  intrinsic metrics on  $\{t_1-\eta t^{\frac{1}{2-N}}< u^{\frac{1}{2-N}}<t_2+\eta t^{\frac{1}{2-N}}\}$ and on $\{t_1-\eta<\lambda \sfd_{O_Y}< t_2+\eta\}$. Moreover, provided that $t_1+\eps t<t_2$, 
	\begin{equation*}
		\mathbb{D}\left(\big(\{t_1\le u^{\frac{1}{2-N}}\le t_2\}, \sfd_{\X}^{\eta},\mea\restr{\{t_1\le u^{\frac{1}{2-N}}\le t_2\}}\big),\big(\{t_1\le \lambda \sfd_{O_Y}\le t_2\}, \sfd_{Y}^{\eta},\mea_Y\restr{\{t_1\le \lambda \sfd_{O_Y}\le t_2\}}\big)  \right)<\eps.
	\end{equation*}
\end{theorem}

Before passing to the proof we explain why the bound on $\||\nabla u|\|_{L^\infty(\{u<t\})}$ is  natural and often satisfied. The immediate observation is that from the gradient bound for harmonic functions \eqref{eq:cheng} we deduce
\begin{equation}\label{eq:grad bound weak}
	|\nabla u|\le \frac{C(N)}{\eps}, \quad \text{$\mea$-a.e. in $\{u<1\}\cap \{x \ : \ \sfd(x,\partial \Omega)>\eps\}$}, \quad \forall \, \eps>0.
\end{equation}
In particular for fixed $\eps>0$, thanks to the assumption $\liminf_{x\to \partial \Omega}u(x)\ge 1$, for $t$ sufficiently small (but depending on $u$) the gradient bound $\||\nabla u|\|_{L^\infty(\{u<t\})}\le C(N)\eps^{-1}$  is always satisfied.  An estimate on the value of $t$  can be given in the case $B_{\eps}(x_0)\subset \Omega.$ Indeed applying the lower bound for $u$ given by \eqref{eq:twosidebound},  it is immediately seen that $\{u<t\}\subset  \{x \ : \ \sfd(x,\partial \Omega)>\eps\}$ for any $0<t < \frac 12\left( \frac{\eps}{\diam(\Omega^c)+\eps}\right)^{N-2}.$

Something more explicit can be said if we consider $u$ to be an electrostatic potential. Indeed combining \eqref{eq:grad bound weak} with the continuity estimate \eqref{eq:continuity estimate electro} one can  easily prove the following:
\begin{prop}
	For all numbers $\eps\in(0,1/3)$, $N\in (2,\infty)$ there exists $0<C=C(\eps,N)$ such that the following holds.  Let $(\X,\sfd,\mea)$ be a non compact ${RCD}(0,N)$ m.m.s.  and let $E\subset \X$ be open and bounded  with uniformly $\Cap$-fat boundary with parameters $(\eps,\eps)$ (see Def. \ref{def:capfat}). Let $u$ be the capacitary potential relative to $E$ (see Theorem \ref{thm:potential}). Then
	\[
	|\nabla u|\le C(\eps,N), \quad \text{$\mea$-a.e. in $\{u<t\}$, for every $t \in(0,1-\eps)$}.
	\]
\end{prop}

We pass to the proof of Theorem \ref{thm:almrigv1} and Theorem \ref{thm:almrigv2}, which are almost corollaries of Theorem \ref{thm:almconev1} and Theorem \ref{thm:almconev2}.

\begin{proof}[Proof of Theorem \ref{thm:almrigv1} and Theorem \ref{thm:almrigv2}]	
	Observe first that by Bishop Gromov inequality 
	$$\mea (B_1(x_0))\ge \mea(B_{\eps^{-1}})\eps^N\ge \eps^{N-1}/f(\eps^{-1}).$$
	From the second in \eqref{eq:twosidebound}  we have that there exist a positive constant $C_1=C_1(\eps,R_0,N)$, such that 
	\begin{equation}\label{eq:unifdecay}
		u(x) \le C_1\int_{\sfd(x,x_0)}^\infty f(s)\,\d s, \quad \forall \text{ $x\in B_{\eps^{-1}\vee R_0}(x_0)^c$}.
	\end{equation}
	In particular, since $t>\eps$, $\diam(\{u<t\}^c)\le \diam(\{u<\eps\}^c)\le C_2$, for some constant $C_2$ depending only on $\eps,R_0,N,f$.  Therefore, again since $t>\eps$, up to rescaling $u$ as $ut^{-1}$ we can assume that $t=1$ (observe also that, called $\tilde U_{\beta}$ the function relative to $t^{-1}u$, it holds that $\tilde U_{\beta}(s)=U_\beta(ts)t^{\beta\frac{N-1}{N-2}-\beta-1}$, $s\in(0,1)$). 
	
	Define $v\coloneqq u^{\frac{2}{2-N}}$ and set $\Omega'=\{u<1\}=\{v>1\}$. Note also that $\diam(\{v\le 1+c_N\eps\}^c)<R_0$ for some constant $c_N\le 1$. From \eqref{mainpde} we have that $\Delta  v=N|\nabla \sqrt {2v}|^2 $, $\mea$-a.e. in $\Omega'$ and that $\limsup_{x\to \partial \Omega'}v(x)\le1$. 
	
	To apply Theorems \ref{thm:almconev1} and \ref{thm:almconev2} it still remains to check the bounds on $|\nabla \sqrt v|.$

	Observe that from the assumption $\sfd(x_0,\{u\le 1\})>\eps$ and the first in \eqref{eq:twosidebound} we deduce that 
	\begin{equation}\label{eq:u non troppo piccola}
		u_n(x) \ge c\sfd(x_n,x)^{2-N}, \text{ for every $x \in  \Omega'$},
	\end{equation}
	for some positive constant $c=c(\eps)$. Combining \eqref{eq:u non troppo piccola}, the assumption $\||\nabla u|\|_{L^\infty(\{u<1\})}\le \eps^{-1}$ and  the gradient estimate \eqref{eq:cheng}, it easily follows that
	\[
	|\nabla \sqrt v|=(N-2)^{-1}|\nabla u|u^{\frac{1-N}{N-2}}\le C_3,
	\]
	for some positive constant $C_3=C_3(\eps)$. Finally from \eqref{eq:U'positive} and \eqref{eq:small der} we have
	\begin{equation*}
		\int_{\{u<1\}} \frac{1}{v^{N-2}} \left |\nabla |\nabla \sqrt v|^{\beta/2} \right |^2\, \d \mea\le C_{\beta,N}^{-1}{U_\beta^-}'(1)<C_{\beta,N}^{-1}\delta .
	\end{equation*}
	We are therefore in position to apply Theorem \ref{thm:almconev1} and conclude the proof of Theorem \ref{thm:almrigv1}.
	
	Theorem \ref{thm:almrigv2} follows from Theorem  \ref{thm:almconev2} and observing that  $\{u=s^{\frac{1}{2-N}}\}=\{\sqrt v=s\}$ and that, thanks to \eqref{eq:unifdecay}, for every $t>0$
	\[
	\{\sqrt v\le t\}\subset B_{R_1}(x_0),
	\]
	for some $R_1=R_1(t,\eps,N,f,R_0)$.
\end{proof}

\section{The electrostatic potential}\label{sec:electro}
It was already discussed in Section \ref{sec:nonpara} that a solution \eqref{mainpde} is not granted, in particular already for Riemannian manifolds, the existence of solutions implies non parabolicity. We also showed (see Corollary \ref{cor:green sol}) that the Green function solves \eqref{mainpde}. In this short section we provide another example of solution to \eqref{mainpde} given by the electrostatic potential. We recall that this type of solution was crucial in the recent work \cite{AFM}.

\begin{definition}[Electrostatic potential]
	Given $(\X,\sfd,\mea)$ an (unbounded) infinitesimally Hilbertian m.m.s. and $E \subset \X$ open and bounded, an \emph{electrostatic potential for $E$} is a function  $u \in D(\bd, \X\setminus \bar E)\cap C(\X\setminus E)$ solution to
	\[
	\begin{cases}
		\bd\restr {\X\setminus \bar E} u=0,\\
		u=1, & \text{ in } \partial E,\\
		u(x) \to 0 & \text{as }\sfd(x,\partial E)\to +\infty.
	\end{cases}
	\]
\end{definition}
\begin{remark}
	If an electrostatic potential for $E$ exists, then it is also unique, this follows immediately from the maximum principle (see Proposition \ref{prop:maxprinc}).
\end{remark}
\begin{remark}
   If $u$ is a solution to \eqref{mainpde}, then the function $(1-\eps)^{-1}u\restr{\{u\le 1-\eps\}}$ is immediately seen to be the electrostatic potential for the open set $\X\setminus \{u\le 1-\eps\}.$
\end{remark}

We pass to our main existence result for the electrostatic potential, which holds for sets  with sufficiently regular boundary, namely  with $\Cap$-fat regular boundary.  We refer to Appendix \ref{ap:bjorn} for the definition of $\Cap$-fat regular boundary and for examples of sets satisfying this condition.
\begin{theorem}\label{thm:potential}
	Let $(\X,\sfd,\mea)$ be a nonparabolic ${\rm RCD}(0,N)$ m.m.s. and let  $E \subset \X$ be open and bounded  with $\Cap$-fat boundary. Then the electrostatic potential $u$ for $E$ exists.
	Moreover  the following continuity estimate holds: for every $x \in \partial E$ it holds
	\begin{equation}\label{eq:continuity estimate electro}
		1-u(y)\le \sfd(y,x)^{\alpha_x}, \quad \forall y\in B_{r_x/2}(x)\cap E^c, 
	\end{equation}
	for some positive constant  $\alpha_x=\alpha(r_x,c_x,N)>0$, where $r_x,c_x$ are the $\Cap$-fatness parameters of $x$. 
	
	Finally the function
	\[
	\tilde u\coloneqq\begin{cases}
		u, & \X\setminus E,\\
		1, & \bar E,
	\end{cases}
	\]
	belongs to $S^2(\X)$ (recall Def. \ref{def:sobolev class}) and
	\begin{equation}\label{eq:capacity lim}
		\int_{\X} |D \tilde u|^2\, \d \mea \le \lim_{r\to+\infty} \Cap(E,B_r(x)), \quad \forall x \in E.
	\end{equation}
\end{theorem}
Let us make some comments before passing to proof of this result. We first observe that the limit in \eqref{eq:capacity lim} does not depend on $x \in E$ and that it is actually an $\inf$ and thus  finite. It is not true in general that $\tilde u \in \W(\X)$, indeed it might not be square integrable, as can be seen taking $E$ to be a ball in $\rr^3.$ Nevertheless, if the measure satisfies $0< \liminf_{r\to +\infty} r^{-\lambda}\mea(B_r(x))\le  \limsup_{r\to +\infty} r^{-\lambda}\mea(B_r(x))<+\infty$ for some $\lambda>2$, then $u \in L^p(\mea)$ for every $p >\frac{\lambda}{\lambda-2}$ (see for example the upper bound in \eqref{eq:twosidebound}).

\begin{proof}[Proof of Theorem \ref{thm:potential}]
We will actually build $\tilde u$ and then define $u$ to be the restriction of $\tilde u$ to $E^c.$ The argument is by compactness. Fix $x_0 \in E$ and set $B_n\coloneqq B_n(x_0)$ with $n\in \nn$ and $n>\diam(E)+100$.  Theorem \ref{thm:local potential} guarantees the existence of a function $u_n\in \W_{0}(B_n)\cap C(B)$ harmonic in $B_n\setminus \overline E$, such that $0\le u_n\le 1$, $u_n=1$ on $\bar E$ and $\int_{\X}|\nabla u_n|^2\, \d \mea=\Cap(E,B_n)$. Moreover  from the comparison principle in Proposition \ref{prop:solprop} we must have that $u_n\le u_{n+1}$ in $B_n.$
	
	It follows from Lemma \ref{ascoliarm} that, up to a (non relabelled) subsequence, $u_n$ converges in $\X\setminus \bar E$ uniformly on compact sets to a function $ \tilde u \in C(\X\setminus\bar  E)$ harmonic in $\X\setminus\bar  E.$ In particular $u_n \to \tilde u$ $\mea$-a.e.. Moreover, since $\Cap(E,B_r(x_0))$ is decreasing in $r$, we have that $\sup_n \||\nabla u_n|\|_{L^2(\X)}<+\infty.$  Therefore from the  lower semicontinuity of weak upper gradients \eqref{eq:stability wug}, we deduce that $\tilde u \in {\rm S}^2(\X)$ and 
	$$\int_{\X}|\nabla \tilde u|^2\, \d \mea\le \liminf_n \int_{\X}|\nabla u_n|^2\, \d \mea\le \liminf_n \Cap(E,B_n)=\lim_{r\to+\infty}\Cap (E,B_r(x_0)).$$
	
	The continuity estimate follows directly from the fact that $u_n\le \tilde u\le 1$ for every $n$ and from the continuity estimate in Theorem \ref{thm:local potential}, observing that in the case $K=0$ we con drop the dependence on the diameter of $E$. 
	
	It remains to show that $ u$ goes to 0 at infinity. We prove it by comparison with  the quasi Green function $G^1_{x_0}$ (recall \eqref{eq:quasi green def}). We have that  $G^1_{x_0}$ is  Lipschitz and superharmonic in $\X$.  Moreover $G_{x_0}^1$ is positive, hence  $\lambda G_{x_0}^1\ge \nchi_{E}$ for a large enough constant $\lambda>0$.  In particular the comparison principle in Theorem \ref{thm:local potential} implies that $u_n \le \lambda G_{x_0}^1$ for every $n$, which in turn gives $u \le \lambda G_{x_0}^1.$ Finally from the estimate for the Green function in \eqref{eq:greenestimates} we have
	\[  G_{x_0}^1\le G_{x_0}(x) \le \int_{\sfd(x,x_0)}^\infty \frac{s}{\mea(B_s(x))}\, d s,\]
	in particular $G_{x_0}^1(x)\to 0$ at infinity. This concludes the proof.
\end{proof}

\appendix

\section{Appendix: From outer functional cone to outer metric cone - additional details}\label{ap:cone}

This section is devoted to the  proof of Theorem \ref{thm:functional cone} given the two results presented in Section \ref{sec:functioncone} (that is Proposition \ref{prop:blowdown} and Proposition \ref{prop:localest}).

 Up to Section \ref{sec:levelset} (included) the proof is in large part the same as \cite{volcon}, hence some steps will be only outlined, however there will be  differences and new arguments that will be explained in detail and emphasized along the exposition.  

After Section \ref{sec:levelset} (in particular in Sections \ref{sec:cosine law}, \ref{sec:intrinsic metric} and \ref{sec:conclusion}) the argument diverges from the one in \cite{volcon}, indeed (following a suggestion of an anonymous referee) we replace the second-order analysis of \cite{volcon} with a more direct argument which uses the differentiation formula in \cite{gtamanini} and is  inspired by \cite[Sec. 3.1]{boundary}. This strategy is in turn analogous to the one previously employed by Cheeger and Colding \cite{CC} to derived quantitative almost-splitting results via Hessian estimates. We remark that these type of arguments (as for the one used in \cite{boundary}) were not possible at the time of the writing of \cite{split} and \cite{volcon} (which contains respectively the first versions of the splitting theorem and the volume cone-to-metric cone theorem in ${\rm RCD}$ spaces), since the differentiation formula in \cite{gtamanini} was not yet available.

We also mention the recent \cite{chen} where a similar argument exploiting  \cite{gtamanini} is used  to prove a version of the (almost) volume annulus-to-metric annulus in ${\rm RCD}$ setting.

\subsection{The gradient flow of $\bu$ and its effect on the measure}\label{sec:flow di u}
From the chain rule for the Laplacian \eqref{chainlap}, the positivity of $\bu$ and  recalling  that $\Delta \bu=N$ and $|\nabla \bu|^2=2\bu$ $\mea$-a.e. in $ U$, it follows that
\[
	v\coloneqq\begin{cases}
		\bu^{\frac{2-N}{2}}, & \text{if $N>2$,}\\
		\ln(\frac 1{\sqrt\bu} ), & \text{if $N=2$},
	\end{cases}\quad \text{is harmonic in $ U$.}
\]
In particular the maximum principle and the fact that $\bu_0=\limsup_{x\to \partial  U}\bu(x)$ ensure that 
\begin{equation}\label{eq:unbounded components}
	\text{every connected component of $\{\bu>\bu_0\}$ is unbounded.}
\end{equation}
Since by assumption $\{\bu>\bu_0\}$ is nonempty we must have that $\X$ is unbounded.

For technical reasons we will work locally, in particular we fix a set $V$  open and relatively compact in $U$ and consider $\eta \in \test(\X)$ such that $\eta=1$ in $\overline V$, $0\le \eta\le 1$ and $\supp \eta \subset U$, which exists thanks to Proposition \ref{prop:goodcutoff}.  We then define 
$$u\coloneqq\eta \bu.$$
Since $\bu \in \testloc( U)$ from \eqref{loctotest} we deduce that $u\in \test(X)$. 

We point out that we would like to take right away  $V$  to be of the form $\{t_0<\bu<T_0\}$, however to ensure that this set is relatively compact in $U$ we need first to know that $\bu$ blows up at infinity. This will be proved in Lemma \ref{lem:uinfinity}.

We now consider the regular Lagrangian flow $F : [0,T] \times X \to X$ associated to the autonomous vector field $v=-\nabla u.$ Observe that since $\Delta u\in L^\infty(\mea)$ and $u \in \test(\X)$ the assumptions of Theorem \ref{thm:uniqrfl} are satisfied. In particular the flow $F$ exists unique. Moreover, again thanks to $\Delta u\in L^\infty(\mea)$ and Remark \ref{rmk:extend} we can extend the map $F$ to $(-\infty,+\infty)\times X$.

\begin{prop}\label{prop:basics0}
	\begin{enumerate}
		\item For $\mea$-a.e. $x\in X$ it holds that $F_t(F_s(x))=F_{s+t}(x)$ for every $s,t \in \rr.$
		\item For $\mea$-a.e. $x\in U$  it holds that $(-\infty,\infty)\ni t \mapsto F_t(x)$ is continuous. Moreover denoted by $(a_x,b_x)$ the maximal interval such that $F_t(x)\in U$ for all $ t \in (a_x,b_x)$ (which in particular satisfies $a_x<0, b_x>0$ and possibly $a_x=-\infty$ or $b_x=+\infty$), it holds
		\begin{equation}\label{ualong0}
		\bu (F_t(x))=e^{-2t}\bu (x), \quad \forall t\in(a_x,b_x)
		\end{equation}
		and
		\begin{equation}\label{distancealong0}
		\sfd(F_s(x),F_t(x))\le |e^{-t}-e^{-s}|\sqrt{2\bu (x)}, \quad \forall t,s\in(a_x,b_x).
		\end{equation}
	\end{enumerate}
\end{prop}
\begin{proof}
	The first is just \eqref{eq:group}.
	
	\eqref{ualong0} follows observing that from \eqref{RFL}, since $|\nabla u|^2=2u$ $\mea$-a.e. in $U$, for $\mea$-a.e. $x\in U$ it holds that
	\[
	\frac{\d}{\d t}u(F_t(x))=-2u(F_t(x)),  \quad \text{for a.e.  $t\in(a_x,b_x)$.}
	\]
	\eqref{distancealong0} instead can be derived from the fact that, recalling Remark \ref{rmk:accurverfl}, for $\mea$-a.e. $x\in U$ it holds  $\overset{.}{|F_t(x)|}=\sqrt{2u(F_t(x))}$ for a.e. $t \in (a_x,b_x)$.
\end{proof}

Recall that, as remarked at the beginning of the section, $\X$ is unbounded, hence the following result makes sense.
\begin{lemma}\label{lem:uinfinity}
	\begin{equation}\label{eq:uadinfinto}
	\bu(x)\to +\infty \quad\text{   as  }  \sfd(x, U^c)\to+\infty.
	\end{equation}
\end{lemma}
\begin{proof}
Suppose \eqref{eq:uadinfinto} is false. Then we can find a ball $B_{2R}(\bar x)\subset  U$ such that $R>100\sqrt{\bu(\bar x)}+1.$
	We choose  $\eta \in \test(\X)$ such that $\eta=1$ in $\overline B_{R}(\bar x)$, $0\le \eta\le 1$ and $\supp \eta \subset  U$, which exists from Proposition \ref{prop:goodcutoff}. We define $u\coloneqq \bu\eta \in \test(\X)$ and consider the regular Lagrangian flow $F_t$ relative to $-\nabla u.$ Then from \eqref{distancealong0} (with the choice $V=B_R(\bar x)$), the continuity of $\bu$ and the choice of $R$ we can find $x'\in B_1(\bar x)$ such that the curve $F_t(x')$ is contained in $B_R(\bar x)$ for all $t>0.$ This together with \eqref{ualong0} contradicts the positivity of $\bu$.
\end{proof}

{\bf
From now  until the very last part of the proof  we fix $t_0,T_0 \in \rr^+$ such that $\bu_0<t_0+1<T<T_0-1$ and $T_0-T>T-t_0$, where $T$ is to be chosen later.}

Thanks to both \eqref{eq:uadinfinto} and $\bu_0=\limsup_{x\to \partial  U} \bu(x)$ we have that $\{t_0<\bu <T_0\}$ is compactly contained in $ U$.
Hence we can pick a cut off function $\eta\in \test (\X )$ such that $\eta=1 $ in $\{t_0\le \bu\le T_0\}$, $0\le \eta \le 1$, $\supp \eta \subset  U$ and  define $u\coloneqq\eta \bu \in \test(\X).$ As above we consider $F_t$ the flow relative to $-\nabla u$, which is defined for all positive and negative times. 

Define for every $a,b\in [\bu_0,\infty)$ the open set
\[A_{a,b}\coloneqq\{a<\bu<b\}.\]
From the definition of $u$, the hypotheses on $\bu$ and the locality of the gradient and the Laplacian we have
\begin{equation}\label{functioncone}
\begin{split}
&\Delta u=N, \quad \mea\text{-a.e. in } A_{t_0,T_0}.\\
&|\nabla u|^2=2u, \quad  \mea\text{-a.e. in } A_{t_0,T_0}.
\end{split}
\end{equation}
The following can be proven arguing as in \cite[sec. 3.6.1]{volcon}, however we give a shorter proof, which use the improved Bochner inequality \eqref{eq:improvedboch}.
\begin{prop}
	\begin{equation}
		\H{u}={\sf id}  \quad \mea\text{-a.e. in } A_{t_0,T_0}.
	\end{equation}
\end{prop}
\begin{proof}
	Localizing \eqref{eq:improvedboch} to $A_{t_0,T_0}$ and recalling \eqref{functioncone} we obtain
	\[
	N\ge |\H u|_{HS}^2+\frac{(N-{\sf tr}\H u)^2}{N-\dim(\X)} \quad \mea\text{-a.e. in } A_{t_0,T_0}.
	\]
	By  Cauchy-Swartz and recalling \eqref{eq:coordHS}, \eqref{eq:coordtr}  we observe that $|\H u|_{HS}^2=\sum_{1\le i,j\le \dim(\X)} \H u(e_i,e_j)^2\ge \sum_{i=1}^{\dim(\X)}\H u(e_i,e_i)^2 \ge \frac{{\sf tr}\H u^2}{N}$. Plugging this in the above inequality and applying again Cauchy-Swartz we obtain
	\[
	N\ge \frac{{\sf tr}\H u^2}{N}+\frac{(N-{\sf tr}\H u)^2}{N-\dim(\X)}\ge N \quad \mea\text{-a.e. in } A_{t_0,T_0}.
	\]
	Hence all the inequality we used were actually equalities, in particular $\H u(e_i,e_j)=0$ $\mea\text{-a.e. in } A_{t_0,T_0}$, for every $i\neq j$ and $\H u(e_i,e_i)=1$ $\mea\text{-a.e. in } A_{t_0,T_0}$ for every $i=1,...,\dim(\X)$, which concludes the proof.
\end{proof}

\begin{prop}\label{prop:basics}
	\begin{enumerate}
		\item For $\mea$-a.e. $x\in X$ it holds that $F_t(F_s(x))=F_{s+t}(x)$ for every $s,t \in \rr.$
		\item For $\mea$-a.e. $x\in A_{t_0,T_0}$ it holds that $F_t(x)\in A_{t_0,T_0}$ and
		\begin{equation}\label{ualong}
		u(F_t(x))=e^{-2t}u(x),
		\end{equation}
		for every $t\in (\frac{1}{2}\log{\frac{u(x)}{T_0}},\frac{1}{2}\log{\frac{u(x)}{t_0}}),$
		moreover
		\begin{equation}\label{distancealong}
		\sfd(F_s(x),F_t(x))=|e^{-t}-e^{-s}|\sqrt{2u(x)},
		\end{equation}
		for every $s,t\in (\frac{1}{2}\log{\frac{u(x)}{T_0}},\frac{1}{2}\log{\frac{u(x)}{t_0}}),$ in particular the curve $(\frac{1}{2}\log{\frac{u(x)}{T_0}},\frac{1}{2}\log{\frac{u(x)}{t_0}}) \ni t \mapsto F_t(x)$ is supported on a geodesic.
	\end{enumerate}
\end{prop}
\begin{proof}
Everything except for the equality in \eqref{distancealong} follow Proposition \ref{prop:basics0} together with the observation that in this case $a_x= \frac{1}{2}\log{\frac{u(x)}{T_0}}$ and $b_x= \frac{1}{2}\log{\frac{u(x)}{t_0}}$. 

To show equality in \eqref{distancealong} it is enough to show it for $t=0$ and $s>0$. Hence we fix $s \in (0,\frac{1}{2}\log{\frac{u(x)}{t_0}})$. Thanks to \eqref{distancealong0} we only need to show that
\begin{equation}\label{eq:partialeuqality}
	\sfd(F_s(x),x)\ge (1-e^{-s})\sqrt{2u(x)}.
\end{equation}
We make the intermediate claim that
\begin{equation}\label{eq:boundarydistance}
\sfd(x, U^c)\ge \sqrt{2\bu(x)}-\sqrt{2\bu_0}, \quad \forall x \in  U.
\end{equation}
To prove it we first observe that, since $\bu$ is positive, we have that $\sqrt{\bu}\in \W_{\loc}( U)$ and $|\nabla \sqrt{2\bu}|=1$ $\mea$-a.e. in $ U$.

The properness of the space $\X$ ensures that there exists $\bar x\in \partial  U $ such that $\sfd(x,\bar x)=\sfd(x,\partial  U)$. Moreover, since $\X$ is geodesic, there exists a sequence $x_n\subset  U$ such that $x_n\to x$ and $\sfd(x,x_n)\le\sfd(x,\partial  U)$. Hence recalling that $|\nabla \sqrt{2\bu}|=1$ $\mea$-a.e., we are in position to apply \eqref{eq:localsobolevtolip} and deduce that
\[
\sfd(x,\bar x)=\lim_n \sfd(x,x_n)\ge \liminf_n \sqrt{2\bu(x)}-\sqrt{2\bu(x_n)}\ge \sqrt{2\bu(x)}-\sqrt{2\bu_0},
\]
which proves \eqref{eq:boundarydistance}.

From \eqref{distancealong0} and by how we chose $s$ we have that $\sfd(F_s(x),x)\le (1-e^{-s})\sqrt{2\bu(x)}\le \sqrt{2\bu(x)}-\sqrt{2t_0}\le  \sqrt{2\bu(x)}-\sqrt{2\bu_0}.$ Hence from \eqref{eq:boundarydistance} we deduce
$
\sfd(F_s(x),x)\le \sfd(x,\partial  U)
$
and applying again \eqref{eq:localsobolevtolip} combined with \eqref{ualong} we obtain \eqref{eq:partialeuqality}.
\end{proof}

\begin{lemma}\label{measureannulilemma}
	For every $t \in (0,\frac{1}{2}\log \frac{T_0}{t_0})$ it holds that
	\begin{equation}\label{volannuli}
	\mea(A_{e^{2t}t_0,T_0})=e^{Nt}\mea(A_{t_0,e^{-2t}T_0}).
	\end{equation}
\end{lemma}
\begin{proof}
	The argument is essentially the same as in \cite[Prop. 3.7]{volcon} but reversed, indeed here we start from $\Delta u=N$ and $|\nabla u|^2=2u$ and deduce information on the measure.
\end{proof}

The following result has not a direct counterpart in \cite{volcon}, however it morally substitutes the bound (3.1) in \cite[Prop. 3.2]{volcon} (cf. with \eqref{rhotestimate} below). We remark that the proof of the following Proposition relies on the local estimate of Proposition \ref{prop:localest}.

\begin{prop}\label{prop:measureannuli}
	For every $t \in (0,\frac{1}{2}\log \frac{T_0}{t_0})$ it holds that
	\begin{equation}\label{measureannuli}
		({F_t}_*\mea)\restr{A_{t_0,e^{-2t}T_0}}={F_t}_*(\mea\restr{A_{e^{2t}t_0,T_0}})=e^{Nt}\mea\restr{A_{t_0,e^{-2t}T_0}}.
	\end{equation}
\end{prop}
\begin{proof}
	Consider the probability measure $\mu_0=\frac{\mea\restr{A_{e^{2t}t_0,T_0}}}{\mea(A_{e^{2t}t_0,T_0})}.$ By Theorem \ref{thm:uniqrfl}   $\{{F_s}_*\mu_0\}_{s\in[0,t]}$ are all Borel probability measures, absolutely continuous with respect to $\mea$ and solve the continuity equation with initial datum $\mu_0$. Moreover  \eqref{ualong} implies that $ {F_s}_*\mu_0$ is concentrated on  $A_{t_0,T_0}$ for every $s \in [0,t]$. Therefore setting ${F_t}_*\mu_0=\rho_t\mea $ we are in position to apply \eqref{localest}, that combined with \eqref{functioncone} gives
	\begin{equation}\label{rhotestimate}
	\|\rho_t\|_{\infty} \le \mea(A_{e^{2t}t_0,T_0})^{-1}e^{Nt}.
	\end{equation}
	However applying now \eqref{volannuli} and observing that and  that 
	${F_t}_*\mu_0$ is concentrated in $A_{t_0,e^{-2t}T_0}$, again thanks to \eqref{ualong}, we can compute
	\[ 0\le \int_{A_{t_0,e^{-2t}T_0}}  \mea(A_{e^{2t}t_0,T_0})^{-1}e^{Nt}-\rho_t \d \mea = \mea(A_{e^{2t}t_0,T_0})^{-1}e^{Nt}\mea(A_{t_0,e^{-2t}T_0})-1=0,  \]
	that gives the second in \eqref{measureannuli}. 
	
	The first in \eqref{measureannuli} follows directly from \eqref{ualong}.
\end{proof}
Having at our disposal \eqref{measureannuli} and \eqref{distancealong}, we can argue exactly as in \cite[Cor. 3.8]{volcon} to obtain the following 

\begin{prop}[Continuous disintegration]\label{disintlemma}
	We have 
	\begin{equation}\label{pushforward}
	u_*\mea\restr{A_{t_0,T_0}}=cr^{\frac{N}{2}-1} \mathcal{L}^1\restr{(t_0,T_0)},
	\end{equation}
	where $c\coloneqq\frac{N}{2}\frac{\mea(A_{t_1,t_2})}{t_1^{N/2}-t_2^{N/2}}$, for any $t_1,t_2 \in \rr^+$ with $t_0\le t_2<t_2\le T_0$. Moreover there exists a weakly continuous family of Borel measures $(t_0,T_0)\ni r\mapsto\mea_r \in \mathcal{P}(X)$ such that
	\begin{equation}\label{disintegration}
	\int \phi \d \mea = c \int_{t_0}^{T_0} \int \phi \d \mea_r r^{N/2-1} \d r, \quad  \forall \phi  \in C_c(A_{t_0,T_0}).
	\end{equation}
	Finally, for every $t	\in (0,\log \frac{T_0}{t_0})$ the measures $\mea_r$ satisfies
	\begin{equation}\label{partialmeasures}
	{F_t}_*\mea_r=\mea_{e^{-2t}r}, \quad \text{ for a.e. } r\in (e^{2t}t_0,T_0).
	\end{equation}

\end{prop}
The following result has not a counterpart in \cite{volcon}, since it deals with large scales, while the analysis in \cite{volcon} is local.
\begin{cor}\label{cor:euclvol}
	$\X$ has Euclidean volume growth, in particular $\{\bu>\bu_0\}$ is unbounded, connected  with $\{\bu>\bu_0\}^c$ bounded.
\end{cor}
\begin{proof}
	Combining \eqref{ualong} and \eqref{distancealong} it can be shown that  $A_{t_0,T_0}\subset B_{4\sqrt{T_0}+C}(x_0)$  for every $T_0>2t_0$ for some fixed constant $C>0$ (recall that what we proved so far holds for an arbitrary $T_0>T$). Therefore $A_{t_0,\frac{(R-C)^2}{16}}\subset B_R(x_0)$ for every $R$ big enough and the conclusion follows using \eqref{disintegration}.
	
	Since $\X$ has Euclidean volume growth it is not compact and not a cylinder in the sense of $i)$ of Proposition \ref{prop:ends}. Now observe that $\partial \{\bu>\bu_0\}$ is bounded (as a consequence of \eqref{eq:uadinfinto}), hence compact, and that each connected component of $\{\bu>\bu_0\}$ is unbounded (by \eqref{eq:unbounded components}). Hence the conclusion follows from Proposition \ref{prop:ends} and the fact that $\{\bu>\bu_0\}$ is not empty.
\end{proof}

\begin{lemma}[{\cite[Lemma 3.11]{volcon}}]\label{strongderivative}
	Let $f\in L^p(\mea)$ with $p<+\infty$, then the map $t\mapsto f\circ F_t$ is continuous in $L^p(\mea).$ Moreover if $f\in \W(X)$ the  map $t \mapsto f\circ F_t$ is $C^1$ in $L^2$ and its derivative is given by
	\[\frac{\d}{\d t} f\circ F_t=-\langle \nabla f , \nabla u \rangle \circ F_t. \]
\end{lemma}

\subsection{Effect on the Dirichlet Energy}
Thanks to \eqref{functioncone} and \eqref{measureannuli} we can repeat almost verbatim  the analysis done in \cite[Sec. 3.2]{volcon}. Indeed all the proofs contained there rely only on the analogous properties of the function ${\sf b}$ and ${\sf Fl}_t$, i.e. ``$|D {\sf b} |^2=2{\sf b},\Delta {\sf b}=N$ and ${{\sf Fl}_t}_*\mea=e^{Nt}\mea$".

This said, we will only state, adapted to our case and without proof,  the final result in \cite[Sec. 3.2]{volcon} (i.e. Corollary 3.17), since it is the only statement that is needed for the rest of the argument.

\begin{theorem}\label{gradientcor}
	Let  $t	\in (0,\log \frac{T_0}{t_0})$,  and $f \in L^2(\mea)$ with support in $A_{t_0,e^{-2t}T_0}$.  Then $f \in \W(X)$ if and only if $f\circ F_t \in \W(X)$ and in this case
	\begin{equation}\label{gradientalong}
	|\nabla (f \circ F_t)|= e^{-t}|\nabla f |\circ F_t, \mea-a.e.
	\end{equation}
\end{theorem}

\subsection{Precise representative of the flow}

The following proposition is the analogous of \cite[Thm. 3.18]{volcon}. We point out that the proof in \cite{volcon} contains an oversight in the proof that ${\sf Fl}_t$ has a locally-Lipschitz representative. Indeed it is  claimed that this follows from the fact that  ${\sf Fl}_t$ is Lipschitz in  ${\sf Fl}_t^{-1}(B_r(x_0))$ for every small enough ball $B_r(x_0)$. However, since  ${\sf Fl}_t$ is not yet proven to be continuous, we do not know enough information on the sets ${\sf Fl}_t^{-1}(B_r(x_0))$ to  `patch' them and obtain the claimed local Lipschitzianity. 

For this reason we will give a complete proof which also fixes the  original argument. 
\begin{prop}\label{lipflow}   
	 Let  ${\cal U}_{t_0,T_0}\subset \rr\times\X$ be the open set given by
	   $${\cal U}_{t_0,T_0}\coloneqq\left \{ (x,t) \ : x \in A_{t_0,T_0} \text{ and } t \in  \left (\frac{1}{2}\log \frac{u(x)}{T_0} ,\frac{1}{2}\log \frac{u(x)}{t_0} \right )\right \}.$$ 
	   Then the map 
	 \[ F : {\cal U}_{t_0,T_0}\to  A_{t_0,T_0},\]
	   has a continuous representative w.r.t the measure $\mathcal{L}^1\otimes \mea $. Moreover for such representative (which we denote again by $F$) the map $F_t :  A_{e^{2t}t_0,T_0}\to A_{t_0,e^{-2t}T_0}$ is locally $e^{-t}$-Lipschitz having $F_{-t}$ as its inverse, which is locally $e^{t}$-Lipschitz. Also for every $x \in A_{t_0,T_0}$ and every $s,t \in (\frac{1}{2}\log \frac{u(x)}{T_0} ,\frac{1}{2}\log \frac{u(x)}{t_0} )$ 
	\begin{equation}\label{distancealongpoint}
	\sfd(F_t(x),F_s(x))=|e^{-s}-e^{-t}|\sqrt{2u(x)}
	\end{equation}
	and
	\begin{equation}\label{ualongpoint}
	u(F_t(x))=e^{-2t}u(x).
	\end{equation}
	Finally for every $t \in (0,\log \frac{T_0}{t_0})$ and every curve $\gamma$ with values in $A_{e^{2t}t_0,T_0}$, putting $\tilde \gamma:= F_t\circ \gamma$ we have
	\begin{equation}\label{speedalong}
	|\overset{.}{\tilde \gamma}_s|=e^{-t}|\overset{.}{\gamma}_s|, \quad \text{ for a.e. } s \in [0,1],
	\end{equation}  
	meaning that one is absolutely continuous if and only if the other is absolutely continuous, in which case \eqref{speedalong} holds.
\end{prop}

\begin{proof}
	Fix $t	\in [0,\log \frac{T_0}{t_0})$. We start claiming that 
	\begin{equation}\label{continuousrepr}
	\begin{split}
	 &F_t\restr{A_{e^{2t}t_0,T_0}}\,\, \text{ has a continuous representative that we denote by $\bar F_t$ and}\\
	 &\bar F_t(A_{e^{2t}t_0,T_0})=A_{t_0,e^{-2t}T_0}.
	\end{split}
	\end{equation}
	For the first part it is sufficient to show that for every $a,b\in \rr$ such that $e^{2t}t_0<a<b<T_0$ the map $F_t\restr{A_{a,b}}$ has a continuous representative. Hence we fix such $a,b \in \rr$ and define the open  sets $A'\coloneqq A_{e^{-2t}a,e^{-2t}b}$ and   $A\coloneqq A_{t_0,e^{-2t}T_0}$. Observe that the continuity of $u$ implies $\sfd(A',A^c)\eqqcolon\delta>0.$ Consider now the countable family of 1-Lipschitz functions $\mathcal D \subset \LIP(X)$ defined as 
	\[\mathcal D\coloneqq \{ f_{n,k}  \ | \ n,k\in \mathbb{N}\}=\{ \max(\min(\sfd(.,x_n),k-\sfd(.,x_n)) ,0) \ | \ n,k\in \mathbb{N}\},\]
	where $\{x_n\}_{n\in \mathbb{N}}$ is dense subset of $A.$  We pick a cut off function $\eta \in \LIP_c(A)$ such that  $0\le \eta \le 1$ and $\eta \equiv 1 $ in $A'$ and define the set $\eta \mathcal D \subset \LIP_c(A) $ as $\eta \mathcal D\coloneqq\{ \eta f \ | \ f \in  \mathcal D\}.$ For any $f \in \mathcal{D}$ it holds
	$$\text{Lip} (f\eta)  \le \Lip \eta \sup_A |f|+\Lip f\le \Lip \eta \,  \text{diam}(A)+1\eqqcolon L,$$
	hence the functions in $\eta \mathcal{D}$ are $L$-Lipschitz. We now make the key observation that 
	\begin{equation}\label{liptod}
	\sfd(x,y)=\sup_{f \in  \mathcal{D}} |f(x)-f(y)|=\sup_{f \in  \mathcal{D}} |\eta f(x)-\eta f(y)|=\sup_{f \in  \eta\mathcal{D}} |f(x)- f(y)|,
	\end{equation}
 	for every $x,y \in A'.$ Thanks to Corollary \eqref{gradientcor} we know that $f\circ F_t \in \W(X)$ for every $f \in \eta \mathcal{D}$ and $|D(f\circ F_t)|=e^{-t}|Df|\circ F_t\le e^{-t}L$ $\mea$-a.e. Then from the Sobolev-to-Lipschitz property of $X$, we deduce that $f \circ F_t$ has an $L$-Lipschitz representative. Thus there exists an $\mea$-negligible set $N\subset X$ such that  for every $f \in \eta \mathcal{D}$ the restriction of $f\circ F_t$ to $X\setminus N$ is L-Lipschitz. Moreover from \eqref{ualong} it follows the existence of an $\mea$-negligible set $N'$ such that $F_t(A_{a,b}\setminus N')\subset A'.$ Therefore from \eqref{liptod} it follows that for every $x,y \in A_{a,b}\setminus(N\cup N')$
 	\[\sfd(F_t(x),F_t(y))=\sup_{f \in  \eta\mathcal{D}} |f(F_t(x))- f(F_t(y))|\le e^{-t}L \sfd(x,y). \]
	This proves the first part of \eqref{continuousrepr}. We now show $\subset$ of the second part. From \eqref{ualong} it follows the existence of a negligible set $N$ such that for every set  $U$  relatively compact in $A_{e^{2t}t_0,T_0}$ we have that $\bar F_t(U\setminus N) $ is relatively compact in $A_{t_0,e^{-2t}T_0}.$ Moreover, since negligible sets have empty interior we deduce that $\overline {U\setminus N}$ contains $U.$ Therefore $\bar F_t(U)\subset \bar F_t(\overline {U\setminus N})\subset \overline{\bar F_t(U\setminus N)}$ which is contained in $A_{e^{2t}t_0,T_0}$ thanks to the first observation. We now show $\supset$. Again thanks to \eqref{ualong} the set $N\coloneqq A_{t_0,e^{-2t}T_0}\setminus \bar F_t(A_{e^{2t}t_0,T_0})$ is negligible. Pick  any set $U$ relatively compact in $A_{t_0,e^{-2t}T_0}$ and define $V \coloneqq \bar F_t^{-1}(U\setminus N)$ which is relatively compact in $A_{e^{2t}t_0,T_0}.$ Therefore, since $\bar F_t$ is continuous, the set  $F_t(\bar V)$ is compact in $A_{t_0,e^{-2t}T_0}$ and,  since negligible sets have empty interior, contains $U.$ This concludes the proof of \eqref{continuousrepr}.
	
	For any $t \in (\log \frac{t_0}{T_0},0)$ we can now argue exactly as above, to deduce that 
	\begin{equation}\label{continuousinv}
	\begin{split}
	&F_t\restr{A_{t_0,e^{2t}T_0}}\,\, \text{ has a continuous representative that we denote by $\bar F_t$ and}\\
	&\bar F_t(A_{t_0,e^{2t}T_0})=A_{e^{-2t}t_0,T_0}.
	\end{split}
	\end{equation}
	In particular, since $F_t(F_{-t})=\sf id$ $\mea$-a.e., we deduce that $\bar F_{-t}$ is the continuous inverse of $\bar F_t.$
	
Having proved \eqref{continuousrepr} and \eqref{continuousinv} we can complete the proof arguing as in \cite{volcon} with the obvious modifications.
\end{proof} 

\textbf{From now on we denote by $F$ a representative of $F: \rr \times \X \to \X$ which is continuous (in space and time) on ${\cal U}_{t_0,T_0}$.}

\subsection{Properties of level set $\{\bu=T\}$}\label{sec:levelset}
In this short subsection we prove that the level set $\{\bu =T\}$ is Lipschitz-path connected when $T$ is big enough. We remark that the argument  will rely  on Proposition \ref{prop:ends} and is different from the one used in \cite{volcon} to prove  the same property for the ``sphere".  

For the following result recall  from Section \ref{sec:flow di u} that in our construction we first choose $T$ and then  we choose $t_0$ and $T_0$ accordingly.
\begin{lemma}\label{prop:levelset}
	There exists $T>\bu_0$ and a constant $c>0$ (depending only on $T,t_0$ and $\bu$) such that, if $T_0$ is big enough, for every couple of points $x,y \in \{\bu =T\}$  there exists  $\gamma\in\LIP([0,1],\X)$ joining $x$ and $y$ and such that $\gamma \subset A_{t_0+1,T_0-1 }$ and $\Lip(\gamma)\le 5\sfd(x,y).$
\end{lemma}
\begin{proof}
	From \eqref{eq:uadinfinto} we know that $\sfd(\{\bu = T\},\{\bu>t_0+1\}^c)\to+\infty$ as $T\to+\infty$. Moreover  $\{\bu>t_0+1\}^c$ is bounded and, as proven in Corollary \ref{cor:euclvol}, $\X$ has Euclidean volume growth and in particular it is not a cylinder (since $N\ge 2$). Therefore from Proposition \ref{prop:ends}, provided $T$ is big enough, we have that  for every couple of points $x,y \in \{\bu =T\}$  there exists a Lipschitz path $\gamma$ from $x$  to  $y$ and such that $\gamma \subset \{\bu>t_0+1\}$ and $\Lip (\gamma)\le 5\sfd(x,y)$. Therefore,  again from \eqref{eq:uadinfinto} and since $\{\bu=T\}$ is bounded, if $T_0$ is big enough then it holds that $\gamma\subset A_{t_0+1,T_0-1}$.
\end{proof}

\newcommand{\pr}{{\sf{Pr}}}
Consider now the closed set $S_T\coloneqq\{\bu =T\}$ and the projection map $\pr : A_{t_0,T_0}\to S_T$ defined as
\[\pr(x)\coloneqq F_{\frac{1}{2}\log{\frac{u(x)}{T}}}(x).\]
 Note that from Proposition \ref{lipflow} and \eqref{distancealongpoint}, we have that $\pr$ is well defined and locally Lipschitz.
 \begin{prop}\label{prop:arc connected}
 	The set $S_T$ is Lipschitz path connected, meaning that for every couple of points $x,y \in S_T$ there exists a Lipschitz curve $\gamma$ taking values in $S_T$, joining $x$ and $y$.  Moreover we can choose $\gamma$ so that $\Lip(\gamma)\le c\,\sfd(x,y),$ for some uniform constant $c>0.$
 \end{prop}
\begin{proof}
	From Proposition \ref{prop:levelset} we now that there exists a Lipschitz path connecting $x$ and $y$ taking values in $A_{t_0+1,T_0-1},$ which is relatively compact in $A_{t_0,T_0},$ then we simply consider the curve 
	\[\tilde \gamma \coloneqq \pr \circ \gamma,\]
	which remains a Lipschitz curve, since $\pr$ is locally Lipschitz. The claimed bound on $\Lip(\tilde \gamma)$ follows from the bound on $\Lip(\gamma)$ given in Proposition \ref{prop:levelset} and again by the local Lipschitzianity of $\pr.$
\end{proof}

\subsection{The cosine formula holds}\label{sec:cosine law}

The goal of this section is to prove the following.
\begin{prop}[Local cosine formula]\label{prop:variables splitting}
	For every  $\eps>0$ there exists $\rho=\rho(\eps)>0$ such that for every $x \in A_{t_0+\eps,T_0-\eps}$ it holds
	\begin{equation}\label{eq:variables splitting}
		\sfd(x,y)^2=2u(x)+2u(y)-4\sqrt{u(x)u(y)}\left(1- \frac{\sfd(\pr(x),\pr(y)^2)}{4T} \right), \quad \forall y \in B_\rho(x).
	\end{equation}
\end{prop}
To prove this statement we need to construct some objects and prove some preliminary technical facts.

Fix any $x \in A_{t_0+\eps,T_0-\eps}$ and $\delta>0$ small and to be chosen. For every $r>0$ and $y \in B_\delta(x)\cup B_\delta(\pr(x))$ define the measures $\mu^r\coloneqq (\mea(B_r(\pr(x))))^{-1}\mea\restr{B_r(\pr(x))}$ and   $\nu^r\coloneqq (\mea(B_r(y)))^{-1}\mea\restr{B_r(y)}$. Let $\bar t\in \rr$  be such that $F_{\bar t}(\pr(x))=x$ (i.e.\ $\bar t=\frac12 \log \frac{T}{u(x)}$). For every $t \in [0,\bar t]$ (or in $[\bar t,0]$ if $\bar t<0$) we also set $\mu_t^r\coloneqq {F_t}_*\mu^r$. Finally for everyone of these $t$'s we also consider the unique $W_2$-geodesic $\{\eta_s^{t,r}\}_{s \in [0,1]}$ going from $\nu^r$ to $\mu_t^r$.  For the main computations, we will need the following technical fact: 
\begin{equation}\label{eq:geodesics inside}
		\text{ there exists $\rho=\rho(\eps)>0$ such that, provided $r,\delta<\rho$, $\supp(\eta_s^{t,r})\subset A_{t_0,T_0}$ for all $s \in[0,1]$}.
\end{equation}
Since the measures $\eta_s^{t,r}$ are concentrated on the support of geodesics going from $B_r(y)$ to $F_t(B_r(x))$, to check the above it is sufficient to prove the following.
\begin{prop}\label{prop:geodesics inside prop}
	For every $x \in A_{t_0+\eps,T_0-\eps}$ there exists $\rho=\rho(\eps)>0$ such that the following holds. Let $\bar t\in \rr$  be such that $F_{\bar t}(\pr(x))=x$ (i.e.\ $\bar t=\frac12 \log \frac{T}{u(x)}$). Then every geodesic $\gamma$ going from a point $y \in B_\rho(\pr(x))\cup B_\rho(x)$ to a point $x' \in F_t(B_\rho(\pr(x)))$, with $t \in [0,\bar t]$ (or $[\bar t,0]$), $\gamma$ is contained in $A_{t_0,T_0}$.
\end{prop}
This in turn will follow from the next simpler  result, which roughly says that if two points are sufficiently close, then all the geodesics connecting one of them to the flow line of the other, are contained in $A_{t_0,T_0}.$ 
\begin{lemma}\label{lem:intermediate lem}
	For every $\eps>0$ there exists $\rho=\rho(\eps)>0$ such that the following holds. For every $x,y \in A_{t_0+\eps,T_0-\eps}$ such that $\sfd(x,y)<\rho$ and every geodesic $\gamma$ going from $y$ to a point of the type $z=F_t(x)\in A_{t_0+\eps,T_0-\eps}$, $\gamma$ is contained in $A_{t_0,T_0}$.
\end{lemma}
\begin{proof}
	We first notice that arguing as in the proof of \eqref{eq:boundarydistance}, we have
	\begin{equation}\label{eq:distance St}
		\sfd(x,\{\bu=t_0 \})\ge \sqrt{2u(x)}-\sqrt{2t_0},\quad \sfd(x,\{\bu=T_0 \})\ge \sqrt{2T_0}-\sqrt{2u(x)}, \quad \forall x \in A_{t_0,T_0}.
	\end{equation}
	Fix $x,y \in A_{t_0+\eps,T_0-\eps}$ with $\sfd(x,y)<\rho$ and fix some point $z=F_t(x)\in  A_{t_0+\eps,T_0-\eps}.$  We will only consider the case $t>0$, since the other one is analogous.  Notice that in this case $\sfd(x,z)=\sqrt{2u(x)}-\sqrt{2u(z)}.$ Moreover, since $u$ is Lipschitz,  $|\sqrt{2u(y)}-\sqrt{2u(x)}|\le C \sqrt\rho$ for some $C$ depending only on $u$.
	
	Fix a geodesic $\gamma$ going from $y$ to $z$. Suppose by contradiction that $\gamma$ exits $A_{t_0,T_0}$. This means that $\gamma_t \in \{\bu=t_0 \}\cup \{\bu=T_0 \}$ for some $t \in(0,1)$. Since $\gamma$ is a geodesic we also have
	\[
	\sfd(z,y)\ge \max\{ \sfd(\gamma_t,z), \sfd(\gamma_t,y)\}.
	\]
	Moreover from \eqref{distancealongpoint} and the triangle inequality we deduce that $\sfd(z,y)\le\sqrt{2u(x)}-\sqrt{2u(z)}+\rho.$ 
	
	We now have two possibilities: either $\gamma_t \in \{\bu=t_0 \}$ or $\gamma_t \in \{\bu=T_0 \}$. In the first one, combining the above observations we have
	\[
	 \sqrt{2u(x)}-\sqrt{2u(z)}+\rho\ge \sfd(z,y)\ge  \sfd(\gamma_t,y)\ge  \sqrt{2u(y)}-\sqrt{2t_0}\ge \sqrt{2u(x)}-C\sqrt \rho-\sqrt{2t_0},
	\]
	which gives 
	$$\sqrt{2t_0}+\rho +C\sqrt \rho\ge \sqrt{2u(z)}\ge \sqrt{2(t_0+\eps)},$$ that is clearly a contradiction provided $\rho$ is chosen small enough. In the case that $\gamma_t \in \{\bu=T_0 \}$ we analogously have
	\[
	  \sqrt{2(T-\eps)}-\sqrt{2u(z)}+\rho\ge \sqrt{2u(x)}-\sqrt{2u(z)}+\rho\ge \sfd(z,y)\ge  \sfd(\gamma_t,z)\ge  \sqrt{2T_0}-\sqrt{2u(z)},
	\]
	which is again a contradiction if $\rho$ is small enough.
\end{proof}

We can now prove Proposition \ref{prop:geodesics inside prop}, which as observed above, proves also \eqref{eq:geodesics inside}.
\begin{proof}[Proof of Proposition \ref{prop:geodesics inside prop}]
	Since  $F_{\bar t}$ is locally Lipschitz in $A_{t_0,T_0}$ (recall Proposition \ref{lipflow}), there exists $\delta=\delta(\eps)>0$ such that 
	$$F_{\bar t}(B_\delta(\pr(x))\subset A_{t_0+\eps/2,T_0-\eps/2}$$ Moreover, since $u$ is monotone along flow lines, we also have $F_t(B_\delta(\pr(x))\subset A_{t_0+\eps/2,T_0-\eps/2}$ for every $t \in [0,\bar t]\cup [\bar t,0]$. Let now $\rho=\rho(\eps/2)$ be the one given by Lemma \ref{lem:intermediate lem}  corresponding to $\eps/2$ and note that we can assume that $\rho<\delta.$  Then the conclusion  for $y \in B_\rho(\pr(x))$ follows immediately from Lemma \ref{lem:intermediate lem}. 
	
	For the cases in which $y$ is close to $x$ we observe that, again by the local Lipschitzianity of $F_{\bar t}$, there exists $\bar \rho(\eps) \in(0,\rho/2)$ so that $F_{\bar t}(B_{\bar \rho}(\pr(x))) \subset B_{\rho/2}(x)$. In particular for every $y \in B_{\bar \rho}(x)$ and $x' \in F_{\bar t}(B_{\bar \rho}(\pr(x)))$ we have $\sfd(y,x')<\rho/2+\bar \rho<\rho$. Therefore we conclude again by Lemma \ref{lem:intermediate lem} that every geodesic from $y$ to $F_t(B_{\bar\rho}(\pr(x)))$, with $t \in [0,\bar t]\cup [\bar t,0]$, is contained in $A_{t_0,T_0}.$
\end{proof}

\begin{proof}[Proof of Proposition \ref{prop:variables splitting}]
	It is sufficient to  prove that for every $\eps>0$ there exists $\rho(\eps)>0$ such that, for every $x \in A_{t_0+\eps,T_0-\eps}$ and every $y \in  B_\rho(x)\cup B_{\rho}(\pr(x))$ it holds 
	\begin{equation}\label{eq:splits partial}
			\frac{\sfd(x,y)^2}{2}=u(x)+u(y)+\sqrt{\frac{u(x)}{T}}\left(\frac{\sfd(\pr(x),y)^2}{2}-T-u(y)\right).
	\end{equation}
	Then \eqref{eq:variables splitting} follows applying \eqref{eq:splits partial} first with $x=x' \in  A_{t_0+\eps,T_0-\eps}$ and  $y=y' \in  B_{\rho'}(x)$ (for some $\rho'<\rho(\eps)$ to be chosen) and then with $x=y'$ and $y=\pr(x')$. The second application of \eqref{eq:splits partial} is possible because $\sfd(y,\pr(x))=\sfd(\pr(x'),\pr(y'))\le C_{\eps} \sfd(x',y')\le C_{\eps} \rho'< \rho(\eps)$ (recall that $\pr$ is locally Lipschitz) if we choose $\rho'(\eps)$ small enough.

	Fix $x \in A_{t_0+\eps,T_0-\eps}$ and let $\bar t\in \rr$ be such that $F_{\bar t}(\pr (x))=x$. We will assume that $\bar t>0$, since the other case is exactly the same.  Let $r,\delta<\rho(\eps)$, where $\rho(\eps)$ is given by \eqref{eq:geodesics inside}. Then we also fix $y \in B_\rho(x)\cup B_{\rho}(\pr(x))$. Finally we define the measures $\mu^r,\nu^r,\mu_t^r,\eta_s^{t,r}$ as above, for $t \in [0,\bar t]$ and $s \in[0,1]$.

	From Theorem \ref{thm:uniqrfl} we have that the curve of measures $[0,1]\ni t \mapsto \mu_t^r$ satisfies the continuity equation with vector field $-\nabla u$. Then combining \cite[Theorem 3.5]{GHcont} and \cite[Proposition 3.10]{GHcont} we have that the function $\rr \ni t \mapsto W_2^2(\mu_t^r,\nu^r)$ is absolutely continuous  and
	\begin{equation}\label{eq:derivative W}
			\frac{\d}{\d t} \frac12 W_2^2(\mu_t^r,\nu^r)=- \int \la \nabla \phi_t, \nabla u\ra \, \d \mu_t^r,\quad \text{a.e.\ $t$},
	\end{equation}
	where $\phi_t$ are suitable (Lipschitz) Kantorovich potentials from $\mu_t^r$ to $\nu.$ Moreover from the second differentiation formula along Wasserstein geodesics (\cite[Thm. 5.13]{gtamanini}) we have that, for every $t$, $[0,1]\ni s \mapsto \int u \,\d\eta_s^{t,r}$ is $C^2[0,1]$ and 
	\begin{equation}\label{eq:second order diff}
		\begin{split}
			&\frac{\d}{\d s}\int u \,\d\eta_s^{t,r}=\frac1s  \int \la \nabla u, \nabla \psi_s\ra \,\d\eta_s^{t,r}, \quad \forall s \in[0,1],\\
			&\frac{\d^2}{\d s^2}\int u \,\d\eta_s^{t,r}= \frac1{s^2}\int  \la \H{u}(\nabla \psi_s,\psi_s) \d \eta_s^{t,r},\quad \forall s \in[0,1],\\
		\end{split}
	\end{equation}
where $\psi_s$ are any choice of Kantorovich potentials from $\eta_s^{t,r}$ to $\nu^r.$ We now fix  $t\in [0,\bar t]$ such that the formula in \eqref{eq:derivative W} holds. Since $\eta_1^{t,r}=\mu_t^r$, we can combine it with the first in \eqref{eq:second order diff}  at $s=1$ to get
\begin{equation}\label{eq:magic identity}
	\frac{\d}{\d s}\restr{s=1} \int u \,\d\eta_s^{t,r}=-\frac{\d}{\d t}\restr{t=1} \frac12 W_2^2(\mu_t^r,\nu^r).
\end{equation}

Our goal is now to derive an explicit expression for $\frac{\d}{\d s}\restr{s=1} \int u \,\d\eta_s^{t,r}$. To do so we use the second in \eqref{eq:second order diff} and the fact that $\H u={\sf id}$ on $\supp(\eta_s^{t,r})$ (which is ensured by \eqref{eq:geodesics inside})
\[
\frac{\d^2}{\d s^2}\int u \,\d\eta_s^{t,r}=\frac1{s^2}\int  \la \H{u}(\nabla \psi_s,\psi_s) \d \eta_s^{t,r}= \frac1{s^2} \int |\nabla \psi_s|^2 \, \d \eta_s^{t,r}\quad \forall s \in[0,1].
\]
We now apply the metric version of Brenier theorem (\cite{AGSflow}, see also \cite{AR}) to obtain that $\int |\nabla \psi_s|^2 \, \d \eta_s^{t,r}=W_2^2(\eta_s^{t,r},\nu^r)$ for every $s \in [0,1]$, which combined with the fact that $\eta_s^{t,r}$ is a $W_2$-geodesic from $\mu_t^{r}$ to $\nu^r$ gives
\[
\frac{\d^2}{\d s^2}\int u \,\d\eta_s^{t,r}=W_2^2(\mu_t^{r},\nu^r), \quad \forall s \in[0,1].
\]
Since $s\mapsto\int u \,\d\eta_s^{t,r}$ is $C^2[0,1]$, we must have $\int u \,\d\eta_s^{t,r}=a+bs+W_2^2(\mu_t^{r},\nu^r) \frac{s^2}{2}$ for every $s\in[0,1]$. Substituting $s\in\{0,1\}$ we deduce that
\[
	\int u\d \eta_s^t=\int u \d \nu^r+\left(\int u \d \mu_t^r-\int u \d \nu^r-\frac{W_2(\mu_t^r,\nu^r)^2}{2}\right)s+W_2^2(\mu_t^r,\nu^r)\frac{s^2}{2}, \quad \forall s \in[0,1],
\]
using which we can finally compute 
\[
\frac{\d}{\d s}\restr{s=1} \int u \d \eta_s^t=\int u \d \mu_t^r-\int u \d \nu^r+\frac{W_2(\mu_r^t,\nu^r)^2}{2}.
\]
Combining this with \eqref{eq:magic identity} we obtain that, setting $f(t)\coloneqq W_2(\mu_t^r,\nu_r)^2/2$, $t \in[0,\bar t]$, $f$ is absolutely continuous and
\[
f'(t)=\int u \d \nu_r-\int u \d \mu_t^r-f(t), \quad \text{a.e.\ }t \in(0,\bar t).
\]
Note that the right hand side is actually a continuous function in $t$, hence $f\in C^1[0,1].$ Since we know the value of $f(0)$ we can  now find $f$ explicitly:
\[
\frac{W_2(\mu_t^r,\nu_r)^2}{2}=f(t)=e^{-t} \int_0^t e^{s}\left(\int u \d \nu_r-\int u \d \mu_s^r\right) \d s+\frac{W_2(\mu_0^r,\nu_r)^2}{2} e^{-t}, \quad \forall \, t \in[0,\bar t].
\]
We now let $r\to 0^+$ to obtain
\[
\frac{\sfd(F_t(\pr(x)),y)^2}{2}=e^{-t} \int_0^t e^{s}(u(y)-u(F_s(\pr(x))))\d s+\frac{\sfd(\pr(x),y)^2}{2}e^{-t},\quad \forall \, t \in[0,\bar t].
\]
Note that to pass the limit inside the integral sign, we can use that $\left |\int u \d \mu_s^r\right |\le \int |u|\circ F_s \d \mu^r\le \|u\|_\infty$ and the dominated convergence theorem. Plugging in $u(F_s(\pr(x)))=e^{-2s}u(\pr(x))=e^{-2s}T$ (see \eqref{ualongpoint}) and choosing $t=\bar t$, we obtain
\begin{align*}
		\frac{\sfd(x,y)^2}{2}=\frac{\sfd(F_{\bar t}(\pr(x)),y)^2}{2}&=u(y)e^{-\bar t}(e^{\bar t}-1)+Te^{-\bar t}(e^{-\bar t}-1)+\frac{\sfd(\pr(x),y)^2}{2}e^{-\bar t}=\\
		&=u(x)+u(y)+e^{-\bar t}\left(\frac{\sfd(\pr(x),y)^2}{2}-T-u(y)\right),
\end{align*}
that is precisely \eqref{eq:splits partial}, since $\bar t=\frac12 \log \frac{T}{u(x)}$.
\end{proof}

\subsection{Intrinsic metric on the level set}\label{sec:intrinsic metric}

\begin{definition}\label{d'}
	We put $\X'\coloneqq S_T.$ For $x',y' \in S_T$ we define $\sfd'(x',y')$ as 
	\begin{equation}\label{spheredistance}
		\sfd'(x',y')^2\coloneqq \inf \int_0^1 |\dot{\gamma}_t|^2 \d t,
	\end{equation}
	where the infimum is taken among all Lipschitz path $\gamma: [0,1]\to X'\subset X$ joining $x'$ and $y'$ and the metric speed is computed w.r.t. the distance $\sfd$.
\end{definition}

\begin{lemma}\label{lem:dd'}
	\begin{equation}\label{dd'}
		\sfd(x,y)\le \sfd'(x,y)\le c \sfd(x,y), \quad \text{	for every $x,y \in \X'$,}
	\end{equation}
	where $c>0$ is the constant given in Proposition \ref{prop:arc connected}
\end{lemma}
\begin{proof}
	The first in \eqref{dd'} is immediate from the definition of $\sfd',$ while the second follows directly from Proposition \ref{prop:arc connected}.
\end{proof}
\begin{cor}\label{equivtop}
	The topology induced by $\sfd'$ on $X'$ is the same as the one induced by the inclusion $X'\subset X.$
\end{cor}

We now define the measure $\mea '$ on $X'$ as
\begin{equation}\label{m'}
	\mea'\coloneqq \mea_{T},
\end{equation}
where $\mea_T$ is given in Proposition \ref{disintlemma}. Observe that, thanks to Corollary \ref{equivtop}, $\mea '$ is a Borel probability measure on $X'.$

A straightforward computation, exploiting \eqref{disintegration}, gives also that

\begin{equation}\label{eq:relmisure}
	{\Pr}_* \mea|_{A_{a,b}}=c_{a,b}\, \mea', 
\end{equation}
for every $a,b \in \rr$ such that $t_0\le a<b\le T_0$, where $c_{a,b}=c\int_a^b r^{N/2-1}\d r,$ with $c$ as in Proposition \ref{disintlemma}.

The only step that remains to conclude is to link $\sfd'$ to the formula for the distance with separation of variables given by \eqref{eq:variables splitting}.
To do so we define a new ``distance'' on $S_T$ by:
\[
{\sf D}(x,y)\coloneqq\arccos\left (1-\frac{\sfd(x,y)^2}{4T}\right), \quad \forall\, x,y \in S_T \text{ such that } \sfd(x,y)^2<4T,
\]
where we take the range of $\arccos(.)$ to be in $(-\pi/2,\pi/2).$ We point out that ${\sf D}$ might not be a distance on the whole $S_T$, however we will show that it is so at least locally. Note also that $c^{-1}\sfd(x,y)\le \dis(x,y)\le c\sfd(x,y)$ for some constant $c=c(T)>1.$

Our main goal is to prove the following.
\begin{prop}[$\dis$ locally coincides with $\sfd'$]\label{prop:distances coincide}
	For every $p \in S_T$ there exists $r>0$ such that  
		\begin{equation}\label{eq:identification of d}
		\frac{\sfd'(x,y)}{\sqrt{2T}}=\dis(x,y), \quad \forall x,y \in B_r(p)\cap S_T.
	\end{equation}
\end{prop}
To prove the above proposition we need first to show that $\dis$ is locally an intrinsic distance.

\begin{prop}[$\dis$ is locally a geodesic distance]\label{prop:geodist}
	For every $p \in S_T$ there exists $\delta>0$ such that 
	\begin{enumerate}[label=\roman*)]
		\item\label{it:dist} $\dis(.,.)$ is a distance when restricted to $B_\delta(p)\cap S_T$,
		\item\label{it:geod} there exists a constant $\lambda=\lambda(T)<1$ such that for every $x,y \in B_{\lambda \delta}(p)\cap S_T$ there exists a Lipschitz curve $\gamma:[0,1]\to \subset B_{\delta}(p)\cap S_T$ that is a geodesic for $\dis$, i.e.\ such that $\gamma_0=x,\gamma_1=y$ and $\dis(\gamma_t,\gamma_s)=|t-s|\dis(\gamma_0,\gamma_1)$ for every $t,s \in [0,1].$
	\end{enumerate}
\end{prop}
\begin{proof}
	We preliminarily note  that for every $a,b \in [0,\pi/2]$ there exists $\lambda \in [0,1]$ such that 
	\begin{equation}\label{eq:simple}
		\sqrt{\lambda^2+1-2\lambda \cos(a)}+	\sqrt{\lambda^2+1-2\lambda \cos(b)}=\sqrt{2-2\cos(a+b)}.
	\end{equation}
	An immediate geometric proof of this fact is  the following:  take three  points $P,Q,S$ on the unit circle of center $O$ so that the length of the arcs $PQ$,$QS$ and $PS$ is respectively $a,b$ and $a+b$, denote by $M$ the intersection between the segments $PS$ and $QO$, then $\lambda$ is precisely the distance from $M$ to the center $O$.
	Alternatively \eqref{eq:simple} follows also by continuity and the subadditivity of $\sqrt{2-2\cos(.)}$ in $[0,\pi].$
	
	Fix now $\eps>0$ small so that $T\in (t_0+\eps,T_0-\eps)$ and fix $p \in S_T.$  Let also $\rho=\rho(\eps)>0$ be the one given by Proposition \ref{prop:variables splitting}. By continuity of $u$ and $F_t$ (recall \eqref{distancealongpoint} and \eqref{ualongpoint}) there exist $\delta=\delta(\eps,p)>0$, $\bar t=\bar t(\eps,p)>0$ both small and such that $F_t(B_\delta(p))\subset B_{\rho/2}(p)$ for all $t \in[0,\bar t].$ In particular 
	\begin{equation}\label{eq:ok split}
		\text{for every $x,y \in F_t(B_\delta)$ and every $t \in[0,\bar t]$, \eqref{eq:variables splitting} holds.}
	\end{equation}
	
\noindent \emph{Proof of \ref{it:dist}}: We only need to prove that the triangular inequality holds. We argue  by contradiction assuming that there exist $x,y,z \in B_\delta(p)\cap S_T$ such that 
	\begin{equation}\label{eq:not triangle}
		\dis(x,z)>\dis(x,y)+\dis(y,z).
	\end{equation}
We also let  $\lambda$ be the one given by \eqref{eq:simple} and corresponding to $a=\dis(x,y), b=\dis(y,z)$. Finally let $t\ge 0$ be such that $e^{-2t}T={\lambda^2 T}$. Note that $\lambda\to 1$ as $a+b\to 0$, hence up to decreasing $\delta$ we can assume that $t<\bar t.$ In particular thanks to \eqref{eq:ok split} we can apply \eqref{eq:variables splitting} to $x,F_t(y)$ and to $F_t(y),z$, that coupled with \eqref{eq:simple}  gives
\[
\sfd(x,F_t(y))+\sfd(F_t(y),z)=\sqrt{4T-4T\cos\left(\dis(x,y)+\dis(y,z)\right)}<\sqrt{4T-4T\cos\left(\dis(x,z)\right)},
\]
where we have used the strict monotonicity of $\sqrt{1-\cos(.)}$. However using again \eqref{eq:variables splitting} and the definition of $\dis$ we see that 
$\sqrt{4T-4T\cos\left(\dis(x,z)\right)}=\sfd(x,z)$, which is clearly a contradiction. This concludes the proof of  \ref{it:dist}.

\noindent \emph{Proof of \ref{it:geod}}:	It is sufficient to show that for every couple of points in $B_{\delta/2}(p)$ there exists a $\dis$-midpoint, then the conclusion follows  from standard arguments  (see e.g.\  \cite[Thm. 2.4.16]{burago}) and the fact that $S_T$ is closed (with respect to $\sfd$) and that $\dis$ is comparable to $\sfd$.

Let $x,y \in S_T\cap B_\delta(p)$. Take $z$ such that $\sfd(z,x)=\sfd(z,y)=\frac12\sfd(x,z)$, which exists because $(\X,\sfd)$ is geodesic.  We claim that $\pr(z)$ is a $\dis$-midpoint for $x$ and $y$. Observe that from \eqref{eq:ok split}, \eqref{eq:variables splitting} holds for the couple of points $x,y$; $x,z$ and $z,x$.
From \eqref{eq:variables splitting} we immediately see that
\[
\dis(\pr(z),x)=\dis(\pr(z),y).
\]	Hence it is sufficient to show that $\dis(\pr(z),x)\le \frac{\dis(x,y)}{2}$. To this aim  set $\tilde l\coloneqq \dis(\pr(z),x)$ and $l\coloneqq \sfd(x,z)$  and observe that by \eqref{eq:variables splitting}
\[
\cos(\tilde l)=\frac{2u(z)+2T-l^2}{4\sqrt{u(z)T}}=\frac{\frac{u(z)}{T}+1-(l/\sqrt{2T})^2}{2\sqrt{\frac{u(z)}{T}}}
\]
Moreover, since $\sqrt{2u}$ is 1-Lipschitz in $B_\delta(p)$ (provided $\delta$ is small enough), we see that $\sqrt{\frac{u(z)}{T}}\in(1-l/\sqrt{2T},1+l/\sqrt{2T})$ and that $l/\sqrt{2T}<1$ if $\delta$ is small enough. Next we observe that for any $a\in(0,1)$ the minimum of the function $\frac{t^2+1-a^2}{2t}$ for $t \in(1-a,1+a)$ is $\sqrt{1-a^2}$, which is achieved at $t=\sqrt{1-a^2}.$ Therefore
\[
\cos(\tilde l)\ge \sqrt{1-(l/\sqrt{2T})^2}
\]
and plugging in the identity $l^2=\sfd(x,y)^2/4=T-T\cos(\dis(x,y))$ (obtained from \eqref{eq:variables splitting}) we reach
\[
\cos(\tilde l)\ge \sqrt{\frac{T+T\cos(\dis(x,y))}{2T}}=\sqrt{\frac{1+\cos(\dis(x,y))}{2}}=\cos\left(\frac{\dis(x,y)}{2}\right),
\]
which shows that $\tilde l \le \frac{\dis(x,y)}{2}$ and concludes the proof of  \ref{it:geod}.
\end{proof}

We are now ready to prove that $\dis$ and $\sfd'$ coincides inside small balls.
\begin{proof}[Proof of Proposition \ref{prop:distances coincide}]
	Fix $p \in S_T$ and let $\delta=\delta(p)>0$, $\lambda=\lambda(T)$ be the ones given by Proposition \ref{prop:geodist}. We fix $r>0$ small and to be chosen and fix $x,y \in B_{r}(p)\cap S_T.$ Since  $\sfd'$ is a geodesic distance  there exists a constant speed geodesic from $x$ to $y$ (for $\sfd'$)  $\{\gamma_t\}_{t \in[0,1]}\subset S_T$. Moreover from \eqref{dd'} we have that $\gamma\subset B_\delta(p)$, provided $r$ is chosen small enough. Note also that, since $\dis$ is comparable with $\sfd$, $\gamma$ is a Lipschitz curve with respect to $\dis$ (which is a metric in $B_\delta(p)\cap S_T$). Take $t\in(0,1)$ such that $|\dot{\gamma}_t|$ exists, then since $|\dot{\gamma}_t|$ coincides with the metric speed computed using $\sfd$, we have
	\[
	\sfd'(x,y)=|\dot{\gamma}_t|=\lim_{h \to 0} \frac{\sfd(\gamma_{t+h},\gamma_t)}{h}=\lim_{h \to 0}\frac {\sqrt{4T}}{h}\sqrt{1-\cos(\dis(\gamma_{t+h},\gamma_t))}=\sqrt{2T}\lim_{h \to 0}\frac{\dis(\gamma_{t+h},\gamma_t)}{h}.
	\]
	In particular the length of the curve $\gamma$ computed with the metric $\dis$ is $\sqrt{2T}^{-1}\sfd'(x,y)$ and since the length is always greater or equal than the distance between the two endpoints we just proved that
	\[
	\dis(x,y)\le \frac{\sfd'(x,y)}{\sqrt{2T}}.
	\]
	On the other hand, from \ref{it:geod} in Proposition \ref{prop:geodist}, if we choose $r<\lambda \delta$, there exists a curve $\gamma: [0,1]\to S_T\cap B_{\delta}$ that is a geodesic from $x$ to $y$ with respect to $\dis$. As above we have that $\gamma$ is a Lipschitz curve with respect to $\sfd'$ and that, for any $t\in(0,1)$ such that the metric speed $|\dot{\gamma}_t|$ (computed with $\dis$) exists, we have
	\[
	|\dot{\gamma}_t|=\frac{1}{\sqrt{2T}} \lim_{h \to 0} \frac{\sfd(\gamma_{t+h},\gamma_t)}{h}.
	\]
	In particular the length of the curve $\gamma$ with respect to the metric $\sfd'$ is $\sqrt{2T}\dis(x,y)$, which shows that
	\[
	\sfd'(x,y)\le \sqrt{2T}\dis(x,y).
	\]
\end{proof}

\subsection{Building the cone and conclusion}\label{sec:conclusion}

Let us define the $(Z,\sfd_Z,\mea_Z)$ as $\left (X',\frac{\sfd'}{\sqrt{2T}},\mea'\right )$.
We then define the metric measure space $(Y,\sfd_Y,\mea_Y)$ as the $N$-euclidean cone over the metric measure  space $(Z,\sfd_Z,\mea_Z)$.

For every $0<a<b<\infty$ we set $A^Y_{a,b}=\{ y \in Y \ : \ \sfd_{Y}(y,O_Y)\in (\sqrt{2a},\sqrt{2b})\}=\{(r,z) \in Y: r \in (\sqrt{2a},\sqrt{2b})\}\subset Y. $ 

We then define the map $T : A^Y_{t_0,T_0} \to A_{t_0,T_0} $ as
\begin{equation}\label{eq:defT}
	T((r,z))\coloneqq F_{\frac12 \log \frac{2T}{r^2}}(z),
\end{equation}
which is well defined thanks to \eqref{ualongpoint}, and the map $S : A_{t_0,T_0} \to A^Y_{t_0,T_0} $ defined as
\begin{equation}\label{eq:defS}
	S(x)\coloneqq (\sqrt{2u(x)},\Pr(x)).
\end{equation}
It is immediate from the definition that $S(A_{a,b})=A^Y_{a,b}$ and $T(A^Y_{a,b})=A_{a,b}$ for every $t_0\le a<b\le T_0$. 

Moreover it is clear from the definition of $\Pr$, \eqref{ualongpoint} and the fact that $F_{-t}=F_t^{-1}$, that $S$ and $T$ are one the inverse of the other. 

The following result follows from the definitions, Proposition \ref{prop:variables splitting}, Proposition \ref{prop:distances coincide} and Proposition \ref{disintlemma}.
\begin{prop}\label{prop:ismoetry}
	The maps $S: A_{t_0,T_0}\to A^Y_{t_0,T_0} $ and  $T: A^Y_{t_0,T_0}\to A_{t_0,T_0} $  are measure preserving local isometries.
\end{prop}

We can now give the the last details which complete the proof of the main theorem.

\begin{proof}[Proof of Theorem \ref{thm:functional cone}]
We already know from above that that the maps $S$ and $T$ are measure preserving local isometries from  $A_{t_0,T_0}$ to $A^Y_{t_0,T_0}$ and viceversa.

Pick now any $t_0',T_0'\in (\bu_0,\infty)$ such that $t_0'<t_0<T<T_0<T_0'$, then we can repeat all the arguments in the previous sections with $t_0',T_0'$ in place of $t_0,T_0$ (but using the same $T$ to define $X'$ as in subsection \ref{sec:levelset}) to obtain a map $S':A_{t_0',T_0'}\to A^Y_{t_0',T_0'}$ that is a local isometry, with an inverse $T'$, which is also a local isometry. The key observation is that $S'$ agrees with $S$ in $A_{t_0,T_0}$. Indeed from \eqref{eq:defS} we deduce that $S'$ on $A_{t_0,T_0}$ depends only on the value of the function $u'$ and  the map $\Pr'$ on  $A_{t_0,T_0}$. From the construction is clear that $u'$ agrees with $u$ on $A_{t_0,T_0}$, since both agree with $\bu$ on this set. Therefore we need to show that the two projection maps $\Pr, \Pr' : A_{t_0,T_0}\to S_T$ agree. Suppose they do not, i.e. there exists $x \in A_{t_0,T_0}$ such that $\Pr(x)\neq \Pr'(x)$. Recall that $\Pr(x)=F_{\frac 12 \log \frac{u(x)}{T}}(x)$ and that the curve $\gamma^1_t=F_t(x)$ for $t \in [0,\frac 12 \log \frac{u(x)}{T}]$ is (up to a reparametrization) a minimizing geodesic  joining $x$ to $\Pr(x)$ and with values in $A_{t_0,T_0}$, as shown in Proposition \ref{prop:basics}. With the same argument we deduce the existence of a geodesic $\gamma^2$ joining $x$ and $\Pr'(x)$ with values in $A_{t_0,T_0}.$ Moreover from \eqref{distancealongpoint} we have that $\sfd(x,\Pr(x))=\sfd(x,\Pr'(x))=\sqrt 2|\sqrt T-\sqrt{u(x)}|$, in particular $\gamma^1,\gamma^2$ are geodesics with same length. Since $S$ is a local isometry we have that the curves $S(\gamma_t^i)$ are both geodesics in $Y$ with the same length. In particular $\sfd_Y(S(x),S(\Pr(x)))=\sfd_Y(S(x),S(\Pr'(x)))$ which using the expression for $S$ gives
$$\sqrt 2|\sqrt T-\sqrt{u(x)}|=\sfd_Y((\sqrt{2u(x)},\Pr(x)),(\sqrt{2T},\Pr(x)))=\sfd_Y((\sqrt{2u(x)},\Pr(x)),(\sqrt{2T},{\Pr}'(x))).$$
However, recalling that $\Pr(x)\neq \Pr'(x)$ and from the definition of $\sfd_Y$  we easily deduce that the rightmost term in the above identity is strictly bigger than $\sqrt 2|\sqrt T-\sqrt{u(x)}|$, which is a contradiction.

We can now send $t_0\to \bu_0$ and $T_0\to + \infty$ and obtain a map ${\bf S}: \{\bu>\bu_0\} \to Y\setminus \overline{B_{\sqrt{2\bu_0}}}(O_Y)$ which is a surjective and measure preserving local isometry. Moreover extending analogously the maps $T: A^Y_{t_0,T_0} \to A_{t_0,T_0}$, which are the inverses of the maps $S$, we obtain a map ${\bf T}: Y\setminus \overline{B_{\sqrt{2\bu_0}}}(O_Y)\to  U$, which is the inverse of ${\bf S}$ and a local isometry as well. 

Observe now that, since $\bf S$ and ${\bf T}$ are a local isometries, they send geodesics to geodesics. This easily implies that 
\begin{equation}\label{eq:dequal}
\sfd(x,\partial \{\bu>\bu_0\})=\sfd_Y({\bf S}(x),B_{\sqrt{2\bu_0}}(O_Y))=\sqrt{2u(x)}-\sqrt{2\bu_0},
\end{equation}
from which \eqref{eq:explicit} follows.

We are now in position to apply Proposition \ref{prop:blowdown} to obtain that $Y$ is an ${\rm RCD}(0,N)$ space, which is the unique tangent cone at infinity to $\X$. Moreover from the fact that $Y$ is an ${\rm RCD}(0,N)$  and from \eqref{eq:ketterer} it follows that $(Z,\sfd_Z,\mea_Z)$ is an ${\rm RCD}(N-2,N-1)$ space satisfying $\diam(Z)\le \pi.$

Suppose now that $\diam(Z)=\pi$, then again from Proposition \ref{prop:blowdown} we obtain that $\X$ is isomorphic to $Y.$

The fact that $\X$ has Euclidean volume growth was already proved in Corollary \ref{cor:euclvol}.

It remains to prove the first part of ii). Let $r=\sqrt{2\bu_0}$ and $r_Z$ as in the statement. It is enough to show that for every couple of points $y_1,y_2 \in Y$ such that $\sfd_Y(y_i,O_Y)>r_Z$, $i=1,2$,  all the geodesics connecting them are contained in $\{\sfd(.,O_Y)>r\}.$ Moreover we can clearly restrict ourselves to consider points $y_1,y_2$ of the form $y_i=(t,z_i)$, with $z_i \in Z$, $i=1,2$ and $t>r_Z.$  For such points we have that 
\[ \sfd_Y(y_1,y_2)=t\sqrt{2-2\cos(\sfd(z_1,z_2))}\le t\sqrt{2-2\cos(\diam(Z))}  \]
Let now $\gamma$ be a geodesic between $y_1$ and $y_2$, then by the triangle inequality 
\[
\sfd(\gamma_t,O_Y)\ge t-\frac{\sfd_Y(y_1,y_2)}{2}>r_Z \left (1-\sqrt{\frac{1-\cos(\diam(Z))}{2}}\right)=r, \quad \forall t \in [0,1], 
\]
where the last identity follows from the definition of $r_Z.$
\end{proof}

\section{Appendix: obstacle problem in $\mathsf{RCD}$}\label{ap:bjorn}
\subsection{Relative capacitary potential for  sets with {\sf Cap}-fat boundary}
This appendix is devoted to the proof of the existence (and uniqueness) of a relative capacitary  potential in ${\rm RCD}$ space and we will mainly focus on boundary regularity. The results contained here are needed only in the proof of Theorem \ref{thm:potential}. 

Let us say that we are not proving anything substantially new, since all the results were essentially already present in \cite{bjorn}. Let us also mention that the results concerning boundary regularity and Wiener criterion for harmonic functions originally appeared in  in \cite{wiener1,wiener2,wiener3}. However the results we needed were spread in  many different chapters of \cite{bjorn} and often the language used there (for example for some type of  Sobolev spaces) does not coincide with the one we use in this note. For this reason we decided to gather here in a self-contained exposition all the results that we required. Finally let us say that working in the context of ${\rm RCD}$ will allow to simplify some of the arguments in \cite{bjorn}.

Along all this appendix $(\X,\sfd,\mea)$ is an $\mathsf{RCD}(K,N)$ m.m.s., $N<+\infty$. Even if we will only apply the  result below for $K=0$, we will consider arbitrary $K$ for generality. We only remark that every time a constant will depend on some radius (or diameter of a set), in the case $K=0$ this dependence can be dropped. This is a consequence of the fact that ${\rm RCD(0,N)}$ spaces are uniformly doubling.

Our main goal is to prove the following (see below for the definition of relative Capacity).

\begin{theorem}\label{thm:local potential}
	Let $E\subset \X$ be an open set and $B$ be a ball such that $E\subset\subset B.$
	Suppose also that $E$ has $\Cap$-fat boundary.  Then there exists $u \in \W_{0}(B)\cap C(B)$, superharmonic in $B$ and harmonic in $B\setminus \bar E$ with $0\le u \le 1$, $u=1$ in $\overline E$ and
	\[
	\Cap(E,B)=\int_B |\nabla u|^2\,\d \mea.
	\]
	Moreover we have the following continuity estimate: for every $x \in \partial E$ it holds
	\[
	1-u(y)\le C_x\sfd(y,x)^{\alpha_x}, \quad \forall y\in B_{r_x/2}(x)\cap B, 
	\]
	for some positive constants $C_x=C_x(r_x,c_x,K,N,\delta)$  $\alpha_x=\alpha(r_x,c_x,K,N,\delta)>0$, where $r_x,c_x$ are the $\Cap$-fatness parameters of $x$ and $\delta>0$ is such that $\sfd(E,B^c)\ge \delta.$
	
	\noindent Finally $u$ satisfies the following comparison principle: for every $v \in \W(B)$ superharmonic and such that $v\ge \nchi_{E}$ $\mea$-a.e. in $B,$ it holds that
	\[ u \le v, \quad \text{$\mea$-a.e. in $B$.} \]
\end{theorem}

\begin{definition}[Variational 2-Capacity]\label{def:capacity}
	Let $E\subset X$ and $\Omega$  open containing $E$. We define
	\[
	\Cap(E,\Omega)=\inf\, \left \{ \int_\Omega |\nabla u|^2\, \d \mea \ : \ u \in \W_0(\Omega) \text{ and $u\ge 1$ $\mea$-a.e. in a neighbourhood of $E$}\right \}
	\]
\end{definition}

\begin{definition}[$\Cap$-fat boundary points]\label{def:capfat}
	We say that an open set $E$ is  $\Cap$-fat at a point $x \in \partial E$  if there exists $r,c>0$ such that
	\[
	\frac{\Cap(B_{s}(x)\cap E,B_{2s}(x))}{\Cap(B_{s}(x),B_{2s}(x))}\ge c \, , \quad \forall \, s \in(0,r).
	\] 
	Moreover we say that $E$ has (uniformly) $\Cap$-fat boundary if it is $\Cap$-fat at every point $x \in \partial E$ (with global parameters $c,r>0$).
\end{definition}

A geometric condition that is enough to ensure $\Cap$-fatness of the boundary is the following interior corkscrew condition. This follows essentially from the doubling property of the measure and the Poincaré inequality (see for example \cite[Prop. 6.16]{bjorn}).
\begin{definition}[Corkscrew-condition]\label{def:corkscrew}
	Let $\lambda \in (0,1)$ and $r>0$. We say that $E $  satisfies the (interior) $(\lambda,r)$-corkscrew condition at $x \in \partial E$ if for every $s\in (0,r)$ there exists an  ball of radius $\lambda s$ contained in $B_s(x)\cap  E$.
\end{definition}
It is easily verified that any ball of radius $>\delta$ satisfies the (interior) $(1/4,\delta)$-corkscrew condition.  Moreover arbitrary unions of sets satisfying the (interior) $(\lambda,r)$-corkscrew condition still satisfies the (interior) $(\lambda,r)$-corkscrew condition. In particular union of balls with radius uniformly bounded below satisfies the interior corkscrew condition. It follows that any $\eps$-enlargements of a set, i.e.\ a set of the form $S^{\eps}=\{x \ : \ \sfd(x,S)<\eps\}$, with $\eps>0$ and $S$ an arbitrary set, satisfies the interior corkscrew condition.

\subsection{Preliminaries}

We will need the following variational characterization of sub(super)harmonic functions (see \cite[Theorem 2.5]{grigoni} and also \cite{GM}, \cite{Gappl}).
\begin{prop}\label{prop:variational}
	Let $\Omega \subset \X$ be open. A function $u\in \W(\Omega)$ is superharmonic (resp. subharmonic) in $\Omega$ if and only if
	\[
		\int_{\Omega} |\nabla u|^2 \d \mea \le \int_{\Omega} |\nabla (u+\phi)|^2 \d \mea,
	\]
	for every $\phi \in \LIP_\c(\Omega)$ with $\phi \ge 0$ (resp. $\phi \le 0$) or equivalently for every $\phi \in \W_0(\Omega)$  with $\phi \ge 0$ (resp. $\phi \le 0$) $\mea$-a.e.. 
\end{prop}

Since, as shown in \cite{rajala}, ${\rm RCD}(K,N)$ spaces support a (1,1) Poincarè inequality and they are also (uniformly) locally doubling, a class of Sobolev embeddings can be shown to hold (see for example \cite{sobolevmeets} and also \cite[Chap. 4-5 ]{bjorn}). Therefore a Moser iteration can be performed to obtain the following Harnack inequalities (see for example  \cite{bobo} for the case $K=0$).
\begin{prop}\label{prop:harnack}
	For every $R_0>0$ there exists two positive constants $C_i=C_i(R_0,K^-,N)$, $i=1,2$, such that the following hold for any $R<R_0$
	\begin{enumerate}
		\item if $u$ is  subharmonic function in a ball $B_{2R}(x)$, then
		\[\esssup_{B_{R}(x)} u \le C_2  \fint_{B_{2R}(x)} |u| \d \mea ,\]
		\item if $u$ is a nonnegative superharmonic function in a ball $B_{2R}(x)$, then
		\[\essinf_{B_{R}(x)} u \ge C_1 \fint_{B_{2R}(x)} u \d \mea .\] 
	\end{enumerate}
\end{prop}
The above Harnack inequalities  imply that harmonic functions have a locally H\"older continuous representative (which is actually locally Lipschitz by \cite{jiang}) and that superharmonic functions have a lower semicontinuous representative (see for example \cite[Theorem 8.22]{bjorn}). From now on we will always tacitly consider  these  special representatives.

Lastly, we will need the following technical lemma, whose simple proof is omitted.
\begin{lemma}\label{lem:noioso}
	Let $u \in \W_{0}(\Omega)$, then $u^+\in \W_{0}(\Omega)$. 
	
	Let $v \in \W(\Omega)$, $u \in \W_{0}(\Omega)$ be such that $0\le v\le u$, then $v\in \W_{0}(\Omega)$
\end{lemma}

\subsection{The obstacle problem}

Given a ball $B \subset \X$ and a (Borel) set $E\subset \subset B$  we consider the following minimization problem
\begin{equation}\label{eq:obstacle}
	Obs(E,B)\coloneqq\inf_{u \in \F_{E,B}} \int_B |\nabla u|^2\, \d \mea,  \tag{O}
\end{equation} 
where $\F_{E,B}=\{u \in \W_0(B) \ : u \ge \nchi_E \,\, \mea\text{-a.e. in } B\}.$  

It is clear that if $E$ is open, then
\[
Obs(E,B)=\Cap(E,B).
\]

The proof of the following result is a straightforward application of the direct method of the calculus of variations, recalling that the embedding $\W_0(B)\hookrightarrow L^2(\X)$ is compact (see for example \cite[Theorem 6.3]{GMSconv}) and  from the lower semi continuity and (strict) convexity of the Cheeger energy.
\begin{prop}\label{prop:unique obs}
There exists a unique minimizer to \eqref{eq:obstacle}. Moreover this minimizer  is  superharmonic in $E.$
\end{prop}
We now show the two main properties of the minimizers of \eqref{eq:obstacle}: the first is that $u$ is harmonic far from the obstacle $E$ and  the second says that $u$ is essentially the smallest superharmonic function which stays above $\nchi_E.$
\begin{prop}\label{prop:solprop}
	Let $u$  be the minimum of \eqref{eq:obstacle} for some $E\subset \subset B.$ Then $u=1$  $\mea$-a.e.\ in $E$ and the following hold:
	\begin{enumerate}
		\item $u $ is harmonic in $B\setminus \bar E$,
		\item comparison principle: for every $v \in \W(B)$ superharmonic and such that $v\ge \nchi_{E}$, $\mea$-a.e., it holds that
		\[ u \le v, \quad \text{$\mea$-a.e. in $B$.} \]
	\end{enumerate}
\end{prop}
\begin{proof}
	We start  by showing that $u\le 1$ $\mea$-a.e.\ in $B$. Indeed $u\wedge 1 \in \F_{E,B}$ and $\int_B |\nabla (u\wedge 1)|^2\le \int_B |\nabla u|^2\d \mea$, from which the claim follows. Since $u \ge \nchi_E,$ $\mea$-a.e. it also follows that $u=1$ $\mea$-a.e.\ in $E$.
	
	We pass to the harmonicity. Fix $\phi \in \LIP_c(B\setminus \bar U)$. Clearly $(u+\phi)^+\in \F_{E,B}$, therefore
	\begin{align*}
		\int_{B\setminus \bar E}| \nabla (u+\phi)|^2\, \d \mea &\ge \int_{B\setminus \bar E}| \nabla (u+\phi)^+|^2\, \d \mea \\
		&= \int_{B}| \nabla (u+\phi)^+|^2\, \d \mea-\int_{\bar E} |\nabla u|^2\, \d \mea\ge \int_{B\setminus \bar E}|\nabla u|^2 \d \mea,
	\end{align*}
	where in the equality step we have used that $\phi=0$ in $\bar E$ and the locality of the gradient. This and Proposition \ref{prop:variational} prove the claimed harmonicity. 
	
	It remains to prove the comparison principle.
	We start claiming  that $(u-v)^+\in \W_0(B)$ and  $\min(u,v)\in \F_{E,B} $. Indeed we have that $0\le (u-v)^+\le u,$ $\mea$-a.e.\ in $B$ and $\nchi_E\le \min(u,v)\le u$ $\mea$-a.e. in $B$, therefore the claim follows applying  Lemma \ref{lem:noioso}. 
	
	Observe that $\max(u,v)=v+(u-v)^+$, hence from the superharmonicity of $v$, Proposition \ref{prop:unique obs},  and the locality of the gradient  we have
	\begin{equation*}
		\int_{\{u>v\}}|\nabla u|^2\, \d \mea \ge 	\int_{\{u>v\}}|\nabla v|^2\, \d \mea .
	\end{equation*}
	Therefore from the locality of the gradient it follows that
	\[
	\int_B |\nabla \min(u,v)|^2\, \d \mea \le  \int_B |\nabla u|^2\,\d \mea,
	\]
	that combined with $\min(u,v)\in \F_{E,B} $ and the uniqueness of the solution to \eqref{eq:obstacle} implies that $\min(u,v)=u$ $\mea$-a.e. in $B$.
\end{proof}

We conclude this part with the following technical result.
\begin{lemma}\label{lem:lift}
Let $u$  be the minimum of \eqref{eq:obstacle} for some $E\subset \subset B.$ Then for every $m\in (0,1]$, the function $\frac{u}{m}\wedge 1$ is the minimum of \eqref{eq:obstacle} in $B$ with $E=\{u>m\}$.
\end{lemma}
\begin{proof}
	Set $u_m=\frac{u}{m}\wedge 1$ and fix $v \in \F_{{\{u> m\}},B}$. Observe that $u_m\ge \nchi_{\{u> m\}}$ and that $u_m \in \W_{0}(B)$ by Lemma \ref{lem:noioso}, hence $u_m \in \F_{\{u> m\},B}.$  Define the function $\bar u\coloneqq u +m(v-u_m)$ and observe that $\bar u\in \W_{0}(E)$.  Moreover $\bar u \ge u \ge 0$ $\mea$-a.e. in $B$  and, since from Proposition \ref{prop:solprop} $u=1$ $\mea$-a.e. in $E$, we also have that $\bar u=1$ $\mea$-a.e. in $E$. Therefore $\bar u \in \F_{E,B}.$ This and the fact that $\bar u=mv$ $\mea$-a.e. in $\{u\le m\}$ and $\bar u=u $ $\mea$-a.e. in $\{u> m\}$ gives
	\[
		\int_B|\nabla v|^2\, \d \mea\ge \int_{\{u\le m\}}|\nabla v|^2\, \d \mea \ge \frac{1}{m^2}	\int_{\{u\le m\}}|\nabla u|^2\, \d \mea=	\int_B|\nabla u_m|^2\, \d \mea.	
		\]
Since $v \in \F_{{\{u>m\}},B}$ was arbitrary we conclude.
\end{proof}

\subsection{Proof of Theorem \ref{thm:local potential}}

\begin{prop}\label{prop:below estimate}
For every $r_0<4 \diam(\X)$ there exists $C=C(r_0,K,N)>0$ such that the following holds.  Let $E\subset B_r(x)$ be open, $r<r_0$ let  $2B=B_{2r}(x)$ and let $u$ be the solution to \eqref{eq:obstacle} for $E$ in $2B$. Then
	\[
	u \ge C \frac{\Cap(E,2B)}{\Cap(B,2B)}, \quad \text{$\mea$-a.e. in $B_r(x)$.}
	\]
\end{prop}
\begin{proof}
	Set $B'=B_{\frac 32 r}(x)$ and observe that, since $r_0<4 \diam(\X)$, $\partial B'\neq\emptyset.$ Define $m \coloneqq \max_{\partial B'} u,$ which exists because $u$ is continuous in $\partial B'.$
	We claim that $m>0$. Indeed if  $m=0$, from the maximum principle (Proposition \ref{prop:maxprinc}) we would have that $u=0$  in the ball $2B$ (recall that balls are connected), and thus $u=0$ in $E$, which contradicts the fact that $u=1$ $\mea$-a.e.\ in $E$ with $E$ open.   We claim that
	\begin{equation}\label{eq:tricky}
		u\le m, \quad \text{ in } 2B\setminus B'.
	\end{equation}
	To see this let $m'>m$ and observe that $(u-m')^+\le u^+$ hence by Lemma \ref{lem:noioso} $(u-m')^+\in \W_{0}(2B)$. Moreover, from the continuity of $u$ and the definition of $m$ we have that $(u-m')^+=0$ in a neighbourhood of $\partial B'$. These two observations together imply that $(u-m')^+\in \W_0(2B\setminus\bar B').$ Observe that $\min(u,m')=u-(u-m')^+$ hence from harmonicity of $u$ we deduce that
	\[
	 \int_{2B\setminus \bar B'} |\nabla u|^2\, \d\ \mea \le \int_{2B\setminus \bar B'}|\nabla \min(u,m') |^2\, \d \mea,
	\]
	which combined with the locality of the gradient gives that $|\nabla u|=0$ $\mea$-a.e.\ in $\{2B\setminus \bar B'\}\cap\{u\ge m'\}$. Therefore again by locality $|\nabla (\max(u,m'))|=0$ $\mea$-a.e.\ in $\{2B\setminus \bar B'\}$ and thus  $u \le m'$ in $2B\setminus \bar B'$. Since $m'>m$ was arbitrary \eqref{eq:tricky}  follows.
	
	 Define the functions $u_1=\frac{u}{m}\wedge 1$,  $u_2=\frac{u-mu_1}{1-m}$ and observe that $u_1,u_2 \in \F_{E,2B}$. In particular for every $t \in (0,1)$ $tu_1+(1-t)u_2\in  \F_{E,2B}$ and
	\[
	\int_{2B}|\nabla u|^2\, \d \mea \le t^2I_1\, \d \mea+(1-t)^2I_2,
	\]
	where	$I_i=\int_{2B}|\nabla u_i|^2\,\d \mea.$ Optimizing in $t$ we obtain that 
	\begin{equation}\label{eq:optimization}
		\frac{1}{\int_{2B}|\nabla u|^2 \, \d \mea }\ge \frac{1}{I_1}+\frac{1}{I_2}.
	\end{equation}
	 Observe now that $u_2=0$  in $\{u\le m\}$ and $u_2=(1-m)^{-1}(u-1)$ in $\{u>m\}$ , therefore $|\nabla u_2|=\nchi_{\{u>m\}}|\nabla u|(1-m)^{-1}$ $\mea$-a.e. in $2B$. In particular $I_2\le (1-m)^{-2}\int_{2B}|\nabla u|^2\, \d \mea,$ that combined with \eqref{eq:optimization} gives 
	 \[
	 \Cap(E,2B)	=\int_{2B}|\nabla u|^2\, \d \mea  \le (2m-m^2)I_1\le 2mI_1.
	 \]
	 This combined with Lemma \ref{lem:lift} gives 
	 \begin{equation}\label{eq:capacity chain}
	 	\Cap(E,2B)\le 2m\, Obs(\{u> m\},2B)\le2m\, Obs(B',2B)=2m\Cap(B',2B),
	 \end{equation}
	 where in the second inequality we have used \eqref{eq:tricky}.
	 
	 From the definition of $m$ there exists a ball $B''=B_{r/2}(y)$ with $y \in \partial B'$ such that $\sup_{B''}u\ge m$. Applying twice the Harnack inequality, recalling that $u$ is harmonic in $B''$ and superharmonic in $B$, and using the doubling property we obtain that 
	 \[
	 m\le \sup_{B''} u\le C(r_0,K,N) \essinf_{B} u.
	 \]
	 The conclusion then follows from \eqref{eq:capacity chain} and recalling that thanks to the doubling condition and the Poincaré inequality we have $\Cap(B',2B)\le c\Cap(B,2B)$, for some constant $c$ depending only on $r_0, K$ and $N$ (see \cite[Prop. 6.16]{bjorn})
\end{proof}

\begin{theorem}\label{thm: boh}
	For every $r_0<4 \diam(\X)$ there exists $C=C(r_0,K,N)>0$ such that the following holds. Let $E\subset B_r(x)$ be open, $r<r_0,$ and set $B_i\coloneqq B_{2^{1-i}r}(x)$ for $i\in \mathbb{N}_0.$ Let $u$ be the capacitary potential for $E$ in $B_0$, then for every $i \ge1$ it holds that
	\[
	1-u\le \exp\left (-C \sum_{j=1}^i \frac{\Cap(E\cap B_j,B_{j-1})}{\Cap(B_j,B_{j-1})}\right ), \quad \text{$\mea$-a.e. in $B_i$.}
	\]
\end{theorem}
\begin{proof}
	Let  $u_i$ be the solution to \eqref{eq:obstacle} for $E\cap B_i$ in $B_{i+1}$ (in particular $u=u_1$) and define $a_i\coloneqq \frac{\Cap(E\cap B_i,B_{i-1})}{\Cap(B_i,B_{i-1})}$ for $i\in \mathbb{N}$.  Proposition  \ref{prop:below estimate} ensures that
	\begin{equation}\label{eq:inductive lowerbound}
		\text{ess}\inf_{B_i} u_i \ge Ca_i\ge 1-e^{-Ca_i}.
	\end{equation}
Define the functions $v_i\in \W_{0}(B_0)$ inductively as $v_1=u_1$ and $v_i=1-e^{Ca_{i-1}}(1-v_{i-1}),$ for $i \ge 2.$ Observe that, since $u_1$ is superharmonic in $B_0$, $v_i$ is superharmonic in $B_0$ for all $i\ge 1$. We claim that
\begin{equation}\label{eq:i comparison}
	v_i \ge 0, \quad \text{ $\mea$-a.e. in $B_{i-1}$}.
\end{equation}
We will actually show the stronger estimate $v_i\ge u_i$ $\mea$-a.e.  in $B_{i-1}$. We proceed  by induction. By definition $v_1=u_1$, now suppose that $v_i\ge u_i$ in $B_{i-1}$. It follows from \eqref{eq:inductive lowerbound} that $v_{i+1}\ge 1-e^{Ca_{i}}(1-u_{i})\ge 0$ $\mea$-a.e.  in $B_{i}.$ Moreover, since $u_1=1$ in $E\cap B_1$, evidently $v_{i+1}=1$ $\mea$-a.e.  in $E\cap B_{i+1}.$ Combining these two observations we obtain that $v_{i+1}\ge \nchi_{E\cap B_{i+1}}$ $\mea$-a.e. in $B_{i}.$ Recalling that $v_i$ is superharmonic in $B$ (and thus also on $B_i$) we can apply the comparison principle of Proposition \ref{prop:solprop} to deduce that $v_{i+1}\ge u_{i+1}$  $\mea$-a.e. in $B_i$. This proves the claim. Therefore from \eqref{eq:i comparison}
\[
1-u=1-v_1=e^{-C(a_1+...+a_{i-1})}(1-v_i)\le  e^{-C(a_1+...+a_{i-1})}, \quad \text{$\mea$-a.e.  in $B_{i-1}$},
\]
that concludes the proof.
\end{proof}

\begin{proof}[Proof of Theorem \ref{thm:local potential}]
	
	 Fix $x \in \partial E$ and let $c,r$ be its $\Cap$-fat parameters. Let $B'\coloneqq B_{r_0}(x)$ with $r_0\coloneqq(\delta\wedge r)/4$ and let $u$ to be the solution to  \eqref{eq:obstacle} for $E$ in $2B'$. Fix $y\in B'\setminus \bar E$ with $\sfd(y,x)<r/2$. There exists $i \in \mathbb{N}_0$ such that $2^{-i-1}r_0<\sfd(x,y)<2^{-i}r_0<r$. Therefore from Theorem \ref{thm: boh}  and the continuity of $u$ in $B'\setminus \bar E$ we have
	\begin{equation}\label{eq:continuity est}
	1-u(y)\le (e^{-i})^{c\cdot C}\le(2^{-i})^{c\cdot C}\le (2r_0^{-1})^{c\cdot C} \sfd(x,y)^{c\cdot C}=(8\delta\wedge r)^{-c\cdot C}\sfd(x,y)^{c\cdot C} .
	\end{equation}
	Let  now $\bar u$ to be the solution of \eqref{eq:obstacle} for $E$ in $B$ (where $B$ is as in the hypotheses). Fix $x \in \partial E$ and let $u$ as in the previous part of the proof. Since $B_{2r_0}(x)\subset B$, from the comparison principle of Proposition \ref{prop:solprop} we have that $\bar u\ge u$ $\mea$-a.e. in $B_{r_0}(x)$ and since both $\bar u$ and $u$ are continuous in $B_{r_0}(x)\setminus \bar E$ we have that \eqref{eq:continuity est} holds  for $\bar u$ and every $y\in B\setminus \bar E$ with $\sfd(y,x)<r/2$. This proves that $\lim_{B_{r_0}(x)\setminus \bar E \ni y \to x}\bar u(y)=1$ for every $x\in \partial E$ (recall that $\bar u\le 1$) and since $\bar u$ is also lower semicontinuous we deduce that $\bar u=1$ in $\bar E.$

		The comparison principle is already contained in Proposition \ref{prop:solprop}.
 \end{proof}

	\end{document}